\documentclass[9pt,reqno]{amsart}  

\usepackage[english]{babel}
\usepackage[p,osf]{cochineal}
\usepackage[scale=.95,type1]{cabin}
\usepackage[zerostyle=c,scaled=.94]{newtxtt}
\usepackage[T1]{fontenc} 
\usepackage[utf8]{inputenc} 
\usepackage{amsthm}
\usepackage{amsmath}
\usepackage{amssymb}
\usepackage{amsfonts}
\usepackage{amsaddr}
\usepackage{xspace}
\usepackage{enumitem} 
\usepackage{csquotes}
\usepackage{xcolor}
\usepackage[colorlinks]{hyperref}
\hypersetup{colorlinks, breaklinks, citecolor=teal,filecolor=black}
\usepackage{graphicx}
\usepackage{mathtools}
\mathtoolsset{centercolon}
\usepackage{bm}
\usepackage{pgfkeys}
\usepackage{pgf,interval}
\intervalconfig{soft open fences}
\DeclareMathOperator{\diag}{diag}
\DeclareMathOperator{\Tr}{Tr}

\DeclareMathOperator{\sgn}{sgn}
\DeclareMathOperator{\supp}{supp}
\DeclareMathOperator{\Cov}{Cov}
\DeclareMathOperator{\Var}{Var}

\DeclareMathOperator{\E}{\mathbf{E}}
\DeclareMathOperator{\Prob}{\mathbf{P}}

\newcommand{\ov}{\overline}

\newcommand{\ii}{\mathrm{i}}

\newcommand{\F}{\mathbf{F}}

\ifdefined\C
\renewcommand{\C}{\mathbf{C}}
\else
\newcommand{\C}{\mathbf{C}}
\fi

\newcommand{\un}{\underline}
\newcommand{\vx}{\bm{x}}

\newcommand{\bu}{\bm{u}}

\newcommand{\vy}{\bm{y}}

\newcommand{\wt}{\widetilde}

\newcommand{\R}{\mathbf{R}}
\newcommand{\N}{\mathbf{N}}
\newcommand{\Z}{\mathbf{Z}}

\newcommand{\cG}{\mathcal{G}}
\newcommand{\cO}{\mathcal{O}}
\newcommand{\co}{{\scriptstyle\mathcal{O}}}

\newcommand{\dif}{\operatorname{d}\!{}}
\newcommand\restr[3]{{%
  \left.\kern-\nulldelimiterspace %
  #1 %
  \vphantom{\big|} %
  \right|_{#2}^{#3} %
  }}
\DeclarePairedDelimiter{\braket}{\langle}{\rangle}%
\DeclarePairedDelimiter{\abs}{\lvert}{\rvert}%
\DeclarePairedDelimiter{\norm}{\lVert}{\rVert}%
\providecommand\given{}
\newcommand\SetSymbol[1][]{\nonscript\:#1\vert\allowbreak\nonscript\:\mathopen{}}
\DeclarePairedDelimiterX{\tuple}[1](){\renewcommand\given{\SetSymbol[\delimsize]}#1}
\DeclarePairedDelimiterX{\set}[1]\{\}{\renewcommand\given{\SetSymbol[\delimsize]}#1}
\DeclarePairedDelimiterXPP{\landauO}[1]{\cO}(){}{#1}
\DeclarePairedDelimiterXPP{\landauo}[1]{\co}(){}{#1}
\DeclarePairedDelimiterXPP{\landauOprec}[1]{\cO_\prec}(){}{#1}
\DeclarePairedDelimiterXPP{\landauOstd}[1]{\cO_\prec^2}(){}{#1}
\DeclarePairedDelimiterXPP{\landauOE}[1]{\cO_\prec^1}(){}{#1}
\DeclarePairedDelimiterXPP{\landauOd}[1]{\cO_\mathrm{m}}(){}{#1}

\usepackage{tikz,pgfkeys,subcaption}
\usetikzlibrary{calc,patterns,decorations.pathmorphing,decorations.markings,arrows}

\usepackage[giveninits=true,url=false,doi=false,isbn=false,eprint=true,datamodel=mrnumber,sorting=nty,maxcitenames=4,maxbibnames=99,backref=false,block=space,backend=biber,bibstyle=phys]{biblatex} 
\AtEveryBibitem{\clearfield{month}}
\AtEveryCitekey{\clearfield{month}} 
\renewbibmacro{in:}{}
\DeclareFieldFormat[article]{title}{\emph{#1}} 
\DeclareFieldFormat{mrnumber}{\ifhyperref{\href{http://www.ams.org/mathscinet-getitem?mr=#1}{\nolinkurl{MR#1}}}{\nolinkurl{#1}}}
\DeclareFieldFormat{pmid}{\ifhyperref{\href{https://www.ncbi.nlm.nih.gov/pubmed/#1}{\nolinkurl{PMID#1}}}{\nolinkurl{#1}}}
\DeclareFieldFormat{eprint}{\ifhyperref{\href{https://arxiv.org/abs/#1}{\nolinkurl{arXiv:#1}}}{\nolinkurl{#1}}}
\renewbibmacro*{doi+eprint+url}{%
  \iftoggle{bbx:doi}{\printfield{doi}}{}
  \newunit\newblock%
  \printfield{mrnumber}%
  \newunit\newblock%
  \printfield{pmid}%
  \newunit\newblock%
  \printfield{eprint}%
  \iftoggle{bbx:url}{\usebibmacro{url+urldate}}{}}
\usepackage{nameref,zref-xr}
\zxrsetup{toltxlabel=true, tozreflabel=false}
\zexternaldocument*[eth-]{eth_v4}
\bibliography{bibclt} 

\newtheorem{theorem}{Theorem}[section]
\newtheorem{assumption}[theorem]{Assumption}
\newtheorem{lemma}[theorem]{Lemma}
\newtheorem{proposition}[theorem]{Proposition}
\newtheorem{definition}[theorem]{Definition}

\newtheorem{notation}[theorem]{Notation} 
\newtheorem{remark}[theorem]{Remark}

\newtheorem{corollary}[theorem]{Corollary}

\date{\today}
\author{Giorgio Cipolloni \and L\'aszl\'o Erd\H{o}s}
\address{IST Austria, Am Campus 1, 3400 Klosterneuburg, Austria}
\author{Dominik Schr\"oder\(^{\ast}\)}
\address{Institute for Theoretical Studies, ETH Zurich, Clausiusstr.\ 47, 8092 Zurich, Switzerland}
\email{giorgio.cipolloni@ist.ac.at} 
\email{lerdos@ist.ac.at}
\email{dschroeder@ethz.ch}
\thanks{\(^\ast\)Supported by Dr.\ Max R\"ossler, the Walter Haefner Foundation and the ETH Z\"urich Foundation}
\subjclass[2010]{60B20, 15B52} 
\keywords{Quantum Unique Ergodicity, Multi-resolvent local law, Multiscale Gaussian fluctuation, Eigenfunction Thermalization Hypothesis.}
\title{Functional Central Limit Theorems for Wigner Matrices}
\date{\today}
\begin{document}  
\thispagestyle{empty}


\begin{abstract}
  We consider the fluctuations of  regular functions \(f\) of a Wigner matrix \(W\) viewed as an entire matrix \(f(W)\).   Going 
  beyond the well studied tracial mode, \(\Tr f(W)\), which is equivalent to the customary  linear statistics of eigenvalues,
  we show that \(\Tr f(W)A\) is asymptotically normal for any non-trivial bounded deterministic matrix \(A\). We identify three different and asymptotically independent modes of this fluctuation, corresponding to the tracial part, the traceless diagonal part and the off-diagonal part of \(f(W)\) in the entire mesoscopic regime, where we find that the off-diagonal modes fluctuate on a much smaller scale than the tracial mode. As a main motivation to study CLT in such generality on small mesoscopic scales, we determine the fluctuations in the Eigenstate Thermalization Hypothesis~\cite{9905246}, i.e.\ prove that
  the eigenfunction overlaps with any deterministic matrix are asymptotically Gaussian after a small spectral averaging. Finally, in  the macroscopic regime our result also generalises~\cite{MR3155024} to complex \(W\) and to all crossover ensembles in between. The main technical inputs are the recent multi-resolvent local laws with traceless deterministic matrices from the companion paper~\cite{2012.13215}.
\end{abstract}

\maketitle

\section{Introduction}
The eigenvalues \(\{ \lambda_i\}_{i=1}^N\)  of large  \(N\times N\)  Hermitian random matrices \(W\) form a strongly correlated system of random points on the real line. One manifestation of this feature is that their linear statistics, \(\Tr f(W)=\sum_{i=1}^N f(\lambda_i)\)  with a regular test function
\(f\colon\R\to\R\) has a variance of order one, in fact   
it satisfies a central limit theorem (CLT) but without the customary \(N^{-1/2}\) scaling factor. Note that Gaussian fluctuations normally emerge with the \(N^{-1/2}\) factor as a cumulative  effect of  \(N\) independent
or weakly dependent random variables. Thus it is quite remarkable that CLT  holds for the
strongly correlated eigenvalues and the anomalous scaling alone offsets all effects of these
correlations, rendering the fluctuations
of \(\sum_i f(\lambda_i)\) still Gaussian. 

What about the fluctuations of \(f(W)\) viewed as a matrix and not just considering its trace?
In this paper we show that \(f(W)\) tested against any bounded deterministic  matrix \(A\), \(\|A\|\le 1\),
is still asymptotically normal, provided that \(\Tr AA^\ast\gtrsim N^\epsilon\). Our result holds in  the macroscopic and in the
entire mesoscopic regime, including spectral edges.
More precisely,
we consider
the centred \emph{functional linear statistics} 
\begin{equation}\label{trfha}
  L_N(f, A):= \Tr  \big[ f(W)A\big] - \E \Tr  \big[ f(W)A\big]= \sum_{i=1}^N f(\lambda_i) \langle \bu_i, A\bu_i\rangle -\E \Big[ \ldots \Big],
\end{equation} 
where \(\bu_i\) is the  normalized eigenvector of \(W\) corresponding to \(\lambda_i\). The statistics is called \emph{macroscopic}  
if \(f\) is \(N\)-independent, and \emph{mesoscopic on scale \(N^{-a}\)} with some exponent \(a\in (0,1)\) if \(f\) is 
of the form \(f(x) = g(N^a (x-E))\) with some \(N\)-independent compactly supported function \(g\), i.e.\ if \(f\) lives on a scale \(N^{-a}\) around a fixed energy \(E\in [-2,2]\) in the spectrum. 

One prominent motivation to study functional CLT on small mesoscopic scales is to understand the 
fluctuation in 
the  Eigenstate Thermalization Hypothesis in physics~\cite{9905246}, also known 
as the strong Quantum Unique Ergodicity (QUE) in mathematics~\cite{MR1266075}, see~\cite{2012.13215}
for further references. QUE for Wigner matrices asserts that 
a law of large numbers holds for the eigenvector overlaps with deterministic matrices $A$, i.e.\ that
\(\langle\bu_i, A\bu_i\rangle \) converges to the normalized trace of $A$ as $N\to\infty$.
In our companion paper~\cite{2012.13215} we established 
the optimal convergence rate of order \(N^{-1/2+\epsilon}\), for any \(\epsilon>0\),
with a very high  probability.  In  Theorem~\ref{theo:sharpcom} of the current paper we prove 
that the overlaps \(\langle\bu_i, A\bu_i\rangle\) 
are asymptotically Gaussian after a small spectral averaging in the index \(i\), which
corresponds to the mesoscopic functional CLT for~\eqref{trfha} when \(f\) is a characteristic function 
supported on a small spectral interval containing about  \(N^\epsilon\) eigenvalues
for any arbitrary small \(\epsilon>0\). 
We remark that the Gaussian fluctuation of \(\langle\bu_i, A\bu_i\rangle\)  is expected to hold
for each \(i\) individually, but this result has only been proven for finite
rank \(A\) using the  Dyson Brownian motion for eigenvectors, see~\cite{MR3606475, MR4156609, 2005.08425}.

For \(A=I\), the quantity   \(L_N(f, I)\) is the standard linear statistics of the eigenvalues that have been
studied extensively both in the macroscopic regime by many authors~\cite{MR1487983,MR1411619, MR1647832, MR2189081, MR2489497, MR2829615, MR3116567, MR3568772,1303.1045,MR1899457}  and in the entire mesoscopic regime \(a\in (0,1)\) by He and Knowles~\cite{MR3678478, MR3959983, MR4095015}, see also~\cite{MR2489497, MR3852256, MR3459158, MR3865662, MR3914908, MR4009708, MR4168391, 2001.08725, 2001.07661,1909.12821}  for related models on mesoscopic scales and~\cite{MR1678012, MR1689027,MR3302637,MR3311888}  for previous works on non-optimal intermediate scales.
It is therefore well known that \(L_N(f, I)\) is asymptotically normal, i.e.\ without  a further
\(N^{-a/2}\) normalization 
it satisfies a central limit theorem with a variance given by essentially the \(H^{1/2}\)-norm of \(f\),
see~\eqref{eq:V1}.   Note that the entire analysis of the special case \(A=I\)  is \emph{tracial}, it relies only on 
the eigenvalues of \(W\) and is insensitive to its eigenvectors.

For the  case of general observables,
we  decompose \(A\) as
\begin{equation}\label{decomp}
  A= \langle A\rangle I  + \mathring{A}_{\rm d}  + A_{{\rm od}}, \qquad  \langle A\rangle:=\frac{1}{N}\Tr A,
\end{equation}
where \(\mathring{A}_{\rm d}= A_{\rm d} - \langle A_{\rm d} \rangle\) is the
traceless component of the diagonal part \(A_{\rm d}\) of \(A\) and \( A_{{\rm od}}:= A-A_{\rm d}\)
is the off-diagonal part of \(A\).   Following this decomposition, 
\(L_N(f, A)\) has three different, mutually asymptotically independent Gaussian fluctuation modes,
their expectations and variances are given in Theorem~\ref{theo:CLT}. On the macroscopic scale and for
real symmetric Wigner matrices  this result was essentially obtained by Lytova in~\cite{MR3155024}. 
In Theorem~\ref{theo:CLT} we extend~\cite{MR3155024} to  complex 
Hermitian Wigner matrices including all crossover ensembles, i.e.\ following the dependence on the real
parameter  \(\sigma:= N \E w_{12}^2\) in its entire range \(\sigma\in [-1,1]\)
under the standard normalization 
\(\E |w_{12}|^2=\frac{1}{N}\), \(\E w_{12}=0\) for the off-diagonal  matrix elements of \(W\).

Our main contribution, however,  is to   establish a similar decomposition  of  fluctuations for the entire mesoscopic
regime, \(a\in(0,1)\), since our  Theorem~\ref{theo:CLT} also allows for mesoscopic test functions. 
The corresponding limiting variances are computed in 
Propositions~\ref{prop bulk}--\ref{prop edge}.  For mesoscopic test functions \(f\) the current paper contains the 
first results on the limiting distribution of  \(\Tr[f(W)A]\), with \(A\ne I\).
It turns out that the  two traceless modes fluctuate on 
a scale of order \(N^{-a/2}\) in the bulk and \(N^{-3a/4}\) at the edge
in contrast to the  \(\landauO{1}\)  fluctuation scale of \(L_N(f, I)\). 
Hence we not only  need to  explore the genuine off-diagonal fluctuations 
involving eigenvectors,  but we also need to
work at a much higher accuracy to detect the relevant fluctuations 
that are subleading compared with the previously explored regimes. This is a major new complication 
not present in the \(a=0\) macroscopic scale in~\cite{MR3155024}.
Furthermore,  we  also show that 
mesoscopic linear statistics living on different scales are asymptotically independent (Theorem~\ref{theo:inddiffeta}).

We explain the phenomenon of  different fluctuation scales on the standard example of the 
resolvents, \(G=G(z) = (W-z)^{-1}\) with spectral parameter $z\in \C\setminus \R$,
that can be viewed as a function \(f\) of \(W\) living on scale \(\eta:= \Im z>0\)
around the point \(E:= \Re z\).  To understand \(\langle GA \rangle\) for a deterministic matrix \(A\), we
decompose \(A\) into its tracial and traceless parts as \(A= : \langle A \rangle + \mathring{A}\) and write
\begin{equation}\label{GA}
  \langle GA\rangle = m \langle A \rangle + \langle A \rangle \langle G-m \rangle + \langle G \mathring{A} \rangle,
\end{equation} 
where \(m=m(z)\) is the Stieltjes transform of the semicircle law.
The first term is deterministic, the second one is asymptotically Gaussian on scale \( \langle (G-m)(z) \rangle\sim (N\eta)^{-1}\)
by~\cite{MR4095015}. We prove that the last term in~\eqref{GA} is also Gaussian, independent of the first one, 
and it  has size \(\langle G \mathring{A} \rangle\sim \langle  \mathring{A}
\mathring{A}^*\rangle^{1/2}/(N\eta^{1/2})\), provided that \(\braket{\mathring A \mathring A^\ast}\gg (N\eta)^{-1}\). In fact, it can be further split into a diagonal and off-diagonal part following~\eqref{decomp}. Thus the fluctuation of the tracial part is much bigger than that of the traceless part in the small \(\eta\) regime, however, the latter determines the fluctuation of \(\langle GA\rangle\)
for traceless observables \(\langle A \rangle=0\).  

%

We now mention a few related works on general Gaussian fluctuations in Wigner matrices. In contrast
to the extensively studied  linear
eigenvalue statistics, this question received much less attention
in the random matrix community, although a Wigner matrix contains  many other physically or mathematically relevant random modes 
and most of them are expected to be Gaussian (notable exception is the eigenvalue gaps 
that follow the Wigner-Dyson  statistics). Besides Lytova's work~\cite{MR3155024}, tracial CLTs for certain minors were obtained in~\cite{MR3805203}. 
Special functional CLTs have been proven for Haar distributed matrices~\cite{MR1062064, 2012.12950}, and for partial traces of invariant ensembles~\cite{1803.02151}. The free probability community has systematically studied Gaussian fluctuations of traces of products of  a Wigner matrix and deterministic matrices via the concept of \emph{second order freeness}~\cite{MR2302524, MR3585560}.   
This theory has recently been extended to polynomials in several independent Wigner matrices~\cite[Theorems~3--4]{2010.02963}. However, these results rely on the moment method 
and handle only  \emph{polynomials} of  Wigner matrices. It is yet unclear if the moment approach can be extended
to general functions on the macroscopic scale;  mesoscopic scales seem inaccessible.

Finally, we mention that the fluctuation of certain specific observables may be non-Gaussian. For example, the fluctuation of matrix entries \(f(W)_{ij}\) of \(f(W)\) for regular test functions \(f\) is a linear combination of \(w_{ij}\) and an independent Gaussian of size \(N^{-1/2}\), see~\cite{MR2489497,MR2880032,MR3090549, MR3600514,1103.2345}. In contrast, our result shows that \(\Tr f(W)A\) is always asymptotically Gaussian whenever \(\norm{A}\sim 1\) and \(\braket{AA^\ast}\gtrsim N^{-1+\epsilon}\). Hence, the non-Gaussian components of \(f(W)\) are only visible for very low rank observables \(A\).

The paper is structured as follows.  After this introduction, we present the main results in the next 
Section~\ref{sec:mainresult}. We start with our motivating Theorem~\ref{theo:sharpcom} 
on the Gaussian fluctuation of the overlaps  \(\langle\bu_i, A\bu_i\rangle\) 
after a small spectral averaging in \(i\). Then we formulate our  functional CLT (Theorem~\ref{theo:CLT})
in full generality in the bulk and at the edge of the spectrum of $W$,  from the macroscopic scale
down to the smallest possible  mesoscopic scale just above the local eigenvalue spacing.  
Our formulation exhibits the three distinguished fluctuation modes with their own scaling factors.
Simplified formulas in the mesoscopic regime for the expectations and the variances of the limit Gaussian processes are given in Proposition~\ref{prop bulk} in the bulk and in Proposition~\ref{prop edge}, respectively. We also include all the additional effects of the fourth 
cumulant  \(\kappa_4=N^2\E\abs{w_{12}}^4-2-\sigma^2\) of 
the off-diagonal matrix element \(w_{12}\), the parameter \(\sigma=N\E w_{12}^2\) describing the crossover regime between complex and real symmetry  class and the size of the diagonal element \(w_2=N\E w_{11}^2\). These three parameters appear in the exact form of the limiting expectations and variances of the  three different modes of \(L_N(f,A)\).
Some earlier works assumed special values of these parameters, e.g.\ \(\sigma =0, 1\) 
and \(w_2=1+\sigma\) is  a typical  choice in certain more restricted definition of the Wigner ensemble.
Consequently, some explicit terms did not always appear. We also identify the cases when
some of these three limiting modes have vanishing variance and explain their algebraic origin in Appendix~\ref{sec:vanvar}.
Finally, in Theorem~\ref{theo:inddiffeta} we show that fluctuations on different scales are asymptotically independent. 
In  Section~\ref{sec:locallaw} we present the necessary multi-resolvent local laws: some of them
have already been proven in~\cite{2012.13215}, some others, especially the ones
involving three resolvents,  are shown here with some proofs deferred to  Appendix~\ref{sec:impgag}. The main technical  input for  all these cases is~\cite[Theorem~\ref{eth-chain G underline theorem}]{2012.13215} and its slight extension in Theorem~\ref{general chain G underline theorem}, proven in Appendix~\ref{sec gen underline},
that control the most critical fluctuation term (the so-called renormalized ``underlined'' term)
in the self-consistent equation for products of resolvents
and deterministic matrices. Some additional technical estimates are deferred to Appendix~\ref{sec:addres}. In Section~\ref{sec:cltres} we prove a general CLT for resolvents; this section is the technical centrepiece of the current paper. Finally, in 
Section~\ref{sec:proosc}
we convert the resolvents into general functions by using  Helffer-Sj\"ostrand type-formulas
and thus prove our general functional CLT's.  The proof of Theorem~\ref{theo:sharpcom}
is given in full details in Section~\ref{sec:proosc}, while several technical calculations for 
the proof of the very general Theorem~\ref{theo:CLT} are deferred to 
Appendix~\ref{sec:proofclt}.

\subsection*{Notations and conventions}
We introduce some notations we use throughout the paper. For integers \(k\in\N \) we use the notation \([k]:= \{1,\dots, k\}\). For positive quantities \(f,g\) we write \(f\lesssim g\) and \(f\sim g\) if \(f \le C g\) or \(c g\le f\le Cg\), respectively, for some constants \(c,C>0\) which depend only on the moments of the matrix elements, i.e.\ on the constants appearing in~\eqref{eq:momentass}. We denote vectors by bold-faced lower case Roman letters \({\bm x}, {\bm y}\in\C ^k\), for some \(k\in\N\). Vector and matrix norms, \(\norm{\vx}\) and \(\norm{A}\), indicate the usual Euclidean norm and the corresponding induced matrix norm. For any \(N\times N\) matrix \(A\) we use the notation \(\braket{ A}:= N^{-1}\Tr  A\) to denote the normalized trace of \(A\). Moreover, for vectors \({\bm x}, {\bm y}\in\C^N\) we define
\[ 
\braket{ {\bm x},{\bm y}}:= \sum \overline{x}_i y_i, \qquad A_{\vx\vy}:=\braket{\vx,A\vy},
\]
with \(A\in\C^{N\times N}\).

We will use the concept of ``with very high probability'' meaning that for any fixed \(D>0\) the probability of the \(N\)-dependent event is bigger than \(1-N^{-D}\) if \(N\ge N_0(D)\). Moreover, we use the convention that \(\xi>0\) denotes an arbitrary small constant which is independent of \(N\).

\section{Main results}\label{sec:mainresult}
Let \(W\) be an \(N\times N\) real or complex Wigner matrix with eigenvalues \(\lambda_1\le \lambda_2 \le \dots \le \lambda_N\)
and corresponding  orthonormal eigenvectors  \({\bm u}_1,\dots, {\bm u}_N\). The eigenvalue density profile is described by the semicircular law 
\begin{equation}
  \label{eq:semicircle}
  \rho(x)=\rho_{\mathrm{sc}}(x):= \frac{\sqrt{4-x^2}}{2\pi}.
\end{equation}
On the entries of \(W\) we formulate the following assumptions.
\begin{assumption}\label{ass:entr}
  The matrix elements $w_{ab}$ of $W$ are independent up to Hermitian symmetry 
  $w_{ab}=\ov{w_{ba}}$. We assume identical distribution in the sense that \(w_{ab}\stackrel{\mathrm{d}}{=} N^{-1/2}\chi_{\mathrm{od}}\), for \(a<b\), \(w_{aa}\stackrel{\mathrm{d}}{=}N^{-1/2}\chi_{\mathrm{d}}\), with \(\chi_{\mathrm{d}}\) being a real, and \(\chi_{\mathrm{od}}\) being either a real or complex random variable such that \(\E\chi_{\mathrm{od}}=\E\chi_{\mathrm{d}}=0\), \(\E\abs{\chi_{\mathrm{od}}}^2=1\) and \(\sigma:=\E \chi_\mathrm{od}^2\in\R\). In addition, we assume the existence of the high moments of \(\chi_{\mathrm{od}}\), \(\chi_{\mathrm{d}}\), i.e.\ that there exist constants \(C_p>0\), for any \(p\in\mathbf{N}\), such that
  \begin{equation}\label{eq:momentass}
    \E\abs{\chi_{\mathrm{d}}}^p+\E\abs{\chi_{\mathrm{od}}}^p\le C_p.
  \end{equation}
\end{assumption}
Notice that \(\sigma \in [-1,1]\);  
the case \(\sigma=0\) corresponds to complex Hermitian Wigner matrices with \(\E w_{ab}^2=0\), the case \(\sigma=1\) corresponds to real symmetric matrices, and the case \(\sigma=-1\) corresponds Wigner matrices \(W=D+\ii O\), with \(D\) being a diagonal matrix and \(O\) being skew-symmetric, i.e.\ \(O^t=-O\).

Finally, in order to state our results compactly,  
we introduce the following notation to indicate that two random vectors have asymptotically equal moments.
\begin{notation}
  For two random vectors \(X=(X_1, \dots, X_k)\), \(Y=(Y_1,\dots, Y_k)\), with \(k\in\mathbf{N}\), of \(N\)-dependent random variables 
  we define of the concept of \emph{closeness in the sense of moments} and we denote it as 
  \[ X\stackrel{\mathrm{m}}{=} Y+\landauOd{N^{-c}} \]
  for some \(c>0\), if for any polynomial \(p(x_1,\dots,x_k)\) it holds that
  \[ \E p(X_1,\dots,X_k)=\E p(Y_1,\dots,Y_k)+\landauO{N^{-c+\xi}},\]
  for any small \(\xi>0\), where the implicit constant in \(\landauO{\cdot}\) depends on \(k,\xi\), the polynomial \(p\)
  and the constants  in Assumption~\ref{ass:entr}.
\end{notation}

\subsection{CLT for eigenvector overlaps}\label{sec:QUE}
As explained in the introduction,
the Gaussian fluctuation of the eigenvector  overlaps \(\braket{{\bm u}_i,A{\bm u}_i}\)
with a deterministic matrix $A$  is a fundamental question since 
it describes the  fluctuation in the strong Quantum Unique Ergodicity for Wigner matrices.
This problem has only been solved 
for finite rank \(A\), see~\cite{MR3606475,MR4156609,2005.08425}. 
Our first theorem 
establishes an averaged version of this CLT for general $A$\@.

\begin{theorem}[CLT for averages of eigenvector overlaps]\label{theo:sharpcom}
  Let \(A\) be a deterministic $N\times N$ matrix
  with  $\|A\|\le 1$ and let $\mathring{A}:= A- \braket{A}$ denote its traceless part.
  Let \(\epsilon>0\) and \(K\in \mathbf{N}\) with \(N^\epsilon\le K\le N^{1-\epsilon}\). 
  Then for some  \(\omega=\omega(\epsilon)>0\)  we have the CLT at the edge:
  \begin{equation}
    \frac{1}{\sqrt{K}}\sum_{i=N-K+1}^{N} \sqrt{N}\Big[\braket{{\bm u}_i,A{\bm u}_i}-\braket{A}\Big] \stackrel{\mathrm{m}}{=} \mathcal N\Bigl(0,\frac{2\sqrt{2}}{3}\big[\braket{\mathring{A}\mathring{A}^*}+\bm1(\sigma=1) \braket{\mathring{A}\overline{\mathring{A}}}\big]\Bigr) + \landauOd*{N^{-\omega}}.
  \end{equation}
  Moreover, for any \(\delta>0\) and \(\delta N<i_0<(1-\delta)N\) and \(\sigma>-1\) we have CLT in  the bulk:
  \begin{equation}\label{eq:QUE}
    \begin{split}
      \frac{1}{\sqrt{2K}}\sum_{|i-i_0|\le K} \sqrt{N}\Big[\braket{{\bm u}_i,A{\bm u}_i}-\braket{A}\Big] &\stackrel{\mathrm{m}}{=} \mathcal N\Bigl(0,\braket{\mathring{A}\mathring{A}^*}+\bm1(\sigma=1) \braket{\mathring{A}\overline{\mathring{A}}}\Bigr) 
      + \landauOd*{N^{-\omega}},
    \end{split}
  \end{equation}
  where the implicit constant in \(\landauOd*{\cdot}\) depends on \(\delta\). Finally, in case \(\sigma=-1\) for any fixed \(c\in(\delta,1-\delta)\) we have  a slightly different CLT in the bulk:
  \begin{equation}
    \frac{1}{\sqrt{2K}}\sum_{|i-cN|\le K} \sqrt{N}\Big[\braket{{\bm u}_i,A{\bm u}_i}-\braket{A}\Big]\stackrel{\mathrm{m}}{=} \mathcal N\Bigl(0,\braket{\mathring{A}\mathring{A}^*}+\bm1(c=1/2) \braket{\mathring{A}\overline{\mathring{A}}} \Bigr) + \landauOd*{N^{-\omega}}. 
  \end{equation}   
\end{theorem}

In the next subsection we formulate the  CLT for the functional linear statistics~\eqref{trfha}
in full generality for regular test functions $f$. Theorem~\ref{theo:sharpcom} is a special case of such CLT on  mesoscopic scales with $f$ essentially being the characteristic function of an interval. While this sharp cut-off test function formally does not satisfy the regularity condition imposed on \(f\) in Theorem~\ref{theo:CLT} below, in Section~\ref{sec:proosc} we will show how to cover  this special case as well.

\subsection{General functional CLT}
Let \(g\in H_0^2(\mathbf{R})\) be a compactly supported real valued test function, then for \(0\le a<1\) and \(\abs{E}\le 2\) we define the test function rescaled to a scale \(N^{-a}\) around \(E\) as 
\begin{equation}\label{eq:restestf}
  f(x):=g\bigl(N^a(x-E)\bigr).
\end{equation}
The scale \(a=0\) corresponds to the \emph{macroscopic regime}.  
The scales \(0<a<1\) in the bulk and \(0<a< 2/3\) at the edges, \(\abs{E}=2\),  correspond to the \emph{mesoscopic regime}. Our result holds uniformly in $E$, i.e.\ it also  covers the entire transitionary regime between bulk and edge.

For deterministic \(N\times N\) matrices \(A\), and test functions \(f\colon\R\to\R\) defined as in~\eqref{eq:restestf}, we define the centred linear statistics 
\begin{equation}\label{eq:CLT}
  L_N(f,A):=\sum_{i=1}^N f(\lambda_i)\braket{{\bm u}_i,A{\bm u}_i}-\E\sum_{i=1}^N f(\lambda_i)\braket{{\bm u}_i,A{\bm u}_i}.
\end{equation}
For the general CLT it is natural
to decompose the space of matrices in three mutually orthogonal subspaces. We will write a general matrix \(A\) as  
\[ A= A_\mathrm{d} + A_\mathrm{od} = \braket{A}I + \mathring{A_\mathrm{d}} + A_\mathrm{od}, \quad \mathring{A}:= A- \braket{A},\]
i.e, as the sum  of a constant multiple of the identity matrix, a diagonal traceless matrix \(\mathring{A_\mathrm{d}}\), and an off-diagonal matrix \(A_\mathrm{od}\). Given the decomposition of \(A\), the linear statistics has three modes 
\begin{equation}\label{linear stats decomp}
  L_N(f,A) = \braket{A} L_N(f,I) + L_N(f,\mathring{A}_\mathrm{d}) + L_N(f,A_\mathrm{od})
\end{equation}
which we prove to be asymptotically independent Gaussians. 

For sake of shorter notations we denote the expectation of a function \(f\) with respect to the semicircular density and its inverse by 
\[\braket{f}_\mathrm{sc}:= \int_{-2}^{2} f(x)\frac{\sqrt{4-x^2}}{2\pi}\dif x, \quad  \braket{f}_{1/\mathrm{sc}}:= \int_{-2}^{2} \frac{f(x)}{\pi\sqrt{4-x^2}}\dif x.\]
We  also define the  Stieltjes transform of the semicircle law 
\begin{equation}\label{msc}
  m(z)=m_{\mathrm{sc}}(z):=\int_{-2}^2 \frac{\rho_{\mathrm{sc}}(x)}{x-z} \dif x, \qquad z\in \C\setminus \R.
\end{equation}
We set  \(\rho(z):= \frac{1}{\pi}|\Im  m_{\mathrm{sc}}(z)|\) and
note that \(\rho(x+\ii 0) =\rho_{\mathrm{sc}}(x)\). 

Finally, we introduce a few notations related to the distribution of the matrix elements of $W$. We denote the normalised fourth cumulant of the off-diagonal entries, the expectation of the square of the off-diagonal entries, and variance of the diagonal entries of \(W\) and a certain frequently  used combination of them by  
\begin{equation}\label{eq kappa sigma w2 defs}
  \kappa_4:= \E \abs{\chi_\mathrm{od}}^4-2-\sigma^2, \quad \sigma:= \E \chi_\mathrm{od}^2,  \quad w_2:=\E \chi_\mathrm{d}^2,\quad\widetilde{w_2}:=w_2-1-\sigma,
\end{equation} 
respectively.

We now state our main result, the functional CLT in both the macroscopic \(a=0\) and mesoscopic \(a>0\) regimes. In Theorem~\ref{theo:CLT} we rescale the traceless diagonal linear statistics \(L_N(f,\mathring A_\mathrm{d})\) and the off-diagonal linear statistics \(L_N(f,A_\mathrm{od})\) in such a way the  limiting processes are, to leading order, \(N\)-independent except for the explicit dependence on \(\braket{\abs{A_\mathrm{d}}^2}\) and \(\braket{A_\mathrm{od}A_\mathrm{od}^\ast}\), irrespective of the scaling parameter \(a\) for test functions of the form~\eqref{eq:restestf}. In Subsection~\ref{meso subsection} below we provide explicit formulas for the mesoscopic limits of the processes in terms of \(g\), demonstrating the \(N\)-independence of \(L_N(f,\mathring A_\mathrm{d})\), \(L_N(f,A_\mathrm{od})\) to leading order.

\begin{theorem}[Macroscopic and mesoscopic functional CLT]\label{theo:CLT}
  Let  \(0\le a<1\) and define the scaling factor for any, possibly $N$-dependent,  \(E=E_N\in [-2,2]\) as
  \[C_N=C_N^{a,E}:=\frac{N^a}{\rho_N^{a,E}}, \quad\text{where}\quad\rho_N=\rho^{a,E}_N := \begin{cases}
    \rho(E+\ii N^{-a}), & a>0\\ 1, & a=0.
  \end{cases}
  \] 
  Note that \(C_N=1\) for the macroscopic \(a=0\) case. Let  \(g\in H_0^2(\R)\) be a compactly supported
  function and set \(f(x):=g(N^a(x-E))\). Let $A$ be a deterministic matrix with $\| A\|\le 1$. Then, in the limiting 
  the regime \(C_N\ll N\),  the 
  three centred linear statistics~\eqref{linear stats decomp} are approximately distributed (in the sense of moments)
  \[
  \begin{split}
  &\Bigl(L_N(f,I), \sqrt{C_N} L_N(f,\mathring{A}_\mathrm{d}),\sqrt{C_N} L_N(f,A_\mathrm{od})\Bigr) \\
  &\qquad\qquad\qquad\quad\stackrel{\mathrm{m}}{=} \Bigl(\xi_\mathrm{tr}(f),\xi_\mathrm{d}(f,\mathring{A}_\mathrm{d}),\xi_\mathrm{od}(f,A_\mathrm{od})\Bigr) + \landauOd*{\sqrt{\frac{C_N}{N}}}
  \end{split}
  \] 
  as three independent centred \(N\)-dependent
  Gaussian processes \(\xi_\mathrm{tr}(f),\xi_\mathrm{d}(f,\mathring{A}_\mathrm{d}),\xi_\mathrm{od}(f,A_\mathrm{od})\) whenever \(\braket{\abs{\mathring A_\mathrm{d}}^2},\braket{A_\mathrm{od}A_\mathrm{od}^\ast}\gtrsim C_N N^{-1+\epsilon}\) for some \(\epsilon>0\).
  Their variances are given by\footnote{
    The Gaussians are scaled such that \(\xi_\mathrm{tr}(f),\xi_\mathrm{d}(f,\mathring A_\mathrm{d})/\braket{\abs{\mathring A_\mathrm{d}}^2}^{1/2},\xi_\mathrm{od}(f,A_\mathrm{od})/\braket{A_\mathrm{od}A_\mathrm{od}^\ast}^{1/2}\) are of order one. The \(N\)-dependence of \(C_N\) is exactly offset by the \(N\)-dependence of \(V_\mathrm{d}^1 (f )\) etc.\ in the mesoscopic regime, see Section~\ref{meso subsection}.}
  \begin{align}
    \E \abs{\xi_\mathrm{tr}(f)}^2 &= V_{\mathrm{tr}}^1(f) + V_{\mathrm{tr}}^2(f,\sigma) + \frac{\kappa_4}{2}\braket{(2-x^2)f}_{1/\mathrm{sc}}^2 + \frac{\widetilde{w_2}}{4}\braket{xf}_{1/\mathrm{sc}}^2  \label{eq:expvarc}\\
    \E \abs{\xi_\mathrm{d}(f,\mathring{A}_\mathrm{d})}^2 &= C_N\braket{\abs{\mathring{A}_\mathrm{d}}^2} \Bigl(V_\mathrm{d}^1(f) + V_\mathrm{d}^2(f,\sigma) + \widetilde{w_2}\braket{fx}_\mathrm{sc}^2 + \kappa_4 \braket{(x^2-1)f}_\mathrm{sc}^2 \Bigr)\label{eq:expvarc2}\\
    \E \abs{\xi_\mathrm{od}(f,A_\mathrm{od})}^2 &= C_N \Bigl(\braket*{A_\mathrm{od}A_\mathrm{od}^*} V_\mathrm{d}^1(f) + \braket{A_\mathrm{od}\overline{A_\mathrm{od}}} V_\mathrm{d}^2(f,\sigma)\Bigr),\label{eq:expvarc3}
  \end{align}
  with
  \begin{align}
    V_\mathrm{tr}^1(f)&:=\frac{1}{4\pi^2}\iint_{-2}^2 \Bigl(\frac{f(x)-f(y)}{x-y}\Bigr)^2 \frac{4-xy}{\sqrt{(4-x^2)(4-y^2)}}\dif x \dif y,\label{eq:V1}\\
    V_\mathrm{tr}^2(f,\sigma)&:= \frac{1}{4\pi^2}\iint_{-2}^2 f(x)f(y) \partial_x\partial_y\log\left[\frac{(x-\sigma y)^2+(\sqrt{4-x^2}+\sigma\sqrt{4-y^2})^2}{(x-\sigma y)^2+(\sqrt{4-x^2}-\sigma\sqrt{4-y^2})^2}\right]\dif x \dif y,\label{eq:V2}\\
    V_\mathrm{d}^1(f)&:=\braket{f^2}_\mathrm{sc}-\braket{f}_\mathrm{sc}^2 \label{eq:V3}\\
    V_\mathrm{d}^2(f,\sigma)&:=\frac{1}{4\pi^2}\iint_{-2}^2 f(x)f(y) \frac{(1-\sigma^2)\sqrt{(4-x^2)(4-y^2)}}{\sigma^2(x^2+y^2)+(1-\sigma^2)^2-xy\sigma(1+\sigma^2)}  \dif x\dif y-\braket{f}_{\mathrm{sc}}^2, \label{eq:V4}
  \end{align}
  for \(\abs{\sigma}<1\), and \(V_\mathrm{tr}^2,V_\mathrm{d}^2\) are extended to \(\sigma=\pm1\) by continuity, \(V_\mathrm{tr/d}^2(f,\pm1):=\lim_{\sigma\to\pm 1}V_\mathrm{tr/d}^2(f,\sigma)\). Moreover, for any $\epsilon>0$, for the expectation of the linear statistics we have the expansions 
  \begin{align}
    &\E\sum_{i} f(\lambda_i)  =N\braket{f}_\mathrm{sc}+\frac{\kappa_4}{2}\braket{(x^4-4x^2+2)f}_{1/\mathrm{sc}} - \frac{\widetilde{w_2}}{2}\braket{(2-x^2)f}_{1/\mathrm{sc}} \\\nonumber
    &\qquad\qquad\qquad- \frac{E_\mathrm{tr}(f,\sigma)}{2} + \landauO*{N^\epsilon\sqrt{\frac{C_N}{N}}},\\\label{eq:exp1}
    &\abs*{\E\sum_i f(\lambda_i)\braket{{\bm u}_i,A_\mathrm{od}{\bm u}_i}} + \abs*{\E\sum_i f(\lambda_i)\braket{{\bm u}_i,\mathring A_\mathrm{d}{\bm u}_i}} =\landauO*{N^\epsilon\sqrt{\frac{C_N}{N}}},
  \end{align}
  where  
  \begin{equation}
    \label{eq:cex}
    E_\mathrm{tr}(f,\sigma):= \braket*{f\Bigl(1-\frac{1-\sigma^2}{(1+\sigma)^2-\sigma x^2}\Bigr)}_{\mathrm{1/sc}}, \quad \abs{\sigma}<1,
  \end{equation}
  and \(E_\mathrm{tr}(f,\pm1):=\lim_{\sigma\to\pm1}E_\mathrm{tr}(f,\sigma)\)\footnote{Note that \(E_{\mathrm{tr}}(f, \pm1)\) is defined as a limit \(\sigma\to \pm 1\) which is different from  plugging \(\sigma=\pm 1\) into~\eqref{eq:cex}. 
  In particular \(\sigma \to E_{\mathrm{tr}}(f, \sigma)\) is continuous on the closed interval \([-1,1]\).}. The implicit constants in
  $\landauO{\cdot}$ in all error terms above  depend only on the model parameters in Assumptions~\ref{ass:entr} and on $\norm{g}_{H_0^2}$, \(\abs{\supp g}\) (additionally in~\eqref{eq:exp1} the constant also depends on $\epsilon$), in particular  they are independent of $E$.
\end{theorem}

Theorem~\ref{theo:CLT} is only meaningful in the regime where \(C_N\ll N\), equivalently, when \(N^{-a}\) is above the local 
eigenvalue spacing around \(E\), by using that \(\rho(E+ iN^{-a}) \sim ( \abs[\big]{\abs{E}-2} + N^{-a})^{1/2}\). 
Thus our result covers the entire mesoscopic range uniformly for any $|E|\le 2$. In particular, we allow
for the  range \(a\in [0,1)\) in the bulk regime, $|E|\le 2-\epsilon$,  and  \(a\in [0,2/3)\) in the edge regime, $|E|=2$.

We note that the expectation of $\Tr f(W)$ is typically of order $N$, hence much larger than
its fluctuation. However, Theorem 2.4 identifies the leading term of $\E \Tr f(W)$ to an accuracy
beyond its fluctuation size. For both $\Tr f(W)\mathring{A}_{\mathrm{d}}$ and $\Tr f(W)A_{\mathrm{od}}$ their expectations are much
smaller than their fluctuation.

For simplicity, we formulated Theorem~\ref{theo:CLT} for linear statistics  with one test function \(f\) only.
Our method, however,
can handle linear combinations of test functions living on different scales since the main input of Theorem~\ref{theo:CLT}, the resolvent CLT in Theorem~\ref{CLT theorem}, allows  for each involved resolvent to be evaluated at its own spectral parameter
with possibly very different imaginary parts. Hence, by standard polarisation, 
a multivariate variant of Theorem~\ref{theo:CLT} directly follows:
\begin{corollary}[Multivariate CLT]\label{cor multivariate}
  Let \(p\in\N\), \(E_1,\ldots,E_p\in [-2,2]\), \(0\le a_1,\ldots,a_p<1\), and let \(g_1,\ldots,g_p\in H_0^2(\R)\) be compactly supported test functions and set and \(f_i(x):=g_i(N^{a_i}(x-E_i))\). Then for deterministic matrices \(A_1,\ldots,A_p\) of bounded norms, \(\norm{A_i}\lesssim1\) the joint linear statistics~\eqref{linear stats decomp} are approximately distributed (in the sense of moments)
  \[ 
  \begin{split}
    &\Bigl(L_N(f_i,I), \sqrt{C_N^{a_i,E_i}} L_N(f_i,(\mathring{A_i})_\mathrm{d}),\sqrt{C_N^{a_i,E_i}} L_N(f_i,(A_i)_\mathrm{od})\Bigr)_{i\in[p]} \\
    &\qquad \stackrel{\mathrm{m}}{=} \Bigl(\xi_\mathrm{tr}(f_i),\xi_\mathrm{d}(f_i,\mathring{(A_i)}_\mathrm{d}),\xi_\mathrm{od}(f_i, {(A_i)}_\mathrm{od})\Bigr)_{i\in[p]} + \landauOd*{\sqrt\frac{\max_i C^{a_i,E_i}_N}{N}}
  \end{split}\]
  as centred Gaussian processes \(\xi_\mathrm{tr},\xi_\mathrm{d},\xi_\mathrm{od}\) of covariances obtained from the variances in Theorem~\ref{theo:CLT} by polarisation, uniformly in $E_i\in [-2,2]$. The implicit constant in the  $\landauO{\cdot}$
  error term above only depends on the model parameters in Assumption~\ref{ass:entr} and on \(g\) via $\norm{g_i}_{H_0^2}$ and \(\abs{\supp g_i}\), in particular it is independent of $E_i$.
\end{corollary}

\begin{remark}[Alternative representation of the variances in Theorem~\ref{theo:CLT} via Chebyshev polynomials]\label{rmk Chebyshev}
  By a direct computation using the geometric series we find
  \begin{align}\label{eq:cheb1}
    V_\mathrm{tr}^1(f)&= \sum_{k\ge 1} k \braket{f t_k }_{\mathrm{1/sc}}^2,\qquad  V_\mathrm{tr}^2(f,\sigma) = \sum_{k\ge 1} k\sigma^k \braket{f t_k }_{\mathrm{1/sc}}^2 \\\label{eq:cheb2}
    V_\mathrm{d}^1(f)&= \sum_{k\ge 1}  \braket{f u_k }_{\mathrm{sc}}^2,\qquad  V_\mathrm{d}^2(f,\sigma) = \sum_{k\ge 1} \sigma^k \braket{f u_k }_{\mathrm{sc}}^2 ,
  \end{align} 
  where \(t_k(x):=T_k(x/2)\), \(u_k(x):=U_k(x/2)\) and \(T_k,U_k\) are the \(k\)-th Chebyshev polynomial of the first and second kind, i.e.\ \(T_k(\cos \theta)=\cos(k\theta)\), \(U_k(\cos \theta)=\sin[(k+1)\theta]/\sin \theta\). In particular, we can recover the representation of \(\E\abs{\xi_{\mathrm{tr}}(f)}^2\) obtained in~\cite[Eq.~(1.5)]{MR3568772} and write
  \begin{equation}\label{xi tr pos}
    \begin{split}
      \E\abs{\xi_\mathrm{tr}(f)}^2 &= \sum_{k\ge 3} k(1+\sigma^k)\braket{ft_k}_{\mathrm{1/sc}}^2 + 2(\kappa_4+1+\sigma^2)\braket*{ ft_2}_{\mathrm{1/sc}}^2 + w_2\braket*{ ft_1}_{\mathrm{1/sc}}^2.
    \end{split}
  \end{equation}
  Similarly, for the \(\xi_\mathrm{d},\xi_\mathrm{od}\) we obtain 
  \begin{align}\label{xi d pos}
    \E\abs{\xi_{\mathrm{d}}(f,\mathring{A}_{\mathrm{d}})}^2& = \braket{\abs{\mathring{A}_{\mathrm{d}}}^2}\biggl[ \sum_{k\ge 3} (1+\sigma^k) \braket{f u_k}_{\mathrm{sc}}^2 + w_2 \braket{f u_1}_\mathrm{sc}^2 + (\kappa_4+1+\sigma^2) \braket{f u_2 }_\mathrm{sc}^2 \biggr]\\\label{xi od pos}
    \E\abs{\xi_{\mathrm{od}}(f,A_{\mathrm{od}})}^2 & = \sum_{k\ge 1}\braket{fu_k}_\mathrm{sc}^2 \Bigl(\braket{A_\mathrm{od}A_\mathrm{od}^\ast} + \sigma^k \braket{A_\mathrm{od}\ov{A_\mathrm{od}}}\Bigr)  
  \end{align}
  Note, that~\eqref{xi tr pos}--\eqref{xi d pos} are sums of non-negative terms since \(w_2\ge 0\) and \(\kappa_4=\E\abs{\chi_\mathrm{od}}^4-2-\sigma^2\ge -1-\sigma^2\) due to \(\E\abs{\chi_\mathrm{od}}^4 \ge (\E\abs{\chi_\mathrm{od}}^2)^2 =1\). Similarly,~\eqref{xi od pos} is a sum of non-negative terms since \(\sigma^k\braket{A_\mathrm{od}\ov{A_\mathrm{od}}}\ge - \abs{\braket{A_\mathrm{od}\ov{A_\mathrm{od}}}}\ge -\braket{A_\mathrm{od}A_\mathrm{od}^\ast}\).
\end{remark}  

\begin{remark}[Explicit formulas for \(\sigma=\pm 1\)]
  The limits of~\eqref{eq:cex} are explicitly given by
  \begin{equation}
    E_{\mathrm{tr}}(f,1)=\braket{f}_{1/sc}-\frac{f(2)+f(-2)}{2}, \qquad E_{\mathrm{tr}}(f,-1)=\braket{f}_{1/sc}-f(0).
  \end{equation}
  For the variances in case \(\sigma=1\) we have \(V_{\mathrm{tr}}^2(f,1)=V_{\mathrm{tr}}^1(f)\) and \(V_{\mathrm{d}}^2(f,1)=V_{\mathrm{d}}^1(f)\), while for \(\sigma=-1\) we have 
  \begin{equation}
    \begin{split}
      V_{\mathrm{tr}}^2(f,-1)&=\frac{1}{4\pi^2}\iint_{-2}^2 \frac{(f(x)-f(y))(f(-x)-f(-y))}{(x-y)^2} \frac{4-xy}{\sqrt{(4-x^2)(4-y^2)}}\dif x \dif y,\\
      V_{\mathrm{d}}^2(f,-1)&=\braket{f(x)f(-x)}_{\mathrm{sc}}-\braket{f}_{\mathrm{sc}}^2.
    \end{split}
  \end{equation}
\end{remark}

\begin{remark}[Cases of vanishing variance in Theorem~\ref{theo:CLT}]\label{rem:vanish}
  From the Chebyshev representation in Remark~\ref{rmk Chebyshev} we can easily identify the necessary and sufficient conditions for the processes \(\xi_\mathrm{tr},\xi_\mathrm{d},\xi_\mathrm{od}\) to vanish\footnote{Note that in case \(\sigma=-1\) the condition on \(f\) differs for the three processes. For \(\xi_\mathrm{od}\) and symmetric \(A_\mathrm{od}=A_\mathrm{od}^t\) any odd function \(f\) results in \(\xi_\mathrm{od}(f,A_\mathrm{od})=0\), while for \(\xi_\mathrm{d}\) and \(\xi_\mathrm{tr}\) only odd functions \(f\) orthogonal to \(x\mapsto x\) with respect to \(\braket{\cdot}_\mathrm{sc}\) and \(\braket{\cdot}_\mathrm{1/sc}\), respectively, result in \(\xi_\mathrm{d},\xi_\mathrm{tr}\) to vanish. Thus, for example, \(\xi_\mathrm{od}(x^3)=0\), \(\xi_\mathrm{d}(x^3-2x)=0\) and \(\xi_\mathrm{tr}(x^3-3x)=0\).}.
  \begin{enumerate}[label=(\alph*)]
    \item\label{xi tr equiv} \(\xi_\mathrm{tr}(f)=0\) if and only if \(f\) is of the form
    \begin{equation}
      f(x)= \bm1(\sigma=-1) \Bigl(\phi(x)- \frac{\braket{\phi x}_\mathrm{1/sc}}{2} x\Bigr) + b + \bm1(w_2=0)cx + \bm1(\kappa_4=-1-\sigma^2) dx^2
    \end{equation}
    for some odd function \(\phi(-x)=-\phi(x)\) and \(b,c,d\in\R\). 
    \item\label{xi d equiv} For each fixed\footnote{Recall that the processes \(\xi_\mathrm{d},\xi_\mathrm{od}\) depend on \(N\) through \(\mathring A_\mathrm{d},A_\mathrm{od}\).} \(N\) we have \(\xi_\mathrm{d}(f,\mathring A_\mathrm{d})=0\) if and only if 
    either (i) \(\mathring A_\mathrm{d}=0\), or (ii) \(f\) is of the form 
    \begin{equation}
      f(x) = \bm1(\sigma=-1) \Bigl(\phi(x)- \braket{\phi x}_\mathrm{sc}x  \Bigr) + b + \bm1(w_2=0)cx + \bm1(\kappa_4=-1-\sigma^2) dx^2
    \end{equation}
    for some odd function \(\phi\) and \(b,c,d\in\R\). 
    \item\label{xi od equiv} For fixed \(N\) we have \(\xi_\mathrm{od}(f,A_\mathrm{od})=0\) if and only if either (i) \(A_\mathrm{od}=0\), or (ii) \(A_\mathrm{od}=-A_\mathrm{od}^t\), \(\sigma=1\), or (iii) \(A_\mathrm{od}=A_\mathrm{od}^t\), \(\sigma=-1\) and \(f(x)=b+\phi(x)\) for some odd function \(\phi\). 
  \end{enumerate}
  In Appendix~\ref{sec:vanvar} we will comment on why these cases naturally yield vanishing variances.
\end{remark}

\subsection{Computation of the expectations and variances in the mesoscopic regime}\label{meso subsection}

Theorem~\ref{theo:CLT} identified the expectations and the variances of the limiting processes \(\xi_\mathrm{tr},\xi_\mathrm{d},\xi_\mathrm{od}\) in terms of 
the test function \(f\). In case of mesoscopic test functions of the form \(f(x) = g(N^a(x-E)) \) with some scaling exponent \(a\), reference energy \(E\in [-2,2]\) and a compactly supported function \(g\in H_0^2(\R)\), we may compute the leading terms
of the variances in terms of \(g\). The result is different in the bulk (\(\abs{E}<2-\epsilon \)
for  any $\epsilon>0$ independent of $N$) and at the edge (\(E=\pm2\)), therefore here we explicitly
distinguish these two regimes. We note, however, that all error terms
in our main Theorem~\ref{theo:CLT} are valid uniformly in $E$, so this distinction is made here only in order to obtain simple limiting formulas. The proofs of the following two propositions follows from Theorem~\ref{theo:CLT} by simple mechanical computations, and so omitted. The variances can be conveniently expressed in terms of the \(L^2\) and \(\dot H^{1/2}\) inner products 
\[ \braket{f,g}_{L^2}:=\int_\R f(x)g(x)\dif x,\quad \braket{f,g}_{\dot H^{1/2}} := \int_{\R^2} \frac{f(x)-f(y)}{x-y}\frac{g(x)-g(y)}{x-y}\dif x \dif y.\] 
\begin{proposition}[Bulk scaling asymptotics]\label{prop bulk}
  Fix an \(a\in (0,1) \)  and an  $\epsilon>0$ (independent of $N$).  Then   for any \(\abs{E}\le 2-\epsilon\), the variances and expectation in Theorem~\ref{theo:CLT} have the following large \(N\) asymptotic behaviour
  \begin{equation}\label{bulkmeso}
    \begin{split}
      V_\mathrm{tr}^1(f) &= \frac{\norm{g}_{\dot H^{1/2}}^2}{4\pi^2} + \landauO*{N^{-a}},  \\ 
      V_\mathrm{tr}^2(f,\sigma) &= \bm1(\sigma=1) \frac{ \norm{g}_{\dot H^{1/2}}^2}{4\pi^2} + 
      \bm1(\sigma=-1)\bm1(E=0)\frac{ \braket{g(x),g(-x)}_{\dot H^{1/2}}}{4\pi^2} + \landauO*{N^{-a}},\\
      C_N V_\mathrm{d}^1(f) &= \norm{g}_{L^2}^2 + \landauO*{N^{-a}},\\
      C_N V_\mathrm{d}^2(f,\sigma) &= \bm1(\sigma=1) \norm{g}_{L^2}^2 + \bm1(\sigma=-1)\bm1(E=0) \braket{g(x),g(-x)}_{L^2} + \landauO*{N^{-a}},\\ 
      E_\mathrm{tr}(f,\sigma) &= \bm1(\sigma=-1)\bm1(E=0)\frac{g(0)}{2} + \landauO*{N^{-a}}.
    \end{split}
  \end{equation}
  The implicit constants in the error terms depend only on \(a, \epsilon\), $\norm{g}_{H_0^2}$, \(\abs{\supp g}\), and on $C_p$ from~\eqref{eq:momentass} and they are uniform in $E, \sigma$ in a specific sense explained in Remark~\ref{rem:unif}. 
\end{proposition}

\begin{proposition}[Edge scaling asymptotics]\label{prop edge}
  For\footnote{The case of the left edge, \(E=-2\) is completely analogous.} \(E=2\), and any  \(0<a< 2/3\) the variances and expectation in Theorem~\ref{theo:CLT} we have the scaling asymptotics
  \begin{equation}\label{edgemeso}
    \begin{split}
      V_\mathrm{tr}^1(f) &= \frac{\norm{g(-x^2)}_{\dot H^{1/2}}^2}{8\pi^2} + \landauO*{N^{-a/2}}, \\ 
      V_\mathrm{tr}^2(f,\sigma) &= \bm1(\sigma=1)\frac{ \norm{g(-x^2)}_{\dot H^{1/2}}^2}{8\pi^2} + \landauO*{N^{-a/2}},\\
      C_N V_\mathrm{d}^1(f) &= \frac{\norm{g(-x^2)x}_{L^2}^2}{\pi}  + \landauO*{N^{-a/2}},\\
      C_N V_\mathrm{d}^2(f,\sigma) &= \bm1(\sigma=1) \frac{\norm{g(-x^2)x}_{L^2}^2}{\pi}+ \landauO*{N^{-a/2}},\\ 
      E_\mathrm{tr}(f,\sigma) &= \bm1(\sigma=1)\frac{g(0)}{4} + \landauO*{N^{-a/2}}.
    \end{split}
  \end{equation}
  The implicit constants in the error terms depend only on \(a\), $\norm{g}_{H_0^2}$, \(\abs{\supp g}\) and on $C_p$ from~\eqref{eq:momentass} and on  $\sigma$ in a specific sense explained in Remark~\ref{rem:unif}. 
\end{proposition} 

\begin{remark}\label{rem:unif} Our proof also gives uniformity of the dependence on the  constants \( E, \sigma \) in the
  error terms in~\eqref{bulkmeso}--\eqref{edgemeso} in the following sense. In
  those formulas among~\eqref{bulkmeso}--\eqref{edgemeso}
  that contain \(\bm1(\sigma=1)\), the error is uniform in \(\sigma\le 1-\epsilon\) for any fixed \(\epsilon>0\)
  when \(\sigma\ne 1\).
  Similarly,  the presence of a factor \(\bm1(\sigma=-1)\)  in the formula comes with 
  uniformity for any \(\sigma\ge -1+\epsilon\) whenever \(\sigma\ne-1\). Finally, in terms with  \(\bm1(E=0)\)
  in~\eqref{bulkmeso}
  we have uniformity for any   \(|E|\ge \epsilon\), whenever \(E\ne 0\). In all other terms,
  our result is uniform for all $|E|\le 2-\epsilon$. See also Remark~\ref{rem:tran}.
\end{remark}

\begin{remark}
  In contrast to the macroscopic scale, note that on the mesoscopic scale the limits in Propositions~\ref{prop bulk}--\ref{prop edge} are independent on \(\kappa_4\) and \(w_2\) and their \(\sigma\)-dependence is via a very simple characteristic function. This shows that the mesoscopic fluctuations are less sensitive to the details of the ensemble, in agreement with the general paradigm that more local statistics are more universal. In fact for \(\sigma>-1\) the appearance of \(\bm1(\sigma=1)\) in the variance \(V_\mathrm{tr}^2\), \(V_\mathrm{d}^2\) corresponds to a factor of \(2\) difference between real symmetric and complex Hermitian symmetry classes. Furthermore, for \(\sigma=-1\), assuming \(w_2=0\),  we have \(W= \ii O\), where \(O= -O^t\) is a skew symmetric matrix, in particular the spectrum of \(W\) is symmetric with respect to zero, i.e.\ the eigenvalues around some energy \(E\) and \(-E\) 
  are strongly dependent. On mesoscopic scale this feature is relevant only for \(E=0\) and it changes the expectation and the variance. In particular, for antisymmetric test functions, \(g(x)= - g(-x)\), we have \(L_N(f, A) = 0\), and indeed, the variances in Proposition~\ref{prop bulk} add up to zero in this case.
\end{remark}

Additionally, we prove that the linear statistics for test functions living on different scales are asymptotically independent. The proof of the following theorem follows by standard arguments completely analogous to the proof of Theorem~\ref{theo:CLT} and is presented in Section~\ref{sec:indlinstat}.

\begin{theorem}\label{theo:inddiffeta}
  Let $\epsilon>0$ and \(E_1,E_2\in [-2+\epsilon,2-\epsilon]\), \(0\le a_1\ne a_2<1\) and let \(g_1,g_2\in H^2_0(\R)\) be compactly supported  functions and set \(f_i(x):=g_i(N^{a_i}(x-E_i))\). Then the limiting Gaussian processes \(\xi_\mathrm{tr}(f_1),\xi_\mathrm{tr}(f_2)\) from Theorem~\ref{theo:CLT} are asymptotically independent in the sense 
  \begin{equation}
  \label{xi tr cov}
    \abs*{\Cov(\xi_\mathrm{tr}(f_1),\xi_\mathrm{tr}(f_2))} \lesssim N^{-\abs{a_1-a_2}}.
  \end{equation}
  Similarly, for bounded deterministic matrices \(A_1,A_2\) the processes \(\xi_\mathrm{d},\xi_\mathrm{od}\) are asymptotically independent in the sense 
  \begin{equation}
    \abs*{\Cov\bigl(\xi_\mathrm{d}(f_1),\xi_\mathrm{d}(f_2)\bigr) }+\abs*{\Cov\bigl(\xi_\mathrm{od}(f_1,A_1),\xi_\mathrm{od}(f_2,A_2)\bigr) } \lesssim N^{-\abs{a_1-a_2}/2}.
  \end{equation}
\end{theorem}

To make our presentation simpler we stated this result only in the bulk, but our proof naturally yields the independence of linear statistics living on different scales uniformly in the spectrum. Moreover, the same argument also yields 
independence of linear statistics living on the same scale at distant energies, i.e.\ for $a_1=a_2=a$ and  $|E_1-E_2|\gg N^{-a}$.

Theorem~\ref{theo:inddiffeta} together with Theorem~\ref{theo:CLT} imply the asymptotic independence of linear statistics living on different scales \(0\le a_1\ne a_2<1\) in the sense 
\begin{equation}\label{assymp Indep LN}
    \abs*{\Cov(L_N(f_1,I),L_N(f_2,I))} \lesssim N^{-\abs{a_1-a_2}} + N^{(a_1-1)/2} + N^{(a_2-1)/2},
\end{equation}
and similarly for \(\sqrt{C_N}L_N(f_i,\mathring A_\mathrm{d}),\sqrt{C_N}L_N(f_i,A_\mathrm{od})\). We note, however, that for large \(\abs{a_1-a_2}\) the estimate on the covariance of linear statistics in~\eqref{assymp Indep LN} may be larger than that of the limiting processes in~\eqref{xi tr cov} owing to the error terms from Theorem~\ref{theo:CLT}.

\subsection{Related earlier results and miscellaneous  remarks}\label{sec:misc}
The  linear eigenvalue statistics \(\smash{\sum_i f(\lambda_i)}\) have been extensively studied, and a CLT has been proven for macroscopic test functions as well as for mesoscopic test functions down to the optimal scale both in the bulk and at the edge,
hence our results  on \(\xi_\mathrm{tr}(f)\) are not new, we only listed them for completeness.
More precisely, the explicit form of the variance \(\E \abs{\xi_\mathrm{tr}(f)}^2\) for macroscopic test functions in~\eqref{eq:expvarc} exactly agrees with~\cite[Eq.~(3.92)]{MR2561434} for \(w_2=2/\beta\) and with~\cite[Eq.~(1.10)]{MR2829615} for the case when \(w_2\ne 2/\beta\). Note that the  parameter \(\beta\), customary in random matrix theory distinguishing  between
the real symmetric \((\beta=1)\) and complex Hermitian \((\beta=2)\) symmetry classes, corresponds to \(\beta=2/(1+\sigma)\)
with our notation 
in the cases  \(\sigma=0,1\).

For mesoscopic test functions the variance \(\E \abs{\xi_\mathrm{tr}(f)}^2\) in~\eqref{eq:expvarc}
with~\eqref{bulkmeso} in the bulk and with~\eqref{edgemeso} at the edge exactly agree with~\cite[Eq.~(2.22)]{1909.12821} and~\cite[Eq.~(2.23)]{1909.12821},~\cite[Eq.~(2.6)]{MR3678478}, respectively,  
in case of \(\sigma=0, 1\). Our formulas for general \(\sigma\) agree with the results in~\cite{MR3959983} for \(\sigma\in (-1,1]\), however the final formula for the variance in case  \(\sigma =-1\)appears to be wrong in~\cite{MR3959983} (probably the error stems from~\cite[Eq.~(6.25)]{MR3959983} overlooking that
\(|\widetilde{T}|\) is not far away from zero, in fact \(|\widetilde{T}|\sim \eta\) in this case).

As far as the expectation (density of states) is concerned, 
the explicit formula for \(\E\smash{\sum_i f(\lambda_i)}\) in~\eqref{eq:exp1} with~\eqref{eq:cex}
exactly agrees with the formula given in~\cite[Theorem 1.1]{MR2556016} for \(\sigma\in \{0,1\}\)
and with~\cite[Eq. (1.4)]{MR3568772} for the general case. We also mention that for the Gaussian case explicit \(N\)-dependent formulas are obtained in~\cite{MR3137043} on the  density of states 
by supersymmetric methods.

The joint linear statistics of eigenvalues and eigenvectors with observable \(A\ne I\), i.e.\ quantities \(\mathrm{Tr}[f(W)A]=\smash{\sum_i f(\lambda_i)}\braket{u_i,Au_i}\) are much less studied. For macroscopic test functions \(f\) the variances \( \E \abs{\xi_\mathrm{d}(f,\mathring{A}_d)}^2\), \(\E \abs{\xi_\mathrm{od}(f,A_\mathrm{od})}^2\) in~\eqref{eq:expvarc} exactly agree with~\cite[Eq.~(4.16), Eq.~(4.19)]{MR3155024} in the real symmetric case.  
For mesoscopic test functions \(f\) the current paper achieves the first results on the limiting distribution of  \(\mathrm{Tr}[f(W)A]\), with \(A\ne I\), in particular, explicit formulas for \( \E \abs{\xi_\mathrm{d}(f,\mathring{A}_d)}^2\) and
\(\E \abs{\xi_\mathrm{od}(f,A_\mathrm{od})}^2\) in~\eqref{eq:expvarc2}--\eqref{eq:expvarc3}, with 
their limiting behaviour  in~\eqref{bulkmeso} and~\eqref{edgemeso}, are new.
\begin{remark}
  In~\eqref{eq:restestf} we assumed that \(g\in H^2_0(\R)\)  to make the proof cleaner. The proof of the functional CLT (Theorem~\ref{theo:CLT}) on the macroscopic scale (\(a=0\))
  presented in Appendix~\ref{sec:proofclt} would work exactly in the same way if~\(f\in W^{2+\epsilon,1}(\R)\), for some small fixed \(\epsilon>0\). The only difference is that throughout the proof we have to replace \(f\) by its cut-off version,
  \(f_\chi:= f\chi\), with \(\chi\) a smooth cut-off function that is equal to one on \([-5,5]\) and equal to zero on \([-10,10]^c\).
\end{remark}

\begin{remark}\label{rem:tran} The formulas in Propositions~\ref{prop bulk}--\ref{prop edge} indicate a somewhat different 
  limiting expectation and variance when \(\sigma=\pm1\) in contrast to the \(|\sigma|<1\) case. With our methods it is
  also possible to study the transitional regime, where \(1-|\sigma|\) vanishes as an \(N\)-power, as it was done for the 
  tracial part in~\cite{MR3959983}, but we refrained from doing so in order to keep the paper more transparent.
\end{remark}

\section{Local laws for multiple resolvents}\label{sec:locallaw}

Given a Wigner matrix \(W\), we define its resolvent by \(G(z):=(W-z)^{-1}\), with \(z\in\mathbf{C}\setminus \mathbf{R}\). In this paper we consider resolvents allowing the spectral parameter \(z\) to have positive or negative imaginary part, in order to conveniently account for possible adjoints of the resolvent since \(G(z)^*=G(\overline{z})\).

In this section we prove local laws for one resolvent and for certain products of two or three resolvent that will be used as an input to prove the Central Limit Theorem for resolvents in Section~\ref{sec:cltres}. These local laws are stated in Propositions~\ref{pro:1g}--\ref{pro:3g}. Additionally, in Lemma~\ref{lem:impvarbclt} we present an improvement for the bound of \(\braket{\vx,\un{GAWG}\vy }^2\) in~\eqref{eq:isll2G}, which we need only in a second moment sense. The main inputs for the proof of these local laws are the bounds in~\cite[Theorem~\ref{eth-chain G underline theorem}]{2012.13215}.

As \(N\to \infty\) the resolvent \(G\) becomes approximately deterministic (local laws). Its deterministic approximation is given by \(m(z)=m_{\mathrm{sc}}(z)\), with \(m_{\mathrm{sc}}(z)\) being the Stieltjes transform~\eqref{msc} of the semicircular law \(\rho_{\mathrm{sc}}\) defined in~\eqref{eq:semicircle}. In particular, \(m=m(z)\) is given by the unique solution of the quadratic equation
\begin{equation}\label{eq:mde}
  -\frac{1}{m}=z+m, \qquad \Im m (z)\Im z>0.
\end{equation}
Recall that the density \(\rho(z)\) is defined as \(\rho(z):=\pi^{-1}|\Im m(z)|\).

In order to formulate the local laws concisely we introduce the commonly used notion of \emph{stochastic domination}. 
\begin{definition}[Stochastic Domination]\label{def:stochDom}
  If \[
  X=\tuple*{ X^{(N)}(u) \given N\in\N, u\in U^{(N)} }\quad\text{and}\quad Y=\tuple*{ Y^{(N)}(u) \given N\in\N, u\in U^{(N)} }\] are families of non-negative random variables indexed by \(N\), and possibly some parameter \(u\), then we say that \(X\) is stochastically dominated by \(Y\), if for all \(\epsilon, D>0\) we have \[\sup_{u\in U^{(N)}} \Prob\left[X^{(N)}(u)>N^\epsilon  Y^{(N)}(u)\right]\leq N^{-D}\] for large enough \(N\geq N_0(\epsilon,D)\). In this case we use the notations \(X\prec Y\) and \(X=\landauOprec{Y}\).
\end{definition}
In addition to the \(\landauOprec{\cdot}\) notation   indicating a stochastic domination in the sense of arbitrary high moments, in this proof we introduce two related new notations, \(\landauOE{\cdot},\landauOstd{\cdot}\), indicating domination only in first and second moment sense. More precisely, we write \(X=\landauOstd{\psi}\) and \(X=\landauOE{\psi}\) if \(\E \abs{X}^2\lesssim N^\epsilon\psi^2\) and \(\E\abs{X}\lesssim N^{\epsilon} \psi\), respectively, for any \(\epsilon>0\) and some deterministic \(\psi\). We note that we trivially have the following product estimates 
\begin{subequations}
  \begin{align}\label{prec221}
    X=\landauOstd{\phi}, Y=\landauOstd{\psi} \quad&\Rightarrow\quad XY=\landauOE{\phi\psi},\\\label{prec11}
    X=\landauOE{\phi}, Y=\landauOprec{\psi} \quad&\Rightarrow\quad XY=\landauOE{\phi\psi},\\\label{prec22}
    X=\landauOstd{\phi}, Y=\landauOprec{\psi} \quad&\Rightarrow\quad XY=\landauOstd{\phi\psi},
  \end{align}
\end{subequations}
so that by~\eqref{prec221}, in particular, \(X=\landauOstd{\psi}\) implies \(X=\landauOE{\psi}\).

We start with the statement of the local laws for single resolvents.
\begin{proposition}[Single \(G\) local laws]\label{pro:1g}
  Let \(z\in \C\setminus\R\). We use the notation \(\eta:=|\Im z|\), \(\rho=\rho(z)\), \(m=m_{\mathrm{sc}}(z)\). Then for any deterministic matrix \(A\) with \(\lVert A\rVert\lesssim 1\) and \(\braket{A}=0\) we have the averaged local laws
  \begin{equation}
    \label{eq:avll1G}
    |\braket{G-m}|\prec \frac{1}{N\eta}, \qquad |\braket{GA}|\prec \frac{\sqrt{\rho}}{N\sqrt{\eta}}.
  \end{equation}
  Additionally, for any deterministic vectors \({\bm x}, {\bm y}\) such that \(\lVert {\bm x}\rVert+\lVert {\bm y}\rVert\lesssim 1\), we have the isotropic law
  \begin{equation}
    \label{eq:isll1G}
    |\braket{{\bm x}, (G-m){\bm y}}|\prec \sqrt{\frac{\rho}{N\eta}}.
  \end{equation}
\end{proposition}

The local law for \(\braket{GA}\) in~\eqref{eq:avll1G} is proven in~\cite[Theorem~\ref{eth-pro:1g}]{2012.13215}. The averaged and isotropic law for \(G-m\) have been proven in~\cite{MR2871147, MR3183577, MR3103909}.

Next, we state averaged and isotropic local laws for products of two resolvents.  
\begin{proposition}[\protect{Local laws for two \(G\)'s}]\label{pro:2g}
  Let \( z_1,z_2\in \C\setminus \R\) and let \(G_i:=G(z_i)\), for \(i\in \{1,2\}\). We use the notation \(\eta_i:=|\Im z_i|\), \(\rho_i=\rho(z_i)\), \(m_i=m_{\mathrm{sc}}(z_i)\), and set \(K:=N\eta_*\rho^*\), \(L:=N\min_i(\eta_i\rho_i)\), where \(\eta_*:=\eta_1\wedge\eta_2\) and \(\rho^*=\rho_1\vee \rho_2\). Then for any deterministic matrices \(A,A'\), with \(\lVert A\rVert+\lVert A'\rVert\lesssim 1\) and \(\braket{A}=\braket{A'}=0\), we have the averaged local laws\footnote{The second error term in~\eqref{eq:avll2G2} is uniform in $\sigma$ as long as $|\sigma|\le 1-\epsilon'$ for any fixed $\epsilon'>0$.} 
  
  \begin{equation}
    \label{eq:avll2G1}
    \braket{G_1G_2}=\frac{m_1m_2}{1- m_1m_2}+\landauOprec*{\frac{1}{N\eta_1\eta_2}},  \quad \braket{G_1AG_2A'}=m_1m_2\braket{AA'}+\landauOprec*{\frac{\rho^*}{\sqrt{K}}},  
  \end{equation}
  \begin{equation}
    \label{eq:avll2G2}
    \braket{G_1G_2^t}=\frac{m_1m_2}{1-\sigma m_1m_2}+\landauOprec*{\frac{\bm1(\sigma=\pm 1)}{N\eta_1\eta_2}+\bm1(|\sigma|<1)\left[\frac{1}{N\eta_*^2}\wedge \frac{\rho^*}{\sqrt{K}}\right]}
  \end{equation}
  \begin{equation}
    \label{eq:avll2G22}
    \braket{G_1AG_2^t A'}=m_1m_2\braket{AA'}+\landauOprec*{\frac{\rho^*}{\sqrt{K}}}. 
  \end{equation}
  We also have the following bounds
  \begin{equation}\label{eq:avll2G4}
    \left|\braket{G_1G_2A}\right|+\left|\braket{G_1G_2^t A}\right|=\landauOstd*{\frac{\sqrt{\rho_1\rho_2}}{\sqrt{NL\eta_1\eta_2}}}.
  \end{equation}
  Moreover, for any deterministic vectors \({\bm x}, {\bm y}\) such that \(\lVert {\bm x}\rVert+\lVert {\bm y}\rVert\lesssim 1\) we have the isotropic laws
  \begin{equation}
    \label{eq:isll2G}
    \braket{{\bm x}, G_1G_2 {\bm y}}=\frac{m_1m_2}{1-m_1m_2}\braket{{\bm x},  {\bm y}}+\landauOprec*{\frac{\sqrt{\rho^*}}{\sqrt{N\eta_*}\eta^*}}, \qquad |\braket{{\bm x}, G_1AG_2 {\bm y}}|\prec \sqrt{\frac{\rho^*}{\eta_*}},
  \end{equation}
  where \(\eta^*:=\eta_1\vee \eta_2\).
\end{proposition}

Now we state averaged laws for certain products of three resolvents.

\begin{proposition}[Local laws for three \(G\)'s]\label{pro:3g}
  Let \( z_1,z_2\in \mathbf{C}\setminus \mathbf{R}\) and let \(G_i:=G(z_i)\), for \(i\in \{1,2\}\). We use the notation  \(\eta_i:=|\Im z_i|\), \(\rho_i=\rho(z_i)\), \(m_i=m_{\mathrm{sc}}(z_i)\), and set \(K:=N\eta_*\rho^*\), \(L:=N\min_i(\eta_i\rho_i)\), where \(\eta_*:=\eta_1\wedge\eta_2\) and \(\rho^*=\rho_1\vee \rho_2\), \(\rho_*=\rho_1\wedge \rho_2\). Then for any deterministic matrices \(A,A'\), with \(\lVert A\rVert+\lVert A'\rVert\lesssim 1\),  \(\braket{A}=\braket{A'}=0\), we have the averaged local laws
  \begin{align}
    \braket{G_1G_2^2}&=\frac{m_1m_2'}{(1-m_1m_2)^2}+\landauOprec*{\frac{\rho_*}{L\eta_1\eta_2}}, \label{eq:avll3G1}\\
    \braket{G_1G_2AG_1A'}&=\frac{m_1^2m_2\braket{AA'}}{1-m_1m_2}+\landauOstd*{\frac{\sqrt{\rho_1\rho_2}}{\sqrt{L\eta_1\eta_2}}}, \label{eq:avll3G11}\\
    \braket{G_1(G_2^t)^2}&=\frac{m_1m_2'}{(1-\sigma m_1m_2)^2}+\landauOprec*{\frac{\rho_1}{\sqrt{L}\eta_1\eta_2}}, \label{eq:avll3G2}\\
    \braket{G_1G_2^t AG_1A'}&=\frac{m_1^2m_2\braket{AA'}}{1-\sigma m_1m_2}+\landauOstd*{\frac{\sqrt{\rho_1\rho_2}}{\sqrt{L\eta_1\eta_2}}}. \label{eq:avll3G21}
  \end{align}
  Additionally, we have the following bounds
  \begin{equation}\label{eq:avll3G3}
    \left|\braket{G_1G_2^2A}\right|+\left|\braket{G_1^t G_2^2A}\right|=\landauOstd*{ \frac{\sqrt{\rho_*}}{L\sqrt{\eta_1}\eta_2}}.
  \end{equation}
\end{proposition}


The local laws and bounds in \eqref{eq:avll2G1}--\eqref{eq:isll2G} and \eqref{eq:avll3G1}--\eqref{eq:avll3G3} all have the structure that the first term in the rhs. is the explicit leading term. The error term in the rhs. is smaller than the typical size of the leading term using $L\gg 1$, the fact that $|m|\sim 1$, $|m'|\sim \rho^{-1}$, and the bound
\begin{equation}
\label{eq:boundm1m2}
  \abs*{\frac{1}{1-m_1m_2}}\lesssim \left.\begin{cases}
    1/\rho^\ast, & \sgn(\Im z_1)=\sgn(\Im z_2),\\
    \sqrt{\rho_1\rho_2/\eta_1\eta_2},&\text{else},
  \end{cases}\right\}\lesssim\sqrt{\frac{\rho_1\rho_2}{\eta_1\eta_2}}
\end{equation}
which follow from elementary calculus.
In the sequel we will often use these local laws in their weaker form just as an upper bound for the lhs. in terms of the upper estimate on the leading term on the rhs. For example, \eqref{eq:avll2G1} together with \eqref{eq:boundm1m2} implies
\[
\big|\braket{G_1G_2}\big|\prec \sqrt{\frac{\rho_1\rho_2}{\eta_1\eta_2}},
\]
and similarly for all the other local laws.

For any given functions \(f,g\) of the Wigner matrix \(W\) we define the renormalisation of the product \(g(W)Wf(W)\) (denoted by underline) as follows:
\begin{equation}
  \label{eq:defunder}
  \underline{g(W)Wf(W)}:=g(W)Wf(W)-\widetilde{\mathbf{E}}g(W)\widetilde{W}(\partial_{\widetilde{W}}f)(W)-\widetilde{\mathbf{E}}(\partial_{\widetilde{W}}g)(W)\widetilde{W}f(W),
\end{equation}
where \(\partial_{\widetilde W} f(W)\) denotes the directional derivative of the function \(f\) in the direction \(\widetilde{W}\) at the point \(W\), and \(\widetilde{W}\) is an independent copy of \(W\). The definition is chosen such that it subtracts the second order term in the cumulant expansion, in particular if all entries of \(W\) were  Gaussian then we  had
\(\mathbf{E} \underline{g(W)Wf(W)}=0\). Note that the definition~\eqref{eq:defunder} only makes sense if it is clear to which \(W\) the underline refers, i.e.\ it would be ambiguous if \(f(W)=W\). In our applications, however, each underlined term contains exactly a single \(W\) factor, and hence such ambiguities will not arise.

The key inputs for the proof of the local laws with two or three $G$'s are strong bounds for renormalised products of the form $\braket{\underline{WG_1 B_1 G_2 \dots G_l B_l}}$.
In Theorem~\ref{eth-chain G underline theorem} of our companion paper~\cite{2012.13215} we proved such estimates but they are in terms $\eta_*$, the minimal of all $\eta$'s, i.e.\ no distinction among different $\eta$'s is made. To remedy this situation, in the following Theorem~\ref{general chain G underline theorem} we prove a generalization of~\cite[Theorem~\ref{eth-chain G underline theorem}]{2012.13215} which allows for the proof of the local laws for two and three \(G\)'s with distinguished $\eta$-dependencies 
as stated above. Furthermore, for a few specific terms we need a somewhat stronger
bound than our general Theorem~\ref{general chain G underline theorem}  gives, but we  need them only in variance sense in contrast to the high probability bounds in Theorem~\ref{general chain G underline theorem}. These specific bounds are listed separately in Lemma~\ref{lem:impvarbclt}. The proof of Theorem~\ref{general chain G underline theorem} is presented in Appendix~\ref{sec gen underline} and the proof of Lemma~\ref{lem:impvarbclt} in Appendix~\ref{sec:impgag}. 
\begin{theorem}\label{general chain G underline theorem} 
  Fix $\epsilon>0$,  let \(l,n_1,\ldots,n_l\in\N\), \(z_{1,1},\ldots,z_{1,n_1}, z_{2,1},\ldots z_{l,n_l}\in \C\setminus\R\) and for \(k\in[l]\), \(j\in[n_k]\) let
  \[G_k\in\set{G_{k,1}G_{k,2}\cdots G_{k,n_k},(G_{k,1}G_{k,2}\cdots G_{k,n_k})^{t}}, \quad G_{k,j}\in \set{G(z_{k,j}),\Im G(z_{k,j})}\] 
  and let \(B_k\) be deterministic \(N\times N\) matrices, and \(\vx,\vy\) be deterministic vectors with bounded norms \(\norm{B_k}\lesssim1\), \(\norm{\vx}+\norm{\vy}\lesssim1\). Set 
  \begin{equation}
    L := N\min_{k} (\eta_k\rho_k),\quad \rho^\ast := \max_{k}\rho_k,\quad \eta_\ast:=\min_k \eta_k, 
  \end{equation}
  with \(\eta_k:=\abs{\Im z_k}\), \(\rho_k:=\rho(z_k)=\abs{\Im m(z_k)}/\pi\) and assume \(L\ge N^\epsilon\) and \(\eta_\ast\lesssim 1\). Let \(\mathfrak a,\mathfrak t\) denote disjoint sets of indices, \(\mathfrak a\cap \mathfrak t=\emptyset\), such that for each \(k\in\mathfrak{a}\) we have \(\braket{B_k}=0\), and for each \(k\in\mathfrak{t}\) exactly one of \(G_k,G_{k+1}\) is transposed, where in the averaged case and \(k=l\) it is understood that \(G_{l+1}=G_1\). Then with \(a:=\abs{\mathfrak a}, t:=\abs{\mathfrak t}\), we have the following bounds: 
  \begin{itemize}
    \item[(av1)] For \(\mathfrak{a}=\mathfrak{t}=\emptyset\) we have
    \begin{equation}\label{eq:under B}
      \begin{split}
        \abs{\braket{\un{WG_1B_1G_2B_2\cdots G_l B_l}}} &\prec \frac{\rho^{\ast}}{ N \eta_\ast^{l} }\prod_{k\in[l]}  \frac{\min_i \eta_{k,i}}{\eta_{k,1} \cdots \eta_{k,n_k}}.
      \end{split}   
    \end{equation}
    \item[(av2)] For \(\mathfrak a,\mathfrak t\subset[l]\), \(\abs{\mathfrak a\cup \mathfrak t}\ge 1\) we have the bound
    \begin{equation}\label{eq:under} 
      \begin{split}
        \abs{\braket{\un{WG_1B_1G_2B_2\cdots G_l B_l}}} &\prec \frac{(\sqrt{N}\eta_\ast)^{a+t}}{N\eta_\ast^l} \sqrt{\frac{\rho^{\ast}}{N\eta_\ast}} \prod_{k\in[l]}  \frac{\min_i \eta_{k,i}}{\eta_{k,1} \cdots \eta_{k,n_k}}.
      \end{split} 
    \end{equation} 
    \item[(iso)] For \(\mathfrak a,\mathfrak t\subset[l-1]\) and for any \(0\le j<l\) we have the bound 
    \begin{equation}\label{eq:underIso}
      \abs{\braket{\vx,\un{G_1 B_1\cdots G_{j} B_j W G_{j+1}B_{j+1}\cdots B_{l-1}G_l }\vy}}  \prec \frac{(\sqrt{N}\eta_\ast)^{a+t}}{\eta_\ast^{l-1}} \sqrt{\frac{\rho^{\ast}}{N\eta_\ast}}\prod_{k\in[l]}  \frac{\min_i \eta_{k,i}}{\eta_{k,1} \cdots \eta_{k,n_k}},
    \end{equation}
    where the \(j=0\) case is understood as \(\braket{\vx,\un{W G_1 B_1\cdots B_{l-1}G_l}\vy}\). 
  \end{itemize}
  In case \(\prod_{k\in\mathfrak i} \rho_k\lesssim (\rho^\ast)^{b+1}\), the bounds~\eqref{eq:under B}--\eqref{eq:underIso} remain valid if the rhs.\ are multiplied by the factor \((\rho^\ast)^{-b-1}\prod_{k\in\mathfrak i} \rho_k\), where \(b:=l\) in case of~\eqref{eq:under B}, \(b:=l-a-t\) in case of~\eqref{eq:under}, and \(b:=l-a-t-1\) in case of~\eqref{eq:underIso}. Moreover, for any \(\eta_\ast\ge 1\) we have the bounds 
  \begin{equation}\label{eq large eta bounds}
    \begin{split}
      \abs{\braket{\un{WG_1B_1G_2B_2\cdots G_l B_l}}} &\prec \frac{1}{N\eta_\ast^l}\prod_{k\in[l]}  \frac{\min_i \eta_{k,i}}{\eta_{k,1} \cdots \eta_{k,n_k}},\\  
      \abs{\braket{\vx,\un{G_1 B_1\cdots G_{j} B_j W G_{j+1}B_{j+1}\cdots B_{l-1}G_l }\vy}} &\prec \frac{1}{N^{1/2}\eta_\ast^l}\prod_{k\in[l]}  \frac{\min_i \eta_{k,i}}{\eta_{k,1} \cdots \eta_{k,n_k}}.
    \end{split}
  \end{equation}
\end{theorem}


\begin{lemma}\label{lem:impvarbclt}
  Let \(z,z_1,z_2\in \mathbf{C}\setminus \mathbf{R}\) and let \(G=G(z), G_i=G(z_i)\). Then, for any fixed deterministic vectors \(\vx,\vy\) and matrix \(A\) with \(\braket{A}=0\) and \(\lVert A\rVert+\lVert {\bm x}\rVert+\lVert {\bm y}\rVert\lesssim 1\), we have
  \begin{equation}
    \label{eq:mainisogon1}
    \abs{\braket{\vx,\un{GAWG}\vy }} =\landauOstd*{ \frac{\rho}{N^{1/2}\eta}},
  \end{equation}
  and
  \begin{equation}\label{eq:bettunder1}
    \abs{\braket{\un{WG_1G_2}A}} +\abs{\braket{\un{WG_1G_2^t }A}} =\landauOstd*{ \frac{\rho_\ast^{1/2}}{N\eta_\ast^{1/2}}\frac{1}{\sqrt{\eta_1\eta_2}}},
  \end{equation}
  \begin{equation}\label{eq:bettunder2}
    \abs{\braket{\un{WG_1G_1G_2}A}}+\abs{\braket{\un{WG_1G_1G_2^t }A}}  =\landauOstd*{ \frac{\rho_\ast^{1/2}}{N\eta_\ast^{1/2}}\frac{1}{\sqrt{\eta_1\eta_2}}\frac{1}{\sqrt{\eta_*\eta_1}}},
  \end{equation}
  \begin{equation}\label{eq:bettunder3}
    \abs{\braket{\un{WG_1G_2AG_1 }A}}+\abs{\braket{\un{WG_1G_2^t AG_1 }A}}  =\landauOstd*{ \frac{\rho_\ast^{1/2}}{N^{1/2}\eta_\ast^{1/2}}\frac{1}{\sqrt{\eta_1\eta_2}}}.
  \end{equation}
\end{lemma}

Notice that the bound in~\eqref{eq:mainisogon1} is better by a factor $\sqrt{\rho/N\eta}$ compared to~\eqref{eq:underIso}. The bounds~\eqref{eq:bettunder1}--\eqref{eq:bettunder3} improve upon~\eqref{eq:under} in two aspects: First, they depend on \(\rho_\ast\) rather than \(\rho^\ast\), and second, in the cases including transposes the bounds distinguish different \(\eta\)'s (note that in Theorem~\ref{general chain G underline theorem} it is not allowed to have both \(G,G^t\) within one \(\cG\)-block, hence the bound is purely in terms \(\eta_\ast\)). 

We conclude this section with the proof of Propositions~\ref{pro:2g}--\ref{pro:3g}.

\begin{proof}[Proof of Proposition~\ref{pro:2g}]
  
  The local laws for \(\braket{G_1AG_2A'}\), \(\braket{G_1AG_2^t A'}\),  in~\eqref{eq:avll2G1},~\eqref{eq:avll2G22}, respectively, and the bound for \(G_1AG_2\) in~\eqref{eq:isll2G} follow by~\cite[Proposition~\ref{eth-pro:lambll}]{2012.13215} together with~\cite[Theorem~\ref{eth-theo:addchan}]{2012.13215}. The bounds in~\eqref{eq:avll2G4} follow by exactly the same proof of~\cite[Eq.~\eqref{eth-eq:boundimGGA}]{2012.13215}, but using the new bound~\eqref{eq:bettunder1} instead of~\cite[Eq.~\eqref{eth-eq:specunder}]{2012.13215} for the underlined term. Also the local law for \(\braket{G_1G_2^t}\) with error term \(\rho^*K^{-1/2}\) for \(|\sigma|<1\) in~\eqref{eq:avll2G2} follows by~\cite[Theorem~\ref{eth-theo:addchan}, Proposition~\ref{eth-pro:lambll}]{2012.13215}. Hence, in order to conclude the proof of Proposition~\ref{pro:2g} we are left with the averaged and isotropic law for \(G_1G_2\) in~\eqref{eq:avll2G1} and~\eqref{eq:isll2G}, respectively, and with the proof of the remaining cases for the local law for \(\braket{G_1G_2^t}\) in~\eqref{eq:avll2G2}.

  We first consider the local laws that involve no transposes, then at the end of the proof of Proposition~\ref{pro:2g} we explain the necessary changes when the transposes are considered.
  
  By the self consistent equation for \(m\) in~\eqref{eq:mde}, and by \(G(W-z)=I\), we have 
  \begin{equation}
    \label{eq:Gnounder}
    G=m-mWG - m \braket{G} G + m\braket{G-m}G.
  \end{equation}
  As a special case of~\eqref{eq:defunder} we have that
  \begin{equation}\label{eq:newunder}
    \underline{WG}=WG+\braket{G}G  + \sigma \frac{G^t G}{N} + \frac{\widetilde{w_2}}{N} \diag(G)G,
  \end{equation}
  where for any matrix \(R\) in this section   we let \(\diag(R)\) denote the matrix of its diagonal that was denoted by \(R_{\mathrm{d}}\) earlier.
  We recall the parameters \(\sigma\), \(\widetilde{w_2}\) from~\eqref{eq kappa sigma w2 defs}.
  Then by~\eqref{eq:Gnounder} and and~\eqref{eq:newunder}, it follows that
  \begin{equation}\label{eq:G}
    G=m-m\underline{WG}+\frac{m\sigma}{N}G^t G+\frac{\widetilde{w_2}}{N}\mathrm{diag}(G)G+m\braket{G-m}G.
  \end{equation}
  
  We now start writing the equation for generic products of two resolvents \(G_1B_1G_2B_2\), where \(G_i=(W-z_i)^{-1}\) and \(B_1,B_2\) are deterministic matrices. Using the equation~\eqref{eq:G} for \(G_1B_1\) and writing \(G_2=m_2+(G_2-m_2)\), we obtain
  \begin{equation}\label{eq:2G}
    \begin{split}
      G_1B_1G_2B_2&=m_1m_2 B_1B_2+m_1B_1(G_2-m_2)B_2-m_1\underline{WG_1B_1G_2}B_2+ m_1\braket{G_1B_1G_2}G_2B_2 \\
      &\quad+m_1\braket{G_1-m_1}G_1B_1G_2B_2+\frac{m_1\sigma}{N}G_1^t G_1B_1G_2B_2+\frac{m_1\sigma}{N} (G_1B_1G_2)^t G_2B_1 \\
      &\quad+ \frac{m_1\widetilde{w_2}}{N}\mathrm{diag}(G_1)G_1B_1G_2B_2+\frac{m_1\widetilde{w_2}}{N}\mathrm{diag}(G_1B_1G_2)G_2B_2,
    \end{split}
  \end{equation}
  where we used that
  \begin{equation}\label{eq:doubunder}
    \underline{WG_1B_1G_2}=\underline{WG_1}B_1G_2+\braket{G_1B_1G_2}G_2+\frac{\sigma}{N}(G_1B_1G_2)^t G_2+
    \frac{\widetilde{w_2}}{N}\mathrm{diag}(G_1B_1G_2)G_2,
  \end{equation}
  with \(\underline{WG_1}\) from~\eqref{eq:newunder}. The identity in~\eqref{eq:doubunder} follows by the definition of underline in~\eqref{eq:defunder}.

  \begin{proof}[Proof of the local law for \(\braket{G_1G_2}\)]
    
    We divide the proof of this local law into two cases: (i) \(\Im z_1 \Im z_2<0\), (ii) \(\Im z_1\Im z_2>0\). The difference in these two cases is that in (ii) the stability factor \(1-m_1m_2\) is bounded from below by \(\rho^*\), whilst in case (i) the stability factor is bounded from below only by \(\eta^*\) and so it is not affordable to invert it.
    
    We start with \(\Im z_1\Im z_2<0\), in this case we can use resolvent identity and the local law \(|\braket{G_i-m_i}|\prec (N\eta_i)^{-1}\) from~\eqref{eq:avll1G}:
    \begin{equation}\label{eq:res12}
      \begin{split}
        \braket{G_1G_2}=\frac{G_1-G_2}{z_1-z_2}&=\frac{m_1-m_2}{z_1-z_2}+\landauOprec*{\frac{1}{N\eta_*|z_1-z_2|}} \\
        &=\frac{m_1m_2}{1-m_1m_2}+\landauOprec*{\frac{1}{N\eta_*|z_1-z_2|}},
      \end{split}
    \end{equation}
    where we used that the self consistent equation~\eqref{eq:mde} for \(m_1\), \(m_2\) in the third equality. This concludes the proof of the local law for \(\braket{G_1G_2}\) when \(\Im z_1\Im z_2<0\).
    
    We now consider the case \(\Im z_1\Im z_2>0\). Choosing \(B_1=B_2=I\) in~\eqref{eq:2G}, and using \(|\braket{G_i-m_i}|\prec (N\eta_i)^{-1}\), we find that
    \begin{equation}\label{eq:12ggi}
      \begin{split}
        \left[1-m_1m_2+\landauOprec*{\frac{1}{N\eta_*}}\right]\braket{G_1G_2}&=m_1m_2-m_1\braket{\underline{WG_1G_2}}+m_1\braket{(G_2-m_2)} \\
        &\quad+\frac{m_1\sigma}{N}\braket{G_1^t G_1G_2}+\frac{m_1\sigma}{N}\braket{(G_1G_2)^t G_2}\\
        &\quad+\frac{m_1\widetilde{w_2}}{N}\braket{\mathrm{diag}(G_1)G_1G_2} +\frac{m_1\widetilde{w_2}}{N}\braket{\mathrm{diag}(G_1G_2)G_2}.
      \end{split}
    \end{equation}
    
    Using a Schwarz inequality we readily conclude that
    \begin{equation}\label{eq:1scw}
      \frac{1}{N}|\braket{G_1^t G_1G_2}|\le \frac{1}{N} \braket{G_1G_1^*}^{1/2}\braket{G_1G_2G_2^*G_1^*}^{1/2}\prec \frac{\rho_1}{N\eta_1\eta_2},
    \end{equation}
    where we used that \(\braket{\Im G_i}\prec \rho_i\), that \(\eta_1 G_1G_1^*=\Im G_1\) by Ward identity, that \(|\braket{G_1G_2G_2^*G_1^*}|\le \lVert G_2G_2^*\rVert\braket{G_1G_1^*}\), and that \(\lVert G_i\rVert\le \eta_i^{-1}\). We also prove that \(|\braket{(G_1G_2)^t G_2}|\prec \rho_2(N\eta_1\eta_2)^{-1}\) using exactly the same computations. Additionally, we get that
    \begin{equation}\label{eq:2scw}
      \frac{1}{N}|\braket{\mathrm{diag}(G_1)G_1G_2}|=\left|\frac{1}{N^2}\sum_i (G_1)_{ii}(G_1G_2)_{ii}\right| \prec\frac{\sqrt{\rho_1\rho_2}}{N\sqrt{\eta_1\eta_2}}
    \end{equation}
    where we used that \(|G_{ii}|\prec 1\), and that \(|(G_1G_2)_{ii}|\prec \sqrt{\rho_1\rho_2/(\eta_1\eta_2)}\) by a Schwarz inequality and Ward identity. The bound for \(|\braket{\mathrm{diag}(G_1G_2)G_2}|\) is completely analogous and so omitted. Combining~\eqref{eq:12ggi} with~\eqref{eq:1scw}--\eqref{eq:2scw} and using that \(|\braket{G_2-m_2}|\prec (N\eta_2)^{-1}\), we finally conclude that
    \[
    \braket{G_1G_2}=\frac{m_1m_2}{1-m_1m_2}-\frac{m_1}{1-m_1m_2}\braket{\underline{WG_1G_2}} +\landauOprec*{\frac{1}{N\eta_1\eta_2}+\frac{1}{N\eta_*(\rho^*)^2}},
    \]
    where we used that by easy computations we have \(|1-m_1m_2|\ge \rho^*\) and that \(\rho^*\gg \frac{1}{N\eta_*}\), by \(K=N\eta_*\rho^*\gg 1\), to divide through the multiplicative factor in the lhs.\ of~\eqref{eq:12ggi}. Finally, using that \(|\braket{\underline{WG_1G_2}}|\prec \rho^*(N\eta_1\eta_2)^{-1}\) by~Theorem~\ref{general chain G underline theorem}, \(|1-m_1m_2|\ge \rho^*\) once again, and that \((\rho^*)^2\gtrsim\eta^*\), \(\eta_*\eta^*=\eta_1\eta_2\), we conclude that
    \begin{equation}\label{eq:2gllaw}
      \braket{G_1G_2}=\frac{m_1m_2}{1-m_1m_2} +\landauOprec*{\frac{1}{N\eta_1\eta_2}}.
    \end{equation}
  \end{proof}
  
  \begin{proof}[Proof of the local law for \(\braket{{\bm x}, G_1G_2{\bm y}}\)]
    The proof of the isotropic law for \(G_1G_2\) is very similar to the proof of the averaged law above, hence we explain only the minor differences. Similarly to the averaged local law, the case \(\Im z_1\Im z_2<0\) trivially follows by resolvent identity. In the opposite case, choosing \(B_1=B_2=I\) in~\eqref{eq:2G}, and that \(|\braket{G_i-m_i}|\prec (N\eta_i)^{-1}\), we find that
    \begin{equation}
    \label{2G iso}
      \begin{split}
        \left[1+\landauOprec*{\frac{1}{N\eta_*}} \right]\braket{{\bm x}, G_1G_2 {\bm y}}&=m_1m_2\braket{{\bm x}, {\bm y}}-m_1 \braket{{\bm x}, \underline{WG_1G_2}{\bm y}}+m_1\braket{G_1G_2}\braket{{\bm x}, G_2 {\bm y}} \\
        &\quad+\landauO*{\frac{\rho^*}{N\eta_1\eta_2}},
      \end{split}
    \end{equation}
    where we used that the terms with a pre-factor \(\sigma\) or \(w_2-1-\sigma\) can be estimated by \(N^{-1}\rho^*(\eta_1\eta_2)^{-1}\) using a Schwarz inequality similarly to~\eqref{eq:1scw}--\eqref{eq:2scw}. Then using that
    \[
    \braket{G_1G_2}=\frac{m_1m_2}{1-m_1m_2}+\landauOprec*{\frac{1}{N\eta_1\eta_2}},
    \]
    by~\eqref{eq:2gllaw}, and that \(|\braket{{\bm x}, \underline{WG_1G_2}{\bm y}}|\prec \sqrt{\rho^*}(N\eta_*)^{-1/2}(\eta^*)^{-1}\) by~Theorem~\ref{general chain G underline theorem}, we finally conclude that
    \[
    \braket{{\bm x}, G_1G_2 {\bm y}}=\frac{m_1m_2}{1-m_1m_2}\braket{{\bm x}, {\bm y}}+\landauO*{\frac{\sqrt{\rho^*}}{\sqrt{N\eta_*}\eta^*}+\frac{1}{N\eta_*\rho^*}}.
    \]
  \end{proof}

  In order to conclude the proof of Proposition~\ref{pro:2g} we are left with considering transposes.
  
  \begin{proof}[Proof of the local law for \(\braket{G_1G_2^t}\)]
    
    The proof of this local law is divided into three cases: (i) \(\sigma=1\), (ii) \(\sigma=-1\), (iii) \(|\sigma|<1\). The main difference compared to the proof of \(\braket{G_1G_2}\) is that the two body stability factor is now given by \(1-\sigma m_1m_2\) instead of \(1-m_1m_2\).

    For \(\sigma=1\) there is nothing else to prove since in this case \(W\) is real symmetric and so \(G_2^t=G_2\).
    
    The proof of the local law for \(|\sigma|<1\) is completely analogous to the proof of~\eqref{eq:2gllaw}, modulo the bound for the underline term that is now given by \(|\braket{\underline{WG_1G_2^t}}|\prec \rho^*(N\eta_*^2)^{-1}\), since the only thing we used in this proof is that the stability factor \(1-m_1m_2\) is bounded from below by \(\rho^*\). This is also the case for \(1-\sigma m_1m_2\) when \(|\sigma|<1\), since \(|1-\sigma m_1m_2|\ge 1-|\sigma|\).
    
    We are now left with the case \(\sigma=-1\), when the stability factor is given by \(1+m_1m_2\). Note that when \(\sigma=-1\) we can write \(W=D+\ii O\) with \(D\) being a diagonal matrix and \(O\) being an skew-symmetric matrix, i.e.\ \(O^t=-O\). If \(D=0\), and either \(\Im z_1\Im z_2>0\) or \(\Im z_1\Im z_2<0\) and \(|z_1+z_2|\ge \eta^*\), using the notation \(R(z_i):=(\ii O-z_i)^{-1}\) and that \(R(z_i)^t=-R(-z_i)\), by resolvent identity we conclude
    \begin{equation}
    \label{eq:resir}
    \begin{split}
      \braket{G(z_1)G(z_2)^t}=-\braket{R(z_1)R(-z_2)}&=-\frac{\braket{R(z_1)}-\braket{R(-z_2)}}{z_1+z_2}\\ 
      &=\frac{m_1m_2}{1+m_1m_2}+\landauOprec*{\frac{1}{N\eta_*|z_1+z_2|}}.
      \end{split}
    \end{equation}
    where we used that \(m(-z_i)=-m(z_i)\), and the local law for \(R_i\), that holds even for Wigner matrices with zero diagonal. For \(\Im z_1\Im z_2<0\) and \(|z_1+z_2|<\eta^*\), using that \(R(z_2)^t=-R(-z_2)\) we proceed exactly as in the proof of the local law for \(\braket{G_1G_2}\) above. This gives the local law for \(\braket{G_1G_2^t}\) in~\eqref{eq:avll2G2}. In order to conclude the proof we are now left only with the case \(D\ne 0\). In this case we use the following lemma whose proof is postponed to Appendix~\ref{sec:addres}.
    
    
    \begin{lemma}\label{lem:remtra}
      Fix \(\epsilon>0\). Let \(W=D+\ii O\) be a Wigner matrix with \(D\) being diagonal and \(O\) skew-symmetric. Denote \(G_i=(W-z_i)^{-1}\) and \(R_i=(\ii O-z_i)^{-1}\), with \(z_1\), \(z_2\in\mathbf{C}\setminus\mathbf{R}\) such that \(\eta_*:=|\Im z_1|\wedge|\Im z_2|\ge N^{-1+\epsilon}\) and \(\eta^*:=|\Im z_1|\vee|\Im z_2|\), then for \(\sigma=-1\) it holds
      \begin{equation}\label{eq:remdiag}
        \braket{G_1G_2^t}=\braket{R_1R_2^t}+\landauOprec*{\frac{1}{N\eta_*}\left[\frac{1}{|z_1+z_2|}\wedge \frac{1}{\eta^*}\right]}.
      \end{equation}
      Moreover, we also have
      \begin{equation}\label{eq:remdiag3g}
        \braket{G_1^t G_2^2}=\braket{R_1^t R_2^2}+\landauOprec*{\frac{1}{N\eta_*|\Im z_2|}\left[\frac{1}{|z_1+z_2|}\wedge \frac{1}{\eta^*}\right]}.
      \end{equation}
    \end{lemma}
    
    In Lemma~\ref{lem:remtra} we stated the result for three \(G\)'s as well~\eqref{eq:remdiag3g} even if not needed for the proof of the local law of \(\braket{G_1G_2^t}\). We will use~\eqref{eq:remdiag3g} later in~\eqref{eq:1}.
    
    
    Finally, combining~\eqref{eq:remdiag} with~\eqref{eq:resir}, we conclude the proof of the local law for \(\braket{G_1G_2^t}\).
  \end{proof}
  
  This concludes the proof of Proposition~\ref{pro:2g}, modulo the proof of Lemma~\ref{lem:remtra}, which is postponed to Appendix~\ref{sec:addres}.
  
\end{proof}

We conclude this section with the proof of the local laws for certain products of three resolvents. We will prove the estimates without transposed resolvents, the analogous results with transposes are proven in Appendix~\ref{addproofs34}.

\begin{proof}[Proof of Proposition~\ref{pro:3g}]
  We start writing the equation for general products of three different resolvents \(G_1,G_2,G_3\) and deterministic matrices \(B_1,B_2, B_3\):
  \begin{equation}\label{eq:3gs}
    \begin{split}
      G_1B_1G_2B_2G_3B_3&=m_1B_1G_2B_2G_3B_3-m_1\underline{WG_1B_1G_2B_2G_3}B_3+m_1\braket{G_1B_1G_2}G_2B_2G_3B_3 \\
      &\quad+m_1\braket{G_1-m_1}G_1B_1G_2B_2G_3B_3 +m_1\braket{G_1B_1G_2B_2G_3}G_3B_3 \\
      &\quad+\frac{m_1\sigma}{N} G_1^t G_1B_1G_2B_2G_3B_3+\frac{m_1\sigma}{N}(G_1B_1G_2)^t G_2B_2G_3B_3 \\
      &\quad+\frac{m_1\sigma}{N}(G_1B_1G_2B_2G_3)^t G_3B_3+\frac{m_1\widetilde{w_2}}{N} \mathrm{diag}(G_1)G_1B_1G_2B_2G_3B_3 \\
      &\quad+\frac{m_1\widetilde{w_2}}{N} \mathrm{diag}(G_1B_1G_2)G_2B_2G_3B_3 +\frac{m_1\widetilde{w_2}}{N} \mathrm{diag}(G_1B_1G_2B_2G_3)G_3B_3,
    \end{split}
  \end{equation}
  where we used that
  \begin{equation}\label{eq:tripunder}
    \begin{split}
      \underline{WG_1B_1G_2B_2G_3}&=\underline{WG_1B_1G_2}B_2G_3+\braket{G_1B_1G_2B_2G_3}G_3+\frac{\sigma}{N}(G_1B_1G_2B_2G_3)^t G_3 \\
      &\quad+\frac{\widetilde{w_2}}{N}\mathrm{diag}(G_1B_1G_2B_2G_3)G_3,
    \end{split}
  \end{equation}
  with \(\underline{WG_1B_1G_2}\) from~\eqref{eq:doubunder}. The identity in~\eqref{eq:tripunder} follows by the definition of the renormalization (denoted by underline) in~\eqref{eq:defunder}.
  
  

  \begin{proof}[Proof of the local law for \(\braket{G_1G_2^2}\) in~\eqref{eq:avll3G1}]
    
    Similarly to the proof of the local law for \(\braket{G_1G_2}\), the proof of the local law for \(\braket{G_1G_2^2}\) is divided into two cases: (i) \(\Im z_1\Im z_2<0\) or \(\Im z_1\Im z_2>0\) and \(|z_1-z_2|\ge \eta^*\)  (ii) \(\Im z_1\Im z_2>0\) and \(|z_1-z_2|<\eta^*\), where we recall that $\eta^*:=\eta_1\vee \eta_2$, $\eta_i:=|\Im z_i|$.
    
    Similarly to~\eqref{eq:res12}, if either \(\Im z_1\Im z_2<0\) or \(\Im z_1\Im z_2>0\) and \(|z_1-z_2|\gtrsim \eta^*\) we use the resolvent identity twice to get
    \begin{equation}\label{eq:resid3g}
      \begin{split}
        \braket{G_1G_2^2}&=\frac{\braket{G_2^2}-\braket{G_1G_2}}{z_2-z_1}=\frac{m_2'}{z_2-z_1}+\frac{\braket{G_1}-\braket{G_2}}{(z_1-z_2)^2}+\landauOprec*{\frac{1}{N\eta_2^2|z_1-z_2|}} \\
        &=\frac{m_2'}{z_2-z_1}+\frac{m_1-m_2}{(z_1-z_2)^2}+\landauOprec*{\frac{1}{N\eta_2^2|z_1-z_2|}+\frac{1}{N\eta_*|z_1-z_2|^2}} \\
        &=\frac{m_1m_2'}{(1-m_1m_2)^2}+\landauOprec*{\frac{\rho_*}{L\eta_1\eta_2}},
      \end{split}
    \end{equation}
    where in the first line we used the local law for \(\braket{G_2^2}\) in~\eqref{eq:avll2G1}, and the identity \(m_2'=m_2^2(1-m_2^2)^{-1}\), and to go from the second to the third line we again used the equation of \(m_1\), \(m_2\). We remark that to estimate the error terms to go from the second to the third line we also used~\eqref{eq:perch} below. The proof of this bound is postponed to Appendix~\ref{sec:addres}.
    
    \begin{lemma}\label{lem:impbneed}
      Let \(z_1, z_2 \in\mathbf{C}\setminus\mathbf{R}\) such that \(|z_1-z_2|\gtrsim \eta^*\), then it holds
      \begin{equation}\label{eq:perch}
        \frac{1}{N\eta_*^2|z_1-z_2|}\lesssim \frac{\rho_*}{L\eta_1\eta_2},
      \end{equation} 
      where \(\eta_i=|\Im z_i|\), \(\eta_*=\eta_1\wedge\eta_2\), \(\rho_*=\rho_1\wedge \rho_2\), \(\eta^*=\eta_1\vee \eta_2\).
    \end{lemma}
    
    Next we consider the last remaining case \(\Im z_1\Im z_2>0\) and \(|z_1-z_2|< \eta^*\). By~\eqref{eq:3gs} with \(B_1=B_2=B_3=I\) and \(G_3=G_2\), we have that
    \begin{equation}\label{eq:Gimpboh}
      \begin{split}
        (1-m_1m_2)\braket{G_1G_2^2}&=m_1\braket{G_2^2}-m_1\braket{\underline{WG_1G_2G_2}}+m_1[\braket{G_1-m_1}+\braket{G_2-m_2}]\braket{G_1G_2^2} \\
        &\quad+m_1\braket{G_1G_2}\braket{G_2^2}+\landauOprec*{\frac{\rho^*}{N\eta_2^2\eta_1}} \\
        &=\frac{m_1m_2'}{1-m_1m_2}+\landauOprec*{\frac{1}{N\eta_2^2\eta_1}},
      \end{split}
    \end{equation}
    where to go from the second to the third line we used the local laws for \(\braket{G_1G_2}\) and \(\braket{G_2^2}\) in~\eqref{eq:avll2G1}, that \(|\braket{G_i-m_i}|\prec (N\eta_i)^{-1}\) by the first local law in~\eqref{eq:avll1G}, and that \(|\braket{\underline{WG_1G_2^2}}|\prec \rho^*(N\eta_2\eta_1)^{-1}\) by~Theorem~\ref{general chain G underline theorem}. We remark that to go from~\eqref{eq:3gs} to~\eqref{eq:Gimpboh} we used that all the terms with a pre-factor \(N^{-1}\) in~\eqref{eq:3gs} are bounded by \(\rho^*(N\eta_2^2\eta_1)^{-1}\). To make this clearer we show this bound for two representative terms:
    \[
    \begin{split}
      \frac{1}{N}|\braket{G_1^t G_1G_2G_2}|&\le \frac{1}{N}\braket{G_1G_1^*}^{1/2}\braket{G_1G_2G_2G_2^*G_2^*G_1^*}^{1/2}\prec \frac{\rho^*}{N\eta_1\eta_2^2},  \\
      \frac{1}{N}|\braket{\mathrm{diag}(G_1)G_1G_2G_2}|&=\left|\frac{1}{N}\sum_i (G_1)_{ii} (G_1G_2G_2)_{ii}\right|\prec\frac{\rho^*}{N\eta_1^{1/2}\eta_2^{3/2}},
    \end{split}
    \]
    where we used a Schwarz inequality, the norm bound \(\lVert G_2G_2G_2^*G_2^*\rVert\le \eta_2^{-4}\), and that \(|(G_1)_{ii}|\prec 1\), \(|(G_1G_2G_2)_{ii}|\prec \rho^* \eta_1^{-1/2}\eta_2^{-3/2}\).
    
    
    Note that the error term in the rhs.\ of~\eqref{eq:Gimpboh} is smaller than our goal \(\rho_*(L\eta_1\eta_2)^{-1}\) in~\eqref{eq:avll3G1}, since for \(|z_1-z_2|< \eta^*\) we have that
    \begin{equation}\label{eq:impneeddr}
      \frac{1}{N\eta_2^2\eta_1}\lesssim \frac{\rho_*}{L\eta_1\eta_2}.
    \end{equation}
    This concludes the proof of the local law for \(\braket{G^2G'}\).
  \end{proof}

  \begin{proof}[Proof of the local law for \(\braket{G_1G_2AG_1A'}\) in~\eqref{eq:avll3G11}]
    Consider the equation in~\eqref{eq:3gs} for \(G_3=G_1\), and \(B_1=I\), \(B_2=A\), \(B_3=A'\), with \(\braket{A}=\braket{A'}=0\). Before proceeding with writing the equation for \(\braket{G_1G_2AG_1A'}\), we bound two representative terms with a pre-factor \(N^{-1}\) in~\eqref{eq:3gs}:
    \begin{equation}\label{eq:preb}
      \begin{split} 
        |\braket{G_1^t G_1G_2AG_1A'}|&\le \braket{G_1G_1^*}^{1/2}\braket{G_2AG_1A'G_1^t(G_1^t)^*A'G_1^*AG_2^*}^{1/2} \\
        &\quad\prec \sqrt{\frac{\rho_1}{\eta_1}}\frac{1}{\sqrt{\eta_1\eta_2}}\braket{\Im G_2AG_1A'\Im G_1^t A'G_1^*A}^{1/2}\prec \frac{\sqrt{N}\rho_1\sqrt{\rho_2}}{\sqrt{\eta_2}\eta_1}, \\
        |\braket{\mathrm{diag}(G_1)G_1G_2AG_1A'}|&\prec \frac{1}{N}\sum_a |(G_1)_{aa}(G_1G_2AG_1A')_{aa}|\prec \sqrt{\frac{N\rho_1\rho_2}{\eta_1\eta_2}},
      \end{split}
    \end{equation}
    where in the first estimate we used~\cite[Lemma~\ref{eth-degree two lemma}]{2012.13215} to bound
    \[ 
    |\braket{\Im G_2AG_1A'\Im G_1 A'G_1^*A}|\prec N\rho_1\rho_2,
    \]
    and in the second estimate we used that
    \[
    |\braket{{\bm x}, G_1G_2AG_1 {\bm y}}|\le \braket{{\bm x}, G_1G_1^*{\bm x}}^{1/2}\braket{{\bm y}, G_1^*AG_2^*G_2AG_1{\bm y}}^{1/2}\prec \sqrt{\frac{N\rho_1\rho_2}{\eta_1\eta_2}},
    \]
    by~\cite[Lemma~\ref{eth-degree two lemma}]{2012.13215} again. The bound of all the other terms with a pre-factor \(N^{-1}\) is analogous and so omitted. Then, by~\eqref{eq:preb} and~\eqref{eq:3gs}, we conclude that
    \begin{equation}
      \begin{split}
        \left[1+\landauOprec*{\frac{1}{N\eta_*}}\right]\braket{G_1G_2AG_1A'}&=m_1\braket{G_2AG_1A'}-m_1\braket{\underline{WG_1G_2AG_1}A'}\\
        &\quad+m_1\braket{G_1G_2}\braket{G_2AG_1A'} \\
        &\quad+m_1\braket{G_1G_2AG_1}\braket{G_1A'}+\landauOprec*{\frac{\sqrt{\rho_1\rho_2}}{\sqrt{L\eta_1\eta_2}}} \\
        &=\frac{m_1^2m_2}{1-m_1m_2}\braket{AA'}-m_1\braket{\underline{WG_1G_2AG_1}A'}\\
        &\quad+\landauOprec*{\frac{\sqrt{\rho_1\rho_2}}{\sqrt{L\eta_1\eta_2}}},
      \end{split}
    \end{equation}
    where we used that
    \begin{equation}
      \label{eq:needllaa}
      \braket{G_2AG_1A'}=m_1m_2\braket{AA'}+\landauOprec*{\frac{\rho^*}{\sqrt{L}}}, \quad \braket{G_1G_2}=\frac{m_1m_2}{1-m_1m_2}+\landauOprec*{\frac{1}{N\eta_1\eta_2}}, 
    \end{equation}
    and
    \begin{equation}
      \label{eq:needllaa2}
      |\braket{G_1A'}|\prec \frac{\sqrt{\rho_1}}{N\sqrt{\eta_1}}.
    \end{equation}
    The local laws in~\eqref{eq:needllaa} follow by~\eqref{eq:avll2G1}, whilst the bound in~\eqref{eq:needllaa2} follows by~\eqref{eq:avll1G}.
    
    
    Finally, using the bound \(|\braket{\underline{WG_1G_2AG_1}A}|\prec \rho^* K^{-1/2}(\eta^*)^{-1}\) from~\cite[Theorem~\ref{eth-chain G underline theorem}]{2012.13215}, we conclude that
    \begin{equation}
      \braket{G_1G_2AG_1A'}=\frac{m_1^2m_2}{1-m_1m_2}\braket{AA'}+\landauOprec*{\frac{\rho_*}{\sqrt{L\eta_1\eta_2}}},
    \end{equation}
    where we recall that \(K=N\eta_*\rho^*\), and \(L=N(\rho_1\eta_1\wedge \rho_2\eta_2)\). This concludes the proof of the local law~\eqref{eq:avll3G11} for \(\braket{G_1G_2AG_1A'}\).
  \end{proof}
  
  \begin{proof}[Proof of the bound for \(\braket{G_1G_2G_2A}\) in~\eqref{eq:avll3G3}]
    Consider the equation in~\eqref{eq:3gs} for \(G_3=G_2\) and \(B_1=B_2=I\), \(B_3=A\), with \(\braket{A}=0\). Proceeding similarly to~\eqref{eq:preb} to estimate the error terms with a pre-factor \(N^{-1}\), we conclude that
    \begin{equation}\label{eq:imp3ga}
      \begin{split}
        \braket{G_1G_2G_2A}&=m_1\braket{G_2G_2A}-m_1\braket{\underline{WG_1G_2G_2}A}+m_1\braket{G_1-m_1}\braket{G_1G_2G_2A} \\
        &\quad+m_1\braket{G_1G_2}\braket{G_2G_2A}+m_1\braket{G_1G_2G_2}\braket{G_2A}+\landauOprec*{\frac{\sqrt{\rho_*}}{L\sqrt{\eta_1}\eta_2}} \\
        &=\landauOstd*{\frac{\sqrt{\rho_*}}{L\sqrt{\eta_1}\eta_2}+\frac{\rho^*\sqrt{\rho_*}}{K\sqrt{\eta_1}\eta_2}} \\
        &=\landauOstd*{\frac{\sqrt{\rho_*}}{L\sqrt{\eta_1}\eta_2}},
      \end{split}
    \end{equation}
    where to go from the second to the third line we used that
    \[
    |\braket{\underline{WG_1G_2G_2}A}|=\landauOstd*{\frac{\rho^*\sqrt{\rho^*}}{K\sqrt{\eta_1}\eta_2}}
    \]
    by~\eqref{eq:bettunder2}, and that
    \begin{equation}\label{eq:finb3g}
      |\braket{G_2G_2A}|=\landauOstd*{\frac{\sqrt{\rho_2}}{N\eta_2^{3/2}}}, \quad |\braket{G_1G_2}|\prec \sqrt{\frac{\rho_1\rho_2}{\eta_1\eta_2}}, \quad |\braket{G_1G_2G_2}|\prec \frac{\sqrt{\rho_1\rho_2}}{\sqrt{\eta_1}\eta_2^{3/2}}.
    \end{equation}
    The first bound in~\eqref{eq:finb3g} follows by~\eqref{eq:avll2G4}, whilst the second and the third bound follow by a simple Schwarz inequality. This concludes the proof of the bound for \(\braket{G_1G_2G_2A}\).
  \end{proof}
  
  The remaining statements of Proposition~\ref{pro:3g} involving transposed resolvents are proved similarly. For completeness we included those proofs in Appendix~\ref{addproofs34}. This concludes the proof of Proposition~\ref{pro:3g}.
\end{proof}


%

\section{CLT for resolvents}\label{sec:cltres}
We now formulate the resolvent CLT identifying the joint distribution of \(\braket{G-\E G},\braket{(G-\E G)A}\) for multiple \(z\)'s and traceless \(A\)'s. An analogous result for only \(\braket{G-\E G}\) factors was proven in~\cite{MR3678478}. Let \(p\le q\in\N\) and let \(A_1,\ldots,A_p\) be matrices with \(\braket{A_i}=0,\norm{A_i}\lesssim1\) and let
\(\bm a_i\) denote the vector of the diagonal elements of \(A_i\). Let
\(z_1,\ldots,z_q\in\C\setminus \R\) be spectral parameters. We then set 
\[ G_i:=G(z_i), \qquad m_i:=m(z_i), \qquad  \rho_i:=\frac{1}{\pi}\abs{\Im m_i}, \qquad \eta_i:=\abs{\Im z_i},
\] 
\begin{equation} \label{Xi Yi def}
  X_i:=\braket{[G_i-\E G_i]A_i}, \quad Y_i:=\braket{G_i-\E G_i}, \quad X_S:=\prod_{i\in S}X_i,\quad Y_S:=\prod_{i\in S}Y_i,
\end{equation}
so that from the local laws in~\eqref{eq:avll1G} we have  the \emph{a priori } bounds. 
\[ \abs{X_S}\prec \Psi_S, \quad \abs{Y_S}\prec \Psi_S, \quad \Psi_S:=\prod_{i\in S}\Psi_i, \quad \Psi_i :=\frac{\rho_i^{1/2}}{N\eta_i^{1/2}}\bm1(i\le p) + \frac{1}{N\eta_i}\bm1(i> p). \]
The following theorem 
identifies the leading terms of the joint moments of \(X\)'s and \(Y\)'s up to an error term that is smaller 
than the  \emph{a priori } bounds.

\begin{theorem}\label{CLT theorem}
  For any \(\epsilon>0\) and \(1\le L := \min_i N\eta_i\rho_i \) we have that 
  \begin{equation}\label{eq CLT statement}
    \begin{split}
      &\E X_{[p]}Y_{(p,q]} = \frac{1}{N^q}\sum_{\substack{P\in\mathrm{Pair}([p])\\Q\in\mathrm{Pair}((p,q])}} \prod_{\set{i,j}\in P} V_{ij}^\circ(A_i,A_j)\prod_{\set{i,j}\in Q}  V_{ij} + \landauO*{\frac{N^\epsilon\Psi}{\sqrt{L}}},
    \end{split}
  \end{equation}
  where \(\Psi:=\Psi_{[q]}\),
  \begin{equation}
    \label{eq:variancesres}
    \begin{split}
      V_{ij}^\circ(A_i,A_j) &:= \frac{m_i^2 m_j^2\braket{A_i A_j}}{1-m_i m_j}+\frac{\sigma m_i^2 m_j^2\braket{A_i A_j^t}}{1-\sigma m_i m_j}+\big[\kappa_4 m_i^3 m_j^3+\widetilde{w_2}m_i^2 m_j^2\big]\braket{\bm a_i\bm a_j}\\
      V_{ij}&:=\frac{m_i' m_j'}{(1-m_i m_j)^2}+\frac{\sigma m_i' m_j'}{(1-\sigma m_i m_j)^2} +  \frac{\kappa_4}{2} (m_i^2)'  (m_j^2)' + \widetilde{w_2}m_1'm_i'
    \end{split}
  \end{equation}
  and \(\mathrm{Pair}(S)\) denotes the set of pairings of a base set \(S\). Moreover, the expectation \(\E G\) is given by 
  \begin{equation}\label{claim exp}
    \begin{split}
      \braket{\E G_i} &= m_i + \frac{1}{N} \Bigl(\frac{m_i'}{m_i}\frac{\sigma m_i^2}{1-\sigma m_i^2} 
      + \widetilde{w_2} m_i' m_i + \kappa_4 m_i' m_i^3\Bigr) + \landauO*{\frac{N^\epsilon \Psi_i}{L^{1/2}}}\\
      \braket{\E G_i A_i} &= \landauO*{\frac{N^\epsilon \Psi_i}{L^{1/2}}}.
    \end{split}
  \end{equation}
\end{theorem}
Note that the first (leading) term in~\eqref{eq CLT statement} has a natural size of oder \(\Psi\) whenever \((\Im z_i)(\Im z_j)<0\) for every \(\set{i,j}\) in the pairings \(P,Q\). 

Within the proof of Theorem~\ref{CLT theorem} we use the classical cumulant expansion in the form
\begin{equation}\label{eq:cumexp}
  \E w_{ab} f(W) = \sum_{k=1}^{R-1} \sum_{\bm\alpha\in\set{ab,ba}^k}\frac{\kappa(ab,\bm \alpha)}{k!} \E \partial_{\bm\alpha}f(W) + \Omega_R,
\end{equation}
where \(\kappa(ab,\bm\alpha)\) denotes the joint cumulant of \(w_{ab},w_{\alpha_1},\ldots,w_{\alpha_k}\) for \(\bm\alpha=(\alpha_1,\ldots,\alpha_k)\). Here for any cut-off index \(R\) the error term \(\Omega_R\) has an explicit integral representation~\cite[Proposition 3.2]{MR3941370}. For our application, where \(f(W)\) is a product of resolvents, the error term can easily be estimated by \(\abs{\Omega_R}\lesssim N^{-(R+1)/2}\). This is due to the fact that the \(k\)-th cumulant scales like \(N^{-k/2}\), and each derivative creates an additional resolvent entry which can be estimated by \(\mathcal{O}(1)\) due to the single resolvent local law. In the sequel we will omit the cutoff $R$ from the formulas and we simply write a cumulant expansion with a formally
infinite sum over $k$, but technically we always estimate the truncated sum. 

\begin{proof}[Proof of Theorem~\ref{CLT theorem}]
  Recalling from~\eqref{eq:G} that
  \begin{equation}\label{G expansion eq}
    G = m - m\un{WG} + m\braket{G-m}G + m\sigma \frac{G^t G}{N} + m\frac{\widetilde{w_2}}{N}\diag(G)G,
  \end{equation}
  it follows that  (with \(\eta=\abs{\Im z}\), \(\rho=\rho(z)\))
  \begin{equation}\label{G-m prec exp}
    \begin{split}
      (1-m^2)\braket{G-m} &= -m\braket{\un{WG}} + m\braket{G-m}^2 + \frac{\sigma m}{N} \braket{G^t G} +  m\frac{\widetilde{w_2}}{N} \braket{\mathrm{diag}(G)G}\\
      &= - m\braket{\un{WG}} + \frac{m}{N}\Bigl( \frac{\sigma m^2}{1-\sigma m^2} + \widetilde{w_2} m^2 \Bigr)  + 
      \landauOprec*{ \frac{\rho}{N\eta L^{1/2}}}
    \end{split}
  \end{equation}
  from~\eqref{eq:avll2G2} and \(\braket{\mathrm{diag}(G)G} = m^2 + 
  \landauOprec{\sqrt{\rho/N \eta}}\) due to~\eqref{eq:isll1G}. 
  Using a cumulant expansion we prove below that 
  \begin{equation}\label{eq exp expansion ge3}
    \E\braket{\un{WG}} = \frac{1}{N}\sum_{k\ge 2}\sum_{ab}\sum_{\bm\alpha\in\set{ab,ba}^k} \frac{\kappa(ab,\bm\alpha)}{k!} \E \partial_{\bm\alpha}G_{ba}= -\kappa_4\frac{m^4}{N} + \landauO*{\frac{\rho}{N^2\eta}+\frac{\rho^{3/2}}{N^{3/2}\eta^{1/2}}},
  \end{equation}
   where we ignored the irrelevant error term \(\Omega_R\). Note that the \(k=1\) summand has been cancelled by the variance term which is included in the definition of the ``underline'' renormalisation. Then, the first claim in~\eqref{claim exp} follows immediately from~\eqref{eq exp expansion ge3} together with~\eqref{G-m prec exp}.

  Similarly, from~\eqref{G expansion eq} we obtain 
  \begin{equation}\label{G-m A prec exp}
    \begin{split}
      (1-m\braket{G-m})\braket{GA} &= - m\braket{\un{WG}A} + \frac{\sigma m}{N}\braket{G^t G A} + m\frac{\widetilde{w_2}}{N} \braket{\diag(G)G A}\\
      &=  - m\braket{\un{WG}A} + \landauOprec*{ \frac{\rho^{1/2}}{N\eta^{1/2}} }
    \end{split}
  \end{equation}
  from~\eqref{eq:avll2G4} and 
  \[ \braket{\diag(G)GA}=m\braket{GA}+\braket{\diag(G-m)GA} =\landauOprec*{\sqrt{\frac{\rho}{N\eta}}}.\]
  For the underlined term in~\eqref{G-m A prec exp} it follows exactly as in~\eqref{eq exp expansion ge3} that 
  \begin{equation}
  \label{EGA}
    \E\braket{\un{WG}A} = \frac{1}{N}\sum_{k\ge 2}\sum_{ab}\sum_{\bm\alpha\in\set{ab,ba}^k} \frac{\kappa(ab,\bm\alpha)}{k!} \E \partial_{\bm\alpha}(GA)_{ba} = \landauO*{\frac{\rho}{N^2\eta}+\frac{\rho^{3/2}}{N^{3/2}\eta^{1/2}}},
  \end{equation}
  where the term corresponding to \(\kappa_4\) vanishes due to \(\braket{A}=0\). 
  Together with~\eqref{G-m A prec exp}, the second claim in~\eqref{claim exp} is also proven. This concludes the computation of the expectation. 
  
  We will now prove an asymptotic Wick theorem and explicitly compute the variance. Using~\eqref{eq:avll1G} and~\eqref{G-m A prec exp}--\eqref{EGA} we replace \(X_1=\braket{[G_1-\E G_1]A_1}\) with its leading term \(\braket{\un{WG_1A_1}}\), i.e.
  \begin{equation}
    \begin{split}
      \braket{[G_1-\E G_1]A_1} &= - m\braket{\un{WG_1}A_1} + m\E \braket{\un{WG_1}A_1} + \landauOprec*{ \frac{\rho}{N^{3/2}\eta} }\\
      &= - m\braket{\un{WG_1}A_1}+\landauOprec*{ \frac{\rho}{N^{3/2}\eta} }.
    \end{split}
  \end{equation}
  Then for
  \begin{equation}
    \begin{split}
      &\E X_{[p]}Y_{(p,q]} = - m_1\E \braket{\un{WG_1A_1}} X_{(1,p]}Y_{(p,q]} + \landauO*{\frac{N^\xi \Psi}{L^{1/2}}}
    \end{split}
  \end{equation}
  we perform a cumulant expansion of \(\braket{\underline{W G_1 A_1}}\) to obtain 
  \begin{equation}\label{CLT big expansion}
    \begin{split}
      \E X_{[p]}Y_{(p,q]} &= \sum_{i\in[p]\setminus\set{1}} m_1 \E\wt\E \braket{\wt W G_1 A_1}\braket{G_i \wt W G_i A_i} X_{[p]\setminus\set{1,i}} Y_{(p,q]} + \landauO*{\frac{N^\epsilon \Psi}{L^{1/2}}} \\
      & \quad + \sum_{i\in(p,q]} m_1 \E \wt\E \braket{\wt W G_1 A_1}\braket{G_i \wt W G_i} X_{[p]\setminus\set{1}}Y_{(p,q]\setminus\set{i}} \\
      & \quad - \sum_{k\ge 2} \sum_{ab} \sum_{\bm\alpha\in\set{ab,ba}^k} \frac{\kappa(ab,\bm\alpha)}{k! N} \E \partial_{\bm\alpha} \biggl[ m_1(G_1A_1)_{ba} X_{[p]\setminus\set{1}}Y_{(p,q]} \biggr]
    \end{split}
  \end{equation}
  
  We note that for arbitrary matrices \(U,V\) independent of \(\wt W\) we have 
  \begin{equation}\label{eq WX WY}
    \wt\E \braket{\wt W U}\braket{\wt W V } =  \frac{1}{N^2} \Bigl( \braket{UV} + \sigma\braket{UV^t} + 
    \widetilde{w_2}\braket{\diag(U)\diag(V)} \Bigr),
  \end{equation}
  so that it follows that 
  \begin{equation}\label{var A 1st}
    \begin{split}
      &N^2\wt\E \braket{\wt W G_1 A_1}\braket{G_i \wt W G_i A_i} \\
      &= \braket{G_1 A_1 G_i A_i G_i} + \sigma \braket{G_1 A_1 G_i^t A_i^t G_i^t } + \widetilde{w_2}\braket{\diag(G_1 A_1) \diag(G_i A_i G_i)}\\
      &= \frac{m_1m_i^2\braket{A_1A_i}}{1-m_1m_i}+\frac{\sigma m_1m_i^2\braket{A_1A_i^t}}{1-\sigma m_1m_i} + m_1 m_i^2 \widetilde{w_2}\braket{\bm a_1\bm a_i} + \landauOstd*{ \frac{N^2\Psi_{\set{1,i}}}{\sqrt{L}} }
    \end{split}
  \end{equation}
  from~\eqref{eq:avll3G11},~\eqref{eq:avll3G21}, and 
  \begin{equation*}
    \begin{split}
      \braket{\diag(G_1 A_1) \diag(G_i A_i G_i)} &= m_1 \braket{\diag(\bm a_1) G_i A_i G_i } + \landauOprec*{\frac{1}{\sqrt{L}} \frac{\sqrt{\rho_i}}{\sqrt{\eta_i}} } \\
      &= m_1 m_i^2 \braket{\bm a_1 \bm a_i} + \landauOprec*{\frac{1}{\sqrt{L}} \frac{\sqrt{\rho_i}}{\sqrt{\eta_i}} }
    \end{split}
  \end{equation*}
  due to the second bound in~\eqref{eq:isll2G} and the second local law in~\eqref{eq:avll2G1}. Similarly, for the second line on the rhs.\ of~\eqref{CLT big expansion} we obtain 
  \begin{equation}\label{var A 2nd}
    \begin{split}
      &N^2\wt\E \braket{\wt W G_1 A_1}\braket{G_i \wt W G_i } \\
      &= \braket{G_1 A_1 G_i^2} + \sigma \braket{G_1 A_1 (G_i^2)^t } + \widetilde{w_2} \braket{\diag(G_1 A_1) \diag(G_i^2)}= \landauOstd*{ \frac{N^2\Psi_{\set{1,i}}}{\sqrt{L}} }.
    \end{split}
  \end{equation}
  from~\eqref{eq:avll3G3} and
  \[ \braket{\diag(G_1 A_1) \diag(G_i^2)} = m_1 \braket{\diag(\bm a_1) G_i^2} + \landauOprec*{ \frac{1}{\sqrt{L}\eta_i} } = \landauOprec*{ \frac{1}{\sqrt{L}\eta_i} }  \]
  due to~\eqref{eq:avll2G4} and the first local law in~\eqref{eq:isll2G}.
  
  It remains to consider the third line in~\eqref{CLT big expansion} where due to the Leibniz rule many terms can arise from the derivative. For \(k\ge 2\) the derivative may act on \(G_1\) or any of the \(X_i,Y_i\) and we consider the corresponding terms separately as in 
  \begin{equation}\label{Xi_k}
    \begin{split}
      &\sum_{k\ge 2} \sum_{ab} \sum_{\bm\alpha\in\set{ab,ba}^k} \frac{\kappa(ab,\bm\alpha)}{k! N} \partial_{\bm\alpha} \biggl[ m_1(G_1A_1)_{ba} X_{[p]\setminus\set{1}}Y_{(p,q]} \biggr]  \\
      &\quad = \sum_{k\ge 2} \sum_{\abs{P_X\cup P_Y}\le k} X_{(1,p]\setminus S(P_X)}Y_{(p,q]\setminus S(P_Y)} \Xi_k(P_X,P_Y),\\
      & \Xi_k(P_X,P_Y):= \sum_{ab} \sum_{\bm\alpha}\kappa(ab,\bm\alpha)\biggl(\partial_{\bm\alpha_1} \frac{m_1(G_1A_1)_{ba}}{k_1!N}\biggr) \biggl(\prod_{i\in S(P_X)} \frac{\partial_{\bm\alpha_i} X_i}{k_i!}\biggr) \biggl(\prod_{i\in S(P_Y)} \frac{\partial_{\bm\alpha_i} Y_i}{k_i!}\biggr),
    \end{split}
  \end{equation}
  where \(P_X,P_Y\) are unordered multisets with \emph{support} \(S(P_X)\subset (1,p]\), \(S(P_Y)\subset (p,q]\). The last summation \(\sum_{\bm\alpha}\) indicates the summation over tuples \(\bm\alpha_1\in\{ab,ba\}^{k_1}\), \(\bm\alpha_i\in\{ab,ba\}^{k_i}\) with \(k_i\ge 1\) denoting the multiplicity of \(i\) in \(S(P_X\cup P_Y)\) and \(k_1:= k-\abs{P_X\cup P_Y}\ge 0\). We will prove below that 
  \begin{equation}
  \label{Xi k claim}
  \begin{split}
    \Xi_k(P_X,P_Y) &=\landauOE*{\frac{\Psi_{\set{1}\cup S(P_X\cup P_Y)}}{\sqrt{L}}} \\
    &\quad - m_1^3 m_i^3 \braket{\bm a_1 \bm a_i} \frac{\kappa_4}{N^2} \bm1\Bigl((k,P_X,P_Y)=(3,\set{i,i},\emptyset), i\in(1,p]\Bigr).
    \end{split}
  \end{equation}
  
  By combining~\eqref{CLT big expansion},~\eqref{var A 1st},~\eqref{var A 2nd},~\eqref{Xi_k} and~\eqref{Xi k claim} we obtain from induction on the number of \(X\)-factors
  \begin{equation}\label{X pairing conclusion}
    \E X_{[p]}Y_{(p,q]} = \E Y_{(p,q]} \frac{1}{N^p} \sum_{P\in\mathrm{Pair}([p])} \prod_{\set{i,j}\in P} V^\circ_{i,j}(A_i,A_j)  + \landauO*{\frac{N^\epsilon \Psi}{L^{1/2}}}.
  \end{equation}
  Here we used that for \(X_{(1,p]\setminus S(P_X)}\) and \(Y_{(p,q]\setminus S(P_Y)}\) we have the high probability a priori bounds \(\abs{X_{(1,p]\setminus S(P_X)}}\prec \Psi_{(1,p]\setminus S(P_X)}\) and \(\abs{Y_{(p,q]\setminus S(P_Y)}}\prec \Psi_{(p,q]\setminus S(P_Y)}\) and therefore 
  \[ \E \abs*{X_{(1,p]\setminus S(P_X)}Y_{(p,q]\setminus S(P_Y)} \landauOE*{\frac{\Psi_{\set{1}\cup S(P_X\cup P_Y)}}{\sqrt{L}}}} = \E\landauOE*{ \frac{\Psi_{[1,q]}}{\sqrt{L}} } \lesssim \frac{N^\epsilon \Psi}{\sqrt{L}}.\]
  In order to complete the proof of the theorem it remains to compute \(\E Y_{(p,q]}\). For convenience of notation we relabel \(Y_{(p,q]}\) to \(Y_{[r]}\) with \(r=q-p+1\) and obtain, analogously to~\eqref{CLT big expansion},
  \begin{equation}\label{CLT big expansion no A}
    \begin{split}
      \E Y_{[r]} ={}& \sum_{i\in[2,r]} \frac{m_1'}{m_1} \E\wt\E\braket{\wt W G_1}\braket{G_i\wt W G_i}Y_{[r]\setminus \set{1,i}} - \frac{2\kappa_4}{N}m_1'm_1^3 Y_{[2,r]} \\
      &\quad - \sum_{k\ge 2} \sum_{ab} \sum_{\bm\alpha\in\set{ab,ba}^k} \frac{\kappa(ab,\bm\alpha)}{k!N } \E \partial_{\bm\alpha} \biggl[ \frac{m_1'}{m_1}(G_1)_{ba} Y_{(1,r]} \biggr].
    \end{split}
  \end{equation}
  For the first term on the rhs.\ of~\eqref{CLT big expansion no A} we obtain with~\eqref{eq WX WY} that
  \begin{equation}\label{var no A 1st}
    \begin{split}
      &N^2\wt\E\braket{\wt W G_1}\braket{G_i\wt W G_i} \\
      &= \braket{G_1 G_i^2} + \sigma\braket{G_1 (G_i^t)^2} + \widetilde{w_2}\braket{\diag(G_1 )\diag(G_i^2)} \\
      &= \frac{m_1 m_i'}{(1-m_1m_i)^2} + \frac{\sigma m_1 m_i'}{(1-\sigma m_1m_i)^2} + \widetilde{w_2}m_1 m_i' + \landauOprec*{ \frac{\rho_1}{\sqrt{L}\eta_1\eta_i} },
    \end{split}
  \end{equation}
  where in the last step we used~\eqref{eq:avll3G1},~\eqref{eq:avll3G2}, and 
  \[ \braket{\diag(G_1 )\diag(G_i^2)}= m_1 \braket{G_i^2} + \landauOprec*{ \frac{\rho_1}{L^{1/2} \eta_i} } = m_1 m_i' + \landauOprec*{ \frac{\rho_1}{L^{1/2} \eta_i} } \]
  due to~\eqref{eq:isll1G}, the first local law in~\eqref{eq:avll2G1}, and the first local law in~\eqref{eq:isll2G}.
  
  For the second line of~\eqref{CLT big expansion no A} we distribute the derivative according to the Leibniz rule as
  \begin{equation}\label{Phi k def}
    \begin{split}
      &\sum_{ab} \sum_{\bm\alpha\in\set{ab,ba}^k} \frac{\kappa(ab,\bm\alpha)}{k!N } \E \partial_{\bm\alpha} \biggl[ \frac{m_1'}{m_1}(G_1)_{ba} Y_{(1,r]} \biggr]= \sum_{\abs{P_Y}\le k}  \E Y_{(1,r]\setminus S(P_Y)} \Phi_k(P_Y),\\
      &\Phi_k(P_Y):={} \sum_{ab}\sum_{\bm\alpha} \frac{\kappa(ab,\bm\alpha)}{N}\biggl(\frac{m_1'}{m_1} \partial_{\bm\alpha_1}\frac{(G_1)_{ba}}{k_1!}\biggr) \biggl(\prod_{i\in S(P_Y)} \frac{\partial_{\bm\alpha_i}Y_i}{k_i!}\biggr),
    \end{split}
  \end{equation}
  where \(P_Y\) is a multiset with support \(S(P_Y)\subset (1,r]\), and the summation \(\sum_{\bm\alpha}\) indicates the summation over tuples \(\bm\alpha_1\in\set{ab,ba}^{k_1}\), \(\bm\alpha_i\in\set{ab,ba}^{k_i}\) with \(k_i\ge 1\) denoting the multiplicity of \(i\in S(P_Y)\) and \(k_1:=k-\abs{P_Y}\ge 0\). Similarly to~\eqref{Xi k claim} (but in high probability sense) we prove below that 
  \begin{equation}\label{Phi k claim}
    \begin{split}
      \Phi_k(P_Y)&= \landauOprec*{\frac{\Psi_{\set{1}\cup S(P_Y)}}{\sqrt{L}}}-(m_1^2)' (m_i^2)' \frac{\kappa_4}{2N^2} \bm1\Bigl((k,P_Y)=(3,\set{i,i}),i\in(1,r]\Bigr) \\
      &\qquad- \frac{\kappa_4}{N} m_1' m_1^3 \bm1\Bigl((k,P_Y)=(3,\emptyset)\Bigr).
    \end{split}
  \end{equation}
  By combining~\eqref{CLT big expansion no A},~\eqref{var no A 1st},~\eqref{Phi k def} and~\eqref{Phi k claim} we conclude 
  \begin{equation}\label{Y pairing conclusion}
    \E Y_{[r]}= \frac{1}{N^r} \sum_{Q\in\mathrm{Pair}([r])} \prod_{\set{i,j}\in Q} V_{ij} + \landauO*{\frac{N^\epsilon \Psi_{[r]}}{L^{1/2}} }.
  \end{equation}
  
  Therefore the claim~\eqref{eq CLT statement} follows immediately from combining~\eqref{X pairing conclusion} and~\eqref{Y pairing conclusion} and the proof of the Theorem~\ref{CLT theorem} is complete, modulo the proofs of~\eqref{Xi k claim} and~\eqref{Phi k claim} which we present after Lemma~\ref{lemma aux bound} below. 
\end{proof}

\subsection{Auxiliary calculations: Proof of~\eqref{eq exp expansion ge3},~\eqref{Xi k claim} and~\eqref{Phi k claim}}\label{aux eqs}
\begin{proof}[Proof of~\eqref{eq exp expansion ge3}]
  For \(k=2\) 
  the summation in~\eqref{eq exp expansion ge3} 
  has either two or none diagonal \(G\)'s and from~\eqref{eq:isll1G} we estimate the corresponding terms by 
  \[ N^{-5/2}\sum_{ab} \abs{G_{ba}}^3 \prec N^{-3/2}+\frac{\rho^{3/2}}{N^2\eta^{3/2}} \]
  and 
  \[ N^{-5/2} \abs*{\sum_{ab} G_{bb} G_{aa} G_{ba}} \prec N^{-3/2} + \frac{\rho^{3/2}}{N^2\eta^{3/2}}. \]
 Here the first estimate uses only \( \abs{G_{ab}}\prec \bm1(a=b)+\sqrt{\rho/N\eta}\), while the second one uses
   \(G=m+(G-m)\) and the \emph{isotropic resummation} procedure.
   More precisely, by this we mean the idea of summing up the free indices into constant vectors, i.e.
    \begin{equation}\label{iso resum}
    \begin{split}
      &\sum_{ab} G_{bb} G_{aa} G_{ba} \\
      &= \sum_{ab}\biggl[ m^2 + m(G-m)_{aa}+m(G-m)_{bb}+\landauO*{\frac{\rho}{N\eta}}\biggr] G_{ba}\\
      &= m^2 G_{\bm 1\bm 1} + m \sum_a G_{\bm 1 a}\landauO*{\sqrt{\frac{\rho}{N\eta}}}  +m \sum_bG_{b\bm 1}\landauO*{\sqrt{\frac{\rho}{N\eta}}} + \sum_{ab} \landauO*{\frac{\rho}{N\eta}} G_{ba} \\
      &=\landauO*{  \abs{G_{\bm 1\bm 1}} + N^{1/2}\sqrt{\frac{\rho}{N\eta}} \sqrt{(GG^\ast)_{\bm 1\bm 1}}+N \frac{\rho}{N\eta}\sqrt{\Tr GG^\ast} } = \landauO*{N+\frac{N^{1/2}\rho^{3/2}}{\eta^{3/2}}},
    \end{split}
  \end{equation}
  where we used a Schwarz inequality, and in the last step the isotropic local law for the all-one vector \(\bm 1=(1,\ldots,1)\) of norm \(\norm{\bm 1}=\sqrt{N}\). Thus we obtain a bound \(N^{-3/2}+N^{-2}\rho^{3/2}/\eta^{3/2}\) 
  for the \(k=2\) terms in~\eqref{eq exp expansion ge3}. 
  
  Next, we consider the \(k=3\) terms which give a contribution of \(\rho N^{-2}\eta^{-1}\) whenever there are at least two off-diagonal \(G\)'s. In order to achieve only diagonal \(G\)'s, \(\bm\alpha\) is necessarily one of
  \((ab,ba,ba)\), \((ba,ab,ba)\), or \((ba,ba,ab)\),
  for which we obtain \(\kappa(ab,ba,ba,ab)=\kappa_4/N^2\) for \(a\ne b\). The derivative then is given by 
  \begin{equation}
    \partial_{\bm\alpha} G_{ba}=-\partial_{ba}\partial_{ab} G_{bb}G_{aa}=2\partial_{ba} G_{ba}G_{bb}G_{aa} = - 2 G_{aa}^2 G_{bb}^2+\cdots
  \end{equation}
  where the neglected terms contain two off-diagonal \(G\)'s and can hence be neglected. Therefore the \(k=3\) contribution of~\eqref{eq exp expansion ge3} is given by 
  \[ -2 \frac{3}{3!} \kappa_4 \frac{1}{N^3} \sum_{ab}G_{aa}^2 G_{bb}^2 = -\frac{\kappa_4}{N}m^4 + 
  \landauOprec*{\frac{\rho}{N^2\eta}+\frac{\rho^{3/2}}{N^{3/2}\eta^{1/2}}}.\]
  By estimating the \(k\ge 4\) contribution trivially via \(\abs{G_{ab}}\prec1\) this concludes the proof of~\eqref{eq exp expansion ge3}.
\end{proof}

\begin{lemma}[Auxiliary a priori estimates]\label{lemma aux bound}
  For \(X_i,Y_i\) from~\eqref{Xi Yi def} and their derivatives \(\partial_{\bm\alpha}X_i,\partial_{\bm\alpha}Y_i\) for any multi-index \(\bm\alpha\) we have the high probability a priori estimates
  \begin{equation}\label{a priori XY}
    \abs{\partial_{\bm\alpha}X_i} \prec \frac{\rho_i^{1/2}}{N\eta_i^{1/2}}, \qquad \abs{\partial_{\bm\alpha}Y_i} \prec \frac{1}{N\eta_i},
  \end{equation}
  and the more precise expansions for the first and second order derivatives
  \begin{equation}\label{Yi der claim}
    \partial_{ab}Y_i = - m_i'\frac{\delta_{ba}}{N} + \landauOprec*{\frac{\rho_i^{3/2}}{(N\eta_i)^{3/2}}}, \quad \partial_{ab}\partial_{cd} Y_i = 2m_i m_i' \frac{\delta_{da}\delta_{bc}}{N} + \landauOprec*{\frac{\rho_i^{3/2}}{(N\eta_i)^{3/2}}}.
  \end{equation}
  Moreover, we have the expansions 
  \begin{equation}
  \label{Xi der claim}
  \begin{split}
    \partial_{ab}X_i &= - m_i^2 \frac{(A_i)_{ba}}{N} + \landauOstd*{\frac{\rho_i}{N^{3/2}\eta_i}}, \\
     \partial_{ab}\partial_{cd} X_i &= m_i^3 \frac{(A_i)_{da}\delta_{bc}+(A_i)_{bc}\delta_{da}}{N} + \landauOstd*{\frac{\rho_i}{N^{3/2}\eta_i}},
     \end{split}
  \end{equation}
  in variance sense. 
\end{lemma}
\begin{proof}
  We first establish an isotropic local law in variance sense using~\eqref{eq:mainisogon1} of the form 
  \begin{equation}\label{GAG iso} 
    (GAG)_{\vx\vy} = m^2A_{\vx\vy} + \landauOstd*{\frac{\rho}{N^{1/2}\eta}}
  \end{equation}
  which is proved analogously to~\eqref{2G iso}. The claims~\eqref{Yi der claim}--\eqref{Xi der claim} then follow directly from~\eqref{GAG iso}, the first local law in~\eqref{eq:isll2G}, and
  \begin{equation}\label{1st der derivation}
    \partial_{ab}\braket{GB} = -\frac{(GBG)_{ba}}{N}
  \end{equation}
  and
  \begin{equation}\label{2nd der derivation}
    \begin{split}
      \partial_{ab}\partial_{cd} \braket{GB}&= -\partial_{cd} \frac{(GBG)_{ba}}{N} =\frac{G_{bc}(GBG)_{da}+G_{da}(GBG)_{bc}}{N}.
    \end{split}
  \end{equation}
  The claim~\eqref{a priori XY} follows inductively by the second local law in~\eqref{eq:isll2G} since each additional derivative just adds an additional factor of \(G\) which is at most of order \(1\). 
\end{proof}

\begin{proof}[Proof of~\eqref{Xi k claim}]
  We prove~\eqref{Xi k claim} by considering the following five cases which cover all possibilities:
  (i) \(k_1\) odd, \(k\ge 4\), (ii) \(k_1\) even, \(k\ge 3\), (iii) \(k=3\), \(k_1=1\) (iv) \(k=3\), \(k_1=3\) and (v) \(k=2\). Before considering each case separately, we outline a few ideas that 
  are used repeatedly in the argument. The first idea is that we often replace diagonal resolvents \(G_{aa}\) and \((GA)_{aa}\) using the isotropic local law \(\braket{\vx,G\vy}=m\braket{\vx,\vy}+\landauOprec{\sqrt{\frac{\rho}{N\eta}}}\) in order to make the leading term independent of the summation index \(a\). For example, for \(\sum_{a} G_{aa}G_{a\vx} \) this allows us
  to sum up the index \(a\) into the constant vector \(\bm1=(1,\dots,1)\) of norm \(\sqrt{N}\), effectively gaining a factor of \(\sqrt{\rho/N\eta}\) over the naive estimate since 
  \[ \sum_{a} G_{aa}G_{a\vx} = m G_{\bm 1\vx} + \landauOprec*{\sqrt{\frac{\rho}{N\eta}}\sum_a \abs{G_{a\vx}}} =  \landauOprec*{N\sqrt{\frac{\rho}{N\eta}}}.\]
  The second idea is that for off-diagonal resolvents we use a Schwarz inequality, followed by the Ward identity to effectively also gain a factor of \(\sqrt{\rho/N\eta}\) over the naive estimate, e.g. 
  \[ \abs*{\sum_{a} \abs{G_{a\vx}}} \le \sqrt{N} \sqrt{\sum_{a}\abs{G_{a\vx}}^2} = \sqrt{N}\sqrt{(G^\ast G)_{\vx\vx}}= \sqrt{\frac{N}{\eta}}\sqrt{(\Im G)_{\vx\vx}} \prec N\sqrt{\frac{\rho}{N\eta}}.\]
  Finally, we also frequently  use a simple parity consideration to count off-diagonal 
  resolvents since the local law gives a stronger estimate for them. For an odd number of \(G\)'s, each evaluated in one of the entries \(aa,bb,ab,ba\) with \(a\ne b\) in total occurring equally often, at least one of the \(G\)'s has to be off-diagonal.
  
  \subsubsection*{Case (i), \(k_1\) odd, \(k\ge 4\)} In this case we estimate \(\abs{\partial_{\bm\alpha_1}(G_1A_1)_{ba}}\prec 1\) in the definition of \(\Xi\) in~\eqref{Xi_k} and obtain from~\eqref{a priori XY} that 
  \[\abs{\Xi_k(P_X,P_Y)} \prec N^{-(k+3)/2}N^2  \Psi_{S(P_X\cup P_Y)}\lesssim N^{-(k-3)/2} \Psi_{\set{1}\cup S(P_X\cup P_Y)}, \]
  from \(N^{-1}\lesssim \rho_1^{1/2}N^{-1}\eta_1^{-1/2}\), confirming~\eqref{Xi k claim}. 
  \subsubsection*{Case (ii), \(k_1\) even, \(k\ge 3\)} Since \(k_1=\abs{\bm\alpha_1}\) is odd it follows by parity that at least one \(G\) or \(GA\) factor is off-diagonal, hence by the local law we have that 
  \[\abs{\partial_{\bm\alpha_1}(G_1A_1)_{ba}}\prec \abs{(G_1)_{ab}}+ \abs{(G_1A_1)_{ba}}\] 
  and therefore, by~\eqref{a priori XY} and a Ward-estimate it follows that 
  \[
  \begin{split}
    \abs{\Xi_k(P_X,P_Y)} &\prec N^{-(k+3)/2}\sum_{ab}  \bigl(\abs{(G_1A_1)_{ba}}+\abs{(G_1)_{ab}}\bigr) \Psi_{S(P_X\cup P_Y)} \\
    &\prec N^{-(k+3)/2} N^2 \frac{\sqrt{\rho_1}}{\sqrt{N\eta_1}}  \Psi_{S(P_X\cup P_Y)}\lesssim N^{-(k-2)/2} \Psi_{\set{1}\cup S(P_X\cup P_Y)},
  \end{split}\]
  confirming~\eqref{Xi k claim}.
  \subsubsection*{Case (iii), \(k=3,k_1=3\)} The three derivatives acting on \((G_1 A_1)_{ab}\) results in one \(G_1A\) and three \(G_1\) factors with a total of four \(a\) and four \(b\) indices. By using the local law we replace each \(G_1\) by \(m_1\) and obtain 
  \[ 
  \begin{split}
    \abs{\Xi_3(\emptyset,\emptyset)} &\lesssim N^{-3}\abs*{\sum_{ab} m_1^4 (A_1)_{aa}} + N^{-3}\abs*{\sum_{ab} m_1^4 (A_1)_{ab}} + \landauOprec*{\frac{\Psi_{\set{1}}}{\sqrt{L}}} \\
    &\prec N^{-2} \sqrt{\sum_{ab}\abs{(A_1)_{ab}}^2} + \frac{\Psi_{\set{1}}}{\sqrt{L}} = N^{-3/2} \sqrt{\braket{A_1 A_1^\ast }} +\frac{\Psi_{\set{1}}}{\sqrt{L}},
  \end{split}  \]
  again confirming~\eqref{Xi k claim}. 
  
  \subsubsection*{Case (iv), \(k=3,k_1=1\)}
  For \(k_1=1\) the derivative of \((G_1A_1)_{ba}\) is given by 
  \[
  \begin{split}
  \sum_{\bm\alpha_1}\partial_{\bm\alpha_1}(G_1A_1)_{ba}&=-(G_1)_{ba}(G_1A_1)_{ba}-(G_1)_{bb}(G_1A_1)_{aa} \\
  &=-m_1^2 (1+\delta_{ba}) (A_1)_{aa} + \landauOprec*{\sqrt{\frac{\rho_1}{N\eta_1}}}.
  \end{split}
  \]
  If \(P_X=\{i,i\}\) for some \(i\in(1,p]\), then we obtain from~\eqref{Xi der claim} that 
  \[ \partial_{ab,ba} X_i = \partial_{ba,ab} X_i = m_i^3 \frac{(A_i)_{bb}+(A_i)_{aa}}{N} + \landauOstd*{\frac{\rho_i}{N^{3/2}\eta_i}},\]
  while both the \(\partial_{ab,ab}\) and \(\partial_{ba,ba}\) derivatives lead to delta functions \(\delta_{ab}\) and are therefore lower order after summation, and thus
  \[
  \begin{split}
    \Xi_3 (\set{i,i},\emptyset) &= -m_1^3 m_i^3 \sum_{ab} \kappa(ab,ba,ab,ba)\frac{(A_1)_{aa}}{N}\frac{(A_i)_{aa}+(A_i)_{bb}}{N} + \landauOstd*{\frac{\Psi_{\set{1,i}}}{L^{1/2}}} \\
    &= - \frac{\kappa_4}{N^2}m_1^3m_i^3 \Bigl(\braket{\bm a_1\bm a_i}+\braket{A_1}\braket{A_i} \Bigr)+ \landauOstd*{\frac{\Psi_{\set{1,i}}}{L^{1/2}}} \\
    &= - \frac{\kappa_4}{N^2}m_1^3m_i^3 \braket{\bm a_1\bm a_i}+ \landauOstd*{\frac{\Psi_{\set{1,i}}}{L^{1/2}}},
  \end{split} \]
  giving the leading contribution to~\eqref{Xi k claim}. 
  
  If \(P_Y=\set{i,i}\) for some \(i\in(p,q]\), then we obtain from the leading term of~\eqref{Yi der claim} that 
  \[ \partial_{ab,ba} Y_i = \partial_{ba,ab} Y_i = m_i m_i' \frac{2}{N} + \landauOprec*{\frac{\rho_i^{1/2}}{(N\eta_i)^{3/2}}},\]
  with the other derivatives again being lower order, hence 
  \[
  \begin{split}
    \Xi_3 (\emptyset,\set{i,i}) &= -m_1^3 m_i m_i' \sum_{ab} \kappa(ab,ba,ab,ba)\frac{(A_1)_{aa}}{N}\frac{2}{N} + \landauOprec*{\frac{\Psi_{\set{1,i}}}{L^{1/2}}} \\
    &= - 2\frac{\kappa_4}{N^2}m_1^3m_i^3 \braket{A_1}+ \landauOprec*{\frac{\Psi_{\set{1,i}}}{L^{1/2}}} =\landauOprec*{\frac{\Psi_{\set{1,i}}}{L^{1/2}}}.
  \end{split} \] 
  
  Finally, if \(P_X\cup P_Y=\set{i,j}\) for some \(1<i<j\), then we either obtain a \(\delta_{ab}\) from~\eqref{Yi der claim} or a \((A_i)_{ba}\) from~\eqref{Xi der claim} and therefore due to 
  \[ \sum_{ab} \abs{(A_i)_{ba}} \le N \sqrt{\sum_{ab}\abs{(A_i)_{ba}}^2} = N^{3/2} \sqrt{\braket{A_i A_i^\ast}}\]
  the leading term is at most of size \(\Psi_{\set{1,i,j}}L^{-1/2}\) and we obtain 
  \[\abs{\Xi_3(\set{i,j},\emptyset)}+\abs{\Xi_3(\emptyset,\set{i,j})}+\abs{\Xi_3(\set{i},\set{j})}=\landauOE*{\frac{\Psi_{\set{1,i,j}}}{L^{1/2}}}.\]
  Here we used~\eqref{prec221}, so that in case \(P_X=\set{i,j}\) with \(i\ne j\), the error terms from~\eqref{Xi der claim} can be multiplied. This concludes the proof of~\eqref{Xi k claim} for the case \(k_1=1,k=3\).
  
  \subsubsection*{Case (v), \(k=2\)} In case \(k=2\) there are five sub-cases to consider; \(k_1=2\), \(P_X=\set{i,i}\), \(P_Y=\set{i,i}\), \(k_1=1\) or \(\abs{S(P_X\cup P_Y)}=2\). 
  
  If \(k_1=2\), then the derivative is given by 
  \[\partial^2 (G_1A_1)_{ba}=(G_1)(G_1)(G_1A_1)\]
  with three \(a\) and three \(b\) indices, so that by parity either all three matrices have indices \(ab,ba\), or only one with the remaining two having \(aa,bb\). For all three matrices having \(ab,ba\) indices, we can gain two factors of \(\sqrt{\rho_1/N\eta_1}\) via Ward-estimates over the naive size \(N^{-1/2}\) in order to obtain \(\rho_1 N^{-3/2}\eta_1^{-1}\le \Psi_{\set{1}}L^{-1/2}\). If two matrices have indices \(aa,bb\), then we replace one diagonal resolvent by \(m\) and estimate, for example, 
  \[ 
  \begin{split}
    N^{-5/2} \sum_{ab} (G_1)_{bb}(G_1A_1)_{aa}(G_1)_{ba} &= m_1 N^{-5/2} \sum_{a} (G_1A_1)_{aa} (G_1)_{\bm 1 a} + \landauOprec*{\frac{\rho_1}{N^{3/2}\eta_1} } \\
    &= \landauOprec*{\frac{\rho_1^{1/2}}{N^{3/2}\eta_1^{1/2}}} = \landauOprec*{\frac{\Psi_{\set{1}}}{L^{1/2}} },
  \end{split}\]
  and similarly for all other index distributions. Thus we obtain 
  \[ \abs{\Xi_2(\emptyset,\emptyset)} \prec \frac{\Psi_{\set{1}}}{L^{1/2}}.\]
  
  Next, if \(P_X=\set{i,i}\), then we obtain from~\eqref{Xi der claim} that 
  \[
  \begin{split}
    &\Xi_2(\set{i,i},\emptyset) = \sum_{ab}\kappa(ab,ab,ba) \frac{m_1(G_1A_1)_{ba}}{N} m_i^3 \frac{(A_i)_{aa}+(A_i)_{bb}}{N} + \landauOstd*{\frac{\Psi_{\set{1,i}}}{L^{1/2}} }  \\
    &\qquad\quad= m_1m_i^3\frac{\kappa_3}{N^{7/2}} \biggl(\sum_b (G_1A_1)_{b\bm a_i} + \sum_a (G_1A_1)_{\bm a_i a} \biggr) + \landauOstd*{\frac{\Psi_{\set{1,i}}}{L^{1/2}} } = \landauOstd*{\frac{\Psi_{\set{1,i}}}{L^{1/2}} }
  \end{split} \]
  using that 
  \[\norm{\bm a_i}:=\norm{\diag A_i} \le N^{1/2}\sqrt{\braket{A_i^\ast A_i}}.\]
  
  The case \(P_Y=\set{i,i}\) is completely analogous, except that using~\eqref{Yi der claim} the constant \(\bm 1\) vector is summed up instead of \(\bm a_i\) and we obtain 
  \[\abs{\Xi_2(\emptyset,\set{i,i})}\prec \frac{\Psi_{\set{1,i}}}{L^{1/2}}.\]
  
  The case \(k_1=1\) can be estimated by 
  \[
  \begin{split}
    \abs{\Xi_2(\set{i},\emptyset)} &\lesssim N^{-7/2}\abs*{ \sum_{ab} \Bigl[(G_1)_{bb}(G_1A_1)_{aa} + (G_1)_{ba}(G_1A_1)_{ba} \Bigr] (G_i A_i G_i)_{ba} } \\
    &\lesssim N^{-7/2} \sum_a \abs{(G_i A_i G_i)_{\bm1a}} + N^{-7/2} \sum_{ab} \sqrt{\frac{\rho_1}{N\eta_1}}\abs{(G_i A_i G_i)_{ba}}\\
    &\lesssim \frac{\rho_1^{1/2}}{N\eta_1^{1/2}} \frac{\rho_i}{N^{3/2}\eta_i} \lesssim \frac{\Psi_{\set{1,i}}}{L^{1/2}}
  \end{split}
  \]
  and similarly 
  \[ \begin{split}
    \abs{\Xi_2(\emptyset,\set{i})} &\lesssim N^{-7/2}\abs*{ \sum_{ab} \Bigl[(G_1)_{bb}(G_1A_1)_{aa} + (G_1)_{ba}(G_1A_1)_{ba} \Bigr] (G_i^2)_{ba}} \\
    &\lesssim N^{-7/2} \sum_a \abs{(G_i^2)_{\bm1a}} + N^{-7/2} \sum_{ab} \sqrt{\frac{\rho_1}{N\eta_1}} \abs{(G_i^2)_{ba}} \lesssim \frac{\rho_1^{1/2}}{N\eta_1^{1/2}} \frac{\rho_i^{3/2}}{N^{3/2}\eta_i^{3/2}} \lesssim \frac{\Psi_{\set{1,i}}}{L^{1/2}}
  \end{split} \]
  from~\eqref{1st der derivation}. 
  
  For the final case \(\abs{S(P_X\cup P_Y)}=2\) we estimate for \(i\ne j\) 
  \[ \begin{split}
    \abs{\Xi_2(\set{i,j},\emptyset)} &\lesssim \frac{1}{N^{5/2}} \sum_{ab}\abs{(G_1A_1)_{ba}} \biggl( \frac{\abs{(A_i)_{ba}}}{N} + \landauOstd*{\frac{\rho_i}{N^{3/2}\abs{\eta_i}}} \biggr) \\
    &\quad\times\biggl( \frac{\abs{(A_j)_{ba}}}{N} + \landauOstd*{\frac{\rho_j}{N^{3/2}\abs{\eta_j}}} \biggr) \\
    &=\landauOE*{N^{-1/2} \sqrt{\frac{\rho_1}{N\abs{\eta_1}}} \Bigl( \frac{1}{N^{3/2}}+ \frac{\rho_i}{N^{3/2}\abs{\eta_i}} \Bigr) \frac{\rho_j^{1/2}}{N\abs{\eta_j}^{1/2}}}= \landauOE*{\frac{\Psi_{\set{1,i,j}}}{L^{1/2}}}\\
    \abs{\Xi_2(\emptyset,\set{i,j})} &\lesssim N^{-5/2} \sum_{ab}\abs{(G_1A_1)_{ba}} \biggl( \frac{\delta_{ba}}{N} + \landauOprec*{\frac{\rho_i^{1/2}}{(N\abs{\eta_i})^{3/2}}} \biggr) \\
    &\quad\times\biggl( \frac{\delta_{ba}}{N} + \landauOprec*{\frac{\rho_j^{1/2}}{(N\abs{\eta_j})^{3/2}}} \biggr) \\
    &\prec N^{-5/2}\frac{1}{N\abs{\eta_j}} +  N^{-5/2} N^2 \sqrt{\frac{\rho_1}{N\abs{\eta_1}}} \frac{\rho_i^{1/2}}{(N\abs{\eta_i})^{3/2}} \frac{1}{N\abs{\eta_j}}\prec \frac{\Psi_{\set{1,i,j}}}{L^{1/2}}
  \end{split} \]
  and similarly 
  \[ \abs{\Xi_2(\set{i},\set{j})} = \landauOstd*{\frac{\Psi_{\set{1,i,j}}}{L^{1/2}}}.\]
  This concludes the proof of~\eqref{Xi k claim}. 
\end{proof}

\begin{proof}[Proof of~\eqref{Phi k claim}]
  The proof of~\eqref{Phi k claim} is very similar to that of~\eqref{Xi k claim} and we again consider the cases (i) \(k_1\) odd, \(k\ge 4\), (ii) \(k_1\) even, \(k\ge 3\), (iii) \(k=3\), \(k_1=1\), (iv) \(k=3\), \(k_1=3\) separately. 
  
  \subsubsection*{Case (i), \(k_1\) odd, \(k\ge 4\)} In this case we estimate \(\abs{\partial_{\bm\alpha_1}(G_1)_{ba}}\prec 1\) and obtain from~\eqref{a priori XY} that 
  \[\abs{\Phi_k(P_Y)} \prec \rho_1^{-1} N^{-(k+3)/2}N^2  \Psi_{S(P_Y)}\lesssim N^{-(k-3)/2} \Psi_{\set{1}\cup S(P_Y)}, \]
  from \(N^{-1}\lesssim \rho_1N^{-1}\eta_1^{-1}\), confirming~\eqref{Phi k claim}. 
  \subsubsection*{Case (ii), \(k_1\) even, \(k\ge 3\)} Since \(k_1\) is odd it follows by parity and the local law that \(\abs{\partial_{\bm\alpha_1}(G_1)_{ba}}\prec \abs{(G_1)_{ab}}\)
  and therefore, by~\eqref{a priori XY} and a Ward-estimate it follows that 
  \[\abs{\Phi_k(P_Y)} \prec \rho_1^{-1} N^{-(k+3)/2} N^2 \frac{\sqrt{\rho_1}}{\sqrt{N\eta_1}}  \Psi_{S(P_Y)}\lesssim N^{-(k-2)/2} \Psi_{\set{1}\cup S(P_Y)},\]
  confirming~\eqref{Phi k claim}.
  \subsubsection*{Case \(k=3,k_1=3\)} The derivatives acting on \((G_1)_{ab}\) results in four \(G_1\) factors with a total of four \(a\) and four \(b\) indices and we obtain
  \[ 
  \begin{split}
    \Phi_3(\emptyset) &= - \kappa_4 \frac{m_1'}{m_1}N^{-3}\sum_{ab} (G_1)_{aa}^2 (G_1)_{bb}^2 + \landauO*{ \rho_1^{-1} N^{-3}\sum_{ab} \abs{(G_1)_{ab}}  } \\
    &= -\frac{\kappa_4}{N} m_1^3m_1' + \landauOprec*{\frac{1}{N^{3/2}\sqrt{\eta_1\rho_1}}} = - \frac{\kappa_4}{N} m_1^3 m_1' + \landauOprec*{\frac{\Psi_{\set{1}}}{\sqrt{L}}},
  \end{split}  \]
  which gives one of the leading terms in~\eqref{Phi k claim}.
  
  \subsubsection*{Case (iii), \(k=3,k_1=1\)}
  For \(k_1=1\) the derivative of \((G_1)_{ba}\) is given by 
  \[\sum_{\bm\alpha_1}\partial_{\bm\alpha_1}(G_1)_{ba}=-(G_1)_{ba}^2-(G_1)_{bb}(G_1)_{aa}=-m_1^2 (1+\delta_{ba}) + \landauOprec*{\sqrt{\frac{\rho_1}{N\eta_1}}}.\]
  If \(P_Y=\{i,i\}\) for some \(i\in(1,r]\), then we obtain from~\eqref{Yi der claim} that 
  \[ \partial_{ab,ba} Y_i = \partial_{ba,ab} Y_i = m_i m_i' \frac{2}{N} + \landauOprec*{\frac{\rho_i^{1/2}}{(N\eta_i)^{3/2}}},\]
  while both the \(\partial_{ab,ab}\) and \(\partial_{ba,ba}\) derivatives lead to delta functions \(\delta_{ab}\) and are therefore lower order, and thus
  \[
  \begin{split}
    \Phi_3 (\set{i,i},\emptyset) &= -m_1 m_1' m_i m_i' \sum_{ab} \kappa(ab,ba,ab,ba)\frac{1}{N}\frac{2}{N} + \landauOprec*{\frac{\Psi_{\set{1,i}}}{L^{1/2}}} \\
    &= - \frac{2\kappa_4}{N^2}m_1 m_1' m_i m_i' + \landauOprec*{\frac{\Psi_{\set{1,i}}}{L^{1/2}}} = -\frac{\kappa_4}{2N^2}(m_1^2)' (m_i^2)' + \landauOprec*{\frac{\Psi_{\set{1,i}}}{L^{1/2}}},
  \end{split} \]
  giving the other leading term in~\eqref{Phi k claim}. On the other hand, if \(P_Y=\set{i,j}\) for some \(1<i<j\), then we obtain a \(\delta_{ab}\) from~\eqref{Yi der claim} and therefore the leading term is at most of size \(\Psi_{\set{1,i,j}}L^{-1/2}\) and we obtain 
  \[\abs{\Phi_3(\set{i,j})}=\landauOprec*{\frac{\Psi_{\set{1,i,j}}}{L^{1/2}}}.\]
  This concludes the proof of~\eqref{Phi k claim} for the case \(k_1=1,k=3\).
  
  \subsubsection*{Case (iv), \(k=2\)} In case \(k=2\) there are four sub-cases to consider; \(k_1=2\), \(P_Y=\set{i,i}\), \(k_1=1\) or \(P_Y=\set{i,j}\) for \(i\ne j\). 
  
  If \(k_1=2\), then the derivative is given by \(\partial^2 (G_1)_{ba}=(G_1)^3\) with three \(a\) and three \(b\) indices, so that by parity either all three matrices have indices \(ab,ba\), or only one with the remaining two having \(aa,bb\). For all three matrices having \(ab,ba\) indices, we can gain two factors of \(\sqrt{\rho_1/N\eta_1}\) via Ward-estimates over the naive size \(N^{-1/2}\) in order to obtain \( N^{-3/2}\eta_1^{-1}\le \Psi_{\set{1}}L^{-1/2}\). If two matrices have indices \(aa,bb\), then we estimate 
  \[ 
  \begin{split}
    \rho_1^{-1} N^{-5/2} \sum_{ab} (G_1)_{bb}(G_1)_{aa}(G_1)_{ba} &= m_1 N^{-5/2} \sum_{a} (G_1)_{aa} (G_1)_{\bm 1 a} + \landauOprec*{\frac{\rho_1}{N^{3/2}\eta_1} } \\
    &= \landauOprec*{\frac{\rho_1^{1/2}}{N^{3/2}\eta_1^{1/2}}} = \landauOprec*{\frac{\Psi_{\set{1}}}{L^{1/2}} },
  \end{split}\]
  in order to conclude
  \[ \abs{\Phi_2(\emptyset)} \prec \frac{\Psi_{\set{1}}}{L^{1/2}}.\]
  
  Next, if \(P_Y=\set{i,i}\), then we obtain from~\eqref{Yi der claim} that 
  \[
  \begin{split}
    \Phi_2(\set{i,i}) &= \frac{m_1'm_i m_i'}{m_1}\sum_{ab}\kappa(ab,ab,ba) \frac{(G_1)_{ba}}{N} \frac{2}{N} + \landauOprec*{\frac{\Psi_{\set{1,i}}}{L^{1/2}} }  \\
    &= \frac{m_1'm_i m_i'}{m_1}\frac{2\kappa_3}{N^{7/2}} \braket{\bm 1, G_1\bm 1} + \landauOprec*{\frac{\Psi_{\set{1,i}}}{L^{1/2}} } = \landauOprec*{\frac{\Psi_{\set{1,i}}}{L^{1/2}} }.
  \end{split} \]
  
  The case \(k_1=1\) can be estimated by 
  \[ \begin{split}
    \abs{\Phi_2(\set{i})} &\lesssim \frac{N^{-7/2}}{\rho_1\rho_i}\abs*{ \sum_{ab} \Bigl[(G_1)_{bb}(G_1)_{aa} + (G_1)_{ba}^2 \Bigr] (G_i^2)_{ba}} \\
    &\lesssim \frac{N^{-7/2}}{\rho_1\rho_i} \sum_a \abs{(G_i^2)_{\bm1a}} + \frac{N^{-7/2}}{\rho_1\rho_i} \sum_{ab} \sqrt{\frac{\rho_1}{N\eta_1}}\abs{(G_i^2)_{ba}}\lesssim \frac{\rho_1^{1/2}}{N\eta_1^{1/2}} \frac{\rho_i}{N^{3/2}\eta_i} \lesssim \frac{\Psi_{\set{1,i}}}{L^{1/2}}.
  \end{split} \]
  
  For the final case \(P_Y=\set{i,j}\) with \(i\ne j\) we obtain from~\eqref{Yi der claim} that
  \[ \begin{split}
    \abs{\Phi_2(\set{i,j})} &\lesssim \frac{N^{-5/2}}{\rho_1\rho_i\rho_j} \sum_{ab}\abs{(G_1)_{ba}} \biggl( \frac{\delta_{ba}}{N} + \landauOprec*{\frac{\rho_i^{1/2}}{(N\eta_i)^{3/2}}} \biggr) \biggl( \frac{\delta_{ba}}{N} + \landauOprec*{\frac{\rho_j^{1/2}}{(N\eta_j)^{3/2}}} \biggr) \\
    &\prec \frac{\Psi_{\set{1,i,j}}}{L^{1/2}}.
  \end{split} \]
  This concludes the proof of~\eqref{Phi k claim}.  
\end{proof}

\section{Functional CLT\@: Proof of Theorems~\ref{theo:sharpcom}--\ref{theo:CLT}.}\label{sec:proosc}



In this section we prove our main results, the functional central limit theorems 
using the resolvent CLT, Theorem~\ref{CLT theorem}. Via standard 
representation formulas  this involves fairly standard but tedious calculations. 
We first give a detailed calculation  for  the case sharp cut-off case in Section~\ref{sec:shaco} using the less-known Pleijel's formula which proves Theorems~\ref{theo:sharpcom}.  The proof of Theorem~\ref{theo:CLT} in Section~\ref{sec:smoco} relies on similar calculations using the more conventional Helffer-Sj\"ostrand formula; the details are deferred to Appendix~\ref{sec:proofclt}.


\subsection{Proof of the functional CLT for the sharp cut-off}
\label{sec:shaco}

\begin{proof}[Proof of Theorem~\ref{theo:sharpcom}]
  Let \(\mathring{A}:=A-\braket{A}\), and define
  \begin{equation}\label{eq:charf}
    \bm1_{K,i_0}(i):=\bm1(|i-i_0|\le K).
  \end{equation}
  
  We recall the rigidity bound
  (see e.g.~\cite[Lemma 7.1, Theorem 7.6]{MR3068390} or~\cite[Section 5]{MR2871147}): 
  \begin{equation}
    \label{eq:rig}
    |\lambda_i-\gamma_i|\prec \frac{1}{N^{2/3}\widehat{i}^{1/3}},
  \end{equation}
  where \(\widehat{i}:=i\wedge (N+1-i)\). Here \(\gamma_i\) are the classical eigenvalue locations (\emph{quantiles}) defined by
  \begin{equation}
    \label{eq:quantin}
    \int_{-\infty}^{\gamma_i} \rho(x)\, \dif x=\frac{i}{N}, \qquad i\in [N],  
  \end{equation}
  where we recall \(\rho(x)=\rho_\mathrm{sc}(x)=(2\pi)^{-1}\sqrt{(4-x^2)_+}\). We now present the proof in the bulk regime, the edge is completely analogous and so omitted. Define \(\eta_K(\gamma_{i_0}) \) implicitly by 
  \begin{equation}
  \label{eq:defetai0}
  \eta_{i_0}=\eta_K(\gamma_{i_0}):=\frac{K}{ N\rho(\gamma_{i_0}+\ii\eta_K(\gamma_{i_0}))},
  \end{equation}
  Then, by~\eqref{eq:charf} and~\eqref{eq:rig}, we readily conclude
  \begin{equation}
    \sqrt{\frac{N}{2 K}}\sum_{|i-i_0|\le K} \braket{{\bm u}_i,\mathring{A}{\bm u}_i}=\sqrt{\frac{N}{2 K}}\sum_{i=1}^N \bm1_{K,i_0}(i)\braket{{\bm u}_i,\mathring{A}{\bm u}_i}=\frac{N^{3/2}}{\sqrt{2 K}}\braket{P(W)\mathring{A}}+\landauOprec*{\frac{1}{\sqrt{K}}},
  \end{equation}
  where we defined the spectral projection
  \begin{equation}
  \label{eq:specpro}
  P(W)=\bm1\left(\gamma_{i_0}-\eta_{i_0}\le W \le \gamma_{i_0}+\eta_{i_0}\right),
  \end{equation}
  and used that \(|\braket{{\bm u}_i,\mathring{A} {\bm u}_i}|\prec N^{-1/2}\) by~\cite[Theorem~\ref{eth-theorem overlap}]{2012.13215}.
  
  %

  Using Pleijel's representation formula of the spectral projection of a Hermitian matrix in terms of contour integral of its resolvent in~\cite[Eq.~(13)]{MR3600514} (see also~\cite[Eq.~(5)]{MR167751}), we find that (see Appendix~\ref{sec:proofple} for more details)
  \begin{equation}
  \label{eq:apprnetao}
    \frac{N^{3/2}}{\sqrt{2 K}}\braket{P(W)\mathring{A}}=\frac{N^{3/2}}{2\pi \ii\sqrt{2 K}}\int_{\Gamma_{K,i_0}}\braket{G(z)\mathring{A}}\, \dif z+\landauOprec*{\frac{N\eta_0}{\sqrt{K}}},
  \end{equation}
  with \(\Gamma_{K,i_0}\) the contour oriented counter-clockwise and defined by
  
  \begin{equation}
  \label{eq:defcontourgamma}
  \begin{split}
    \Gamma_{K,i_0}:&=\set*{ z\in \mathbf{C}\given \Re z\in [\gamma_{i_0}-\eta_{i_0},\gamma_{i_0}+\eta_{i_0}] \text{ and } |\Im z|=M} \\
    &\qquad \cup \set*{ z\in \mathbf{C}\given  \Re z\in \set{\gamma_{i_0}-\eta_{i_0},\gamma_{i_0}+\eta_{i_0}} \text{ and } |\Im z|\in [\eta_0, M]},
  \end{split}
  \end{equation}
  for any \(M>0\), and \(\eta_0\) such that \(N^{-1}\ll \eta_0\ll K^{1/2}/N\).
  
By Young's inequality, for any $p\ge 2$  and  $\delta>0$, we get
  from~\eqref{eq:apprnetao} that
  \begin{equation}
  \begin{split}
    \E\left|\frac{N^{3/2}}{\sqrt{2 K}}\braket{P(W)A}\right|^p 
    &=\big((1+ \mathcal{O}(N^{-\delta})\big) \E\left|\frac{N^{3/2}}{\sqrt{2 K}}\int_{\Gamma_{K,i_0}}\braket{G(z)\mathring{A}}\, \dif z\right|^p
    +\mathcal{O}\left(  
    \left(\frac{N^{1+\delta}\eta_0}{\sqrt{K}}\right)^p\right).
    \end{split}
  \end{equation}


  By~\eqref{eq CLT statement}, it follows that for even \(p\) (for odd \(p\) the leading term is zero) we have
  \begin{equation}\label{eq:varsharp}
    \begin{split}
      \E\left|\frac{N^{3/2}}{\sqrt{2 K}}\braket{P(W)A}\right|^p&=\big((1+ \mathcal{O}(N^{-\delta})\big)\left(\frac{N}{2 K}\right)^{p/2}\sum_{\substack{P\in\mathrm{Pair}([p])}} \prod_{\set{i,j}\in P} \frac{1}{4\pi^2}\int_{\Gamma_{K,i_0}}\dif z_i \int_{\Gamma_{K,i_0}}\dif z_j \\
      &\quad \times \Bigg( \frac{m_i^2 m_j^2\braket{A^2}}{1-m_i m_j}+\frac{\sigma m_i^2 m_j^2\braket{AA^t}}{1-\sigma m_i m_j}+\widetilde{w_2}m_i^2m_j^2\braket{\bm a_i\bm a_j} \\
      &\qquad\qquad\quad+\kappa_4 m_i^3 m_j^3\braket{\bm a_i\bm a_j} \Bigg) \\
      &\quad+\landauO*{\frac{N^\xi}{\sqrt{NM}}\left[\Bigl(\frac{MN}{K}\Bigr)^{p/2}+\Bigl(\frac{MN}{K}\Bigr)^{-p/2}\right]+\left(\frac{N^{1+\delta}\eta_0}{\sqrt{K}}\right)^p},
    \end{split}
  \end{equation}
 for any $\xi, \delta>0$. 
 In estimating the error term coming from \eqref{eq CLT statement} we used that
 \[
 \left(\prod_{i\in [p]} \int_{\Gamma_{K,i_0}} \dif z_i\right) \frac{1}{\sqrt{N\eta_*}} \prod_{i\in [p]} \frac{N^\xi}{N\sqrt{|\eta_i|}}\lesssim \frac{N^\xi}{\sqrt{NM}}\left[\Bigl(\frac{MN}{K}\Bigr)^{p/2}+\Bigl(\frac{MN}{K}\Bigr)^{-p/2}\right].
 \]
   where $(MN/K)^{p/2}$ comes from the vertical lines of the contour \(\Gamma_{K,i_0}\) and $(MN/K)^{-p/2}$ from the horizontal ones.  Note that in order to apply \eqref{eq CLT statement} we had to choose $\eta_0\gg N^{-1}$, which ensures 
   $L=N\min_i(|\eta_i|\rho_i) \sim N\eta_*\gg 1$ (since we are in the bulk regime),  with $\eta_*:=\min_i |\eta_i|$.
 It is clear that we can choose $\xi,\delta$, and $\eta_0\ll K^{1/2}/N$, $M\ll K/N$ so that the error term in \eqref{eq:varsharp} is bounded by $N^{-c(p,\epsilon)}$ for some constant $c(p,\epsilon)>0$, where $N^\epsilon
 \le K\le N^{1-\epsilon}$. In the following \(\eta_0\ll M\ll K/N\) ensures that only the horizontal lines of \(\Gamma_{K,i_0}\) contribute to the integral, the vertical lines are negligible giving a contribution \(MN/K\).

  We start computing
  \begin{equation}\label{eq:mainvarco}
    \begin{split}
      &\frac{2 N}{K}\int\int_{\gamma_{i_0}-\eta_{i_0}}^{\gamma_{i_0}+\eta_{i_0}} \dif x \dif y\, \Re\left[\frac{\sigma m_1^2 m_2^2}{1-\sigma m_1 m_2}-\frac{\sigma m_1^2 \overline{m_2}^2}{1-\sigma m_1 \overline{m_2}}\right] \\
      &\quad= -\frac{N}{K}\int\int_{\gamma_{i_0}-\eta_{i_0}}^{\gamma_{i_0}+\eta_{i_0}} \sqrt{(4-x^2)_+(4-y^2)_+} \, \dif x \dif y \\
      &\quad\quad+\bm1(\sigma=\pm 1)\frac{2 N\pi}{K}\int\int_{\gamma_{i_0}-\eta_{i_0}}^{\gamma_{i_0}+\eta_{i_0}}  \sqrt{(4-x^2)_+} \delta_{x-\sigma y}\, \dif  x\dif y \\
      &\quad\quad-\frac{N}{K}\int\int_{\gamma_{i_0}-\eta_{i_0}}^{\gamma_{i_0}+\eta_{i_0}}\frac{(1-\sigma^2)\sqrt{(4-x^2)_+(4-y^2)_+}}{\sigma^2(x^2+y^2)+(1-\sigma^2)^2-xy\sigma (1+\sigma^2)}+\landauO{\sqrt{M}}.
    \end{split}
  \end{equation}

  In particular, for \(K\ll N\) we get that the rhs.\ of~\eqref{eq:mainvarco} is equal to
  \[  
  8\pi^2\bm1(\sigma=1)+\bm1(\sigma=-1)\frac{2 N\pi}{K}\int_{I_{i_0}}\sqrt{4-x^2}\, \dif x+\landauO*{\sqrt{M}+\frac{K}{N}},
  \]
  with \(I_{i_0}:=[\gamma_{i_0}-\eta_{i_0}, \gamma_{i_0}+\eta_{i_0}]\cap [-\gamma_{i_0}-\eta_{i_0}, -\gamma_{i_0}+\eta_{i_0}]\).  Note that for \(i_0=\lceil cN\rceil\) we have
  \[
  \frac{2 N\pi}{K}\int_{I_{i_0}}\sqrt{4-x^2}\, \dif x=\frac{4 N\pi}{K}|I_{i_0}|+\landauO*{\frac{K}{N}}.
  \]
  
  We now distinguish two cases: (i) \(c\ne 1/2\), (ii) \(c=1/2\). If \(c\ne 1/2\) then \(|\gamma_{i_0}|\gtrsim |c-1/2|\) and so \(I_{i_0}\) is empty, i.e.
  \[
  \frac{2 N\pi}{K}\int_{I_{i_0}}\sqrt{4-x^2}\, \dif x=0.
  \] 
  On the other hand, for \(c=1/2\) we have
  \[
  \frac{2 N\pi}{K}\int_{I_{i_0}}\sqrt{4-x^2}\, \dif x=\frac{4 N\pi}{K}|I_{i_0}|+\landauO*{\frac{K}{N}}=8\pi^2 +\landauO*{\frac{K}{N}},
  \]
  where we used that for \(c=1/2\) we have \(|I_{i_0}|=2\eta_{i_0}+\landauO{N^{-1}}\), as a consequence of \(\gamma_{i_0}=\landauO{N^{-1}}\), and that \(\rho(\ii \eta_{i_0})=\pi^{-1}+\landauO{\eta_{i_0}}\), with \(\eta_{i_0}\lesssim KN^{-1}\) in the bulk. By analogous computations we conclude that in the edge regime the r.h.s.\ of~\eqref{eq:mainvarco} is equal to 
  \(\bm1(\sigma=1)8\sqrt{2}\pi^2/3\).
  
  For mesoscopic scales the third and of the fourth term are negligible. This concludes the proof of Theorem~\ref{theo:sharpcom}.
  
\end{proof}

\subsection{Proof of the functional CLT for smooth cut-off}
\label{sec:smoco}

\begin{proof}[Proof of Theorem~\ref{theo:CLT}]
  
  For simplicity we present the proof in the macroscopic scale and in the mesoscopic scale in the bulk. The computation of the leading term and the estimate of the error terms at the edge are completely analogous and so omitted.
  
  For any \(z=x+\ii \eta\in\mathbf{C}\) we define the almost analytic extension of \(f\in H^2\) by
  \begin{equation}\label{eq:almanest}
    f_\mathbf{C}(z)=f_\mathbf{C}(x+\ii \eta):= \big[f(x)+\ii \eta \partial_x f(x)\big] \chi\big(N^a\eta),
  \end{equation}
  where \(\chi\) is a smooth cut-off equal to \(1\) for \(\eta\in [-5,5]\) and equal to \(0\) for \(\eta\in [-10,10]^c\). Note that
  \begin{equation}\label{eq:deranest}
    |\partial_{\overline{z}}f_\mathbf{C}|\lesssim N^{2a} |g''| |\eta|+N^a\big(|g|+N^a|g'||\eta|\big)|\chi'|,
  \end{equation}
  where \(2\partial_{\overline{z}}:=\partial_x+\ii\partial_\eta\).

  By  the Helffer-Sj\"ostrand formula we have that
  \begin{equation}\label{eq:betHS}
    f(\lambda)=\frac{1}{\pi}\int_\mathbf{C} \frac{\partial_{\overline{z}} f_\mathbf{C}(z)}{\lambda-z} \, \dif^2 z=\frac{2}{\pi}\Re\int_\mathbf{R}\int_\mathbf{R_+} \frac{\partial_{\overline{z}} f_\mathbf{C}(z)}{\lambda-z} \, \dif \eta \dif x,
  \end{equation}
  where \(\dif^2 z=\dif x\dif \eta\) is the Lebesgue measure on \(\mathbf{R}^2\). We recall the following notation
  \begin{align}
    L_N(f,I)&=\sum_{i=1}^N f(\lambda_i)-\E \sum_{i=1}^N f(\lambda_i) \\
    L_N(f,\mathring{A}):&=L_N(f,\mathring{A}_{\mathrm{d}})+L_N(f,A_{\mathrm{od}})=\sum_{i=1}^N f(\lambda_i)\braket{{\bm u}_i,\mathring{A}{\bm u}_i}.
  \end{align}
  
  
  Using~\eqref{eq:betHS}, we write
  \begin{equation}
    \begin{split}
      L_N(f,I)&=\frac{2N}{\pi}\Re\int_\mathbf{R}\int_\mathbf{R_+} \partial_{\overline{z}} f_\mathbf{C}(z) \braket{G(x+\ii \eta)-\E G(x+\ii \eta)} \, \dif \eta\dif x \\
      L_N(f,\mathring{A})&=\frac{2N^{1+a/2}}{\pi}\Re\int_\mathbf{R}\int_\mathbf{R_+} \partial_{\overline{z}} f_\mathbf{C}(z) \braket{G(x+\ii \eta)\mathring{A}} \, \dif \eta\dif x.
    \end{split}
  \end{equation} 
  
  Using that \(|\braket{{\bm u}_i, \mathring{A} {\bm u}_i}|\prec N^{-1/2}\) writing \(G\) in spectral decomposition, we conclude
  \begin{equation}\label{eq:bsmy}
    |\braket{G(x+\ii \eta)\mathring{A}}|\le\frac{1}{N}\sum_{i=1}^N\frac{|\braket{{\bm u}_i,\mathring{A}{\bm u}_i}|}{|\lambda_i-z|}\prec \frac{1}{\sqrt{N}}\left(1+\frac{1}{N\eta}\right)
  \end{equation}
  for any \(\eta\ge N^{-100}\), where we used that the local law for \(|\braket{G-\E G}|\prec (N\eta)^{-1}\) holds for any \(\eta\ge N^{-100}\) (e.g.\ see~\cite[Appendix A]{MR4221653}).
  
  
  Then, using~\eqref{eq:bsmy},~\eqref{eq:deranest}, and the local law  \(|\braket{G-\E G}|\prec (N\eta)^{-1}\), we readily conclude that
  \begin{equation}\label{eq:linstatsmrem}
    \begin{split}
      L_N(f,I)&=\frac{2N}{\pi}\Re\int_\mathbf{R}\int_{\eta_0}^{\eta_a} \partial_{\overline{z}} f_\mathbf{C}(z) \braket{G(x+\ii \eta)-\E G(x+\ii \eta)} \, \dif \eta\dif x+\landauOprec[\big]{\eta_0N^a} \\
      N^{a/2}L_N(f,\mathring{A})&=\frac{2N^{1+a/2}}{\pi}\Re\int_\mathbf{R}\int_{\eta_0}^{\eta_a} \partial_{\overline{z}} f_\mathbf{C}(z) \braket{G(x+\ii \eta)\mathring{A}} \, \dif \eta\dif x+\landauOprec[\big]{\eta_0^2N^{(1+3a)/2}},
    \end{split}
  \end{equation}
  where we defined
  \[
  \eta_0:=N^{-1+\epsilon}, \qquad \eta_a:= 10N^{-a},
  \]
  for some small \(\epsilon>0\) such that \(\eta_0\ll \eta_a\). Note that in~\eqref{eq:linstatsmrem} we used that \(\chi'(N^a\eta)=0\) on \(\eta\in [0, \eta_0]\) since \(\eta_0\ll N^{-a}\), and so that \(|\partial_{\overline{z}} f_\mathbf{C}|\lesssim N^{2a}\eta\) by~\eqref{eq:deranest}. We remark that the regime $\eta\le N^{-100}$ in \eqref{eq:linstatsmrem} is bounded trivially by $N^{-100+2a}$ using that $|\braket{GA}|\le \eta^{-1}$ and that \(|\partial_{\overline{z}} f_\mathbf{C}|\lesssim N^{2a}\eta\).
  
  With the formulas~\eqref{eq:linstatsmrem} we thus reduced the  proof of
  the functional CLT for general test function $f$ to 
  the CLT for resolvents as given in Theorem~\ref{CLT theorem}, modulo  detailed calculations
  of the leading terms. These calculations are deferred to Appendix~\ref{sec:proofclt}
  and with their help we 
  conclude the proof of Theorem~\ref{theo:CLT}. 
\end{proof}

\appendix

\section{Case of vanishing variances in Theorem~\ref{theo:CLT}}\label{sec:vanvar}

In this section we give a short explanation for the cases of vanishing variances in Theorem~\ref{theo:CLT} as listed in Remark~\ref{rem:vanish}.

The fact that for constant \(f\) the limiting processes vanish is obvious since in this case \(\Tr f(W)A\) is deterministic. Similarly, for linear \(f(x)=bx\) and \(w_2=0\) the diagonal processes \(\xi_\mathrm{tr},\xi_\mathrm{d}\) vanish since \(\Tr f(W)A_\mathrm{d}=b\Tr W A_\mathrm{d} =0\) almost surely if \(w_2=0\). For the case of quadratic \(f(x)=cx^2\) and \(\kappa_4=-1-\sigma^2\), i.e.\ \(\abs{w_{12}}=1/\sqrt{N}\) almost surely, we have 
\[\Tr W^2 A_\mathrm{d} = \sum_a (A_\mathrm{d})_{aa} \Bigl(\sum_b w_{ab}w_{ba}\Bigr) = \sum_a (A_\mathrm{d})_{aa} \Bigl( \frac{N-1}{N} + w_{aa}^2\Bigr) \]
almost surely,  so that \(\Var(\Tr W^2 A_\mathrm{d})\lesssim \braket{\abs{A_\mathrm{d}}^2}/N\). For \(\sigma=1\) and skew-symmetric \(A_\mathrm{od}=-A_\mathrm{od}^t\) is clear that \(2\Tr f(W) A_\mathrm{od} = \Tr f(W) A_\mathrm{od} + \Tr(f(W) A_\mathrm{od})^t = \Tr f(W)(A_\mathrm{od}+A_\mathrm{od}^t)=0\) due to \(W,f(W)\) being almost surely real-symmetric.

It remains to consider the cases of vanishing variances for \(\sigma=-1\), i.e.\ when \(W=D+\ii R\) for some real diagonal \(D\) and some real skew-symmetric \(R=-R^t\). 
If \(D=0\), then by  the exact symmetry of the spectrum, \(\lambda_i =-\lambda_{N+1-i}\)
and \({\bm u}_i = \ov{\bm u_{N+1-i}}\) (up to phase),  we immediately see that 
all three linear statistics are constant for odd functions \(f\). In case \(D\ne 0\)
the variance is not algebraically zero but it is vanishing for large \(N\). To illustrate this mechanism,
we consider the odd function \(\phi(x)=x^3\).  Then~\ref{xi tr equiv}--\ref{xi od equiv} in Remark~\ref{rem:vanish} 
are saying that \(\Tr (W^3-3 W)\), \(\Tr(W^3-2W)\mathring A_\mathrm{d}\) and \(\Tr W^3 A_\mathrm{od}\) fluctuate on a scale \(\ll1\). Indeed, 
\[ \Tr W^3 A = \Tr D^3 A + \ii \Tr (D^2 R + D R D + R D^2) A  - \Tr (D R^2 + R D R + R^2 D) A - \ii \Tr R^3 A   \]
so that 
\[\Tr (W^3-3W) =\Tr D^3 -3\Tr D(1+R^2) \]
is \(\ll 1\) since \((R^2)_{aa}=-\sum_{ab}R_{ab}^2\approx 1\). Similarly, 
\[\Tr (W^3 -2 W)A_\mathrm{d}=\Tr D^3 A_\mathrm{d} - \Tr (2D(1+R^2) + R D R ) A_\mathrm{d}  \]
since \(\Tr R^3 A_\mathrm{d}=0\) due to \(R=-R^t\) and \(A_\mathrm{d}=A_\mathrm{d}^t\) and the rhs.\ is \(\ll 1\) since 
\[\abs*{\Tr R D R A_\mathrm{d}}=\abs*{\sum_{ab} R_{ab}^2 D_{bb} (A_\mathrm{d})_{aa}} \lesssim N^{-1/2}\braket{\abs{A_\mathrm{d}}^2}.\]
Finally, 
\[ \Tr W^3 A_\mathrm{od} = \ii \Tr (D^2 R + D R D + R D^2) A_\mathrm{od}  - \Tr (D R^2 + R D R + R^2 D) A_\mathrm{od} - \ii \Tr R^3 A_\mathrm{od}  \]
is \(\lesssim N^{-1/2}\sqrt{A_\mathrm{od}A_\mathrm{od}^\ast}\) due to \(A_\mathrm{od}\) being off-diagonal. A similar argument works for any odd polynomial.

\section{Proofs of Lemmata~\ref{lem:remtra}--\ref{lem:impbneed} }\label{sec:addres}
    \begin{proof}[Proof of Lemma~\ref{lem:remtra}]
      We only prove~\eqref{eq:remdiag}, the proof of~\eqref{eq:remdiag3g} is analogous and so omitted.
      In order to prove~\eqref{eq:remdiag} we will often use the resolvent identity
      \begin{equation}\label{eq:resexp}
        G(z_i)=R(z_i)-R(z_i)DG(z_i).
      \end{equation}
      In particular, using~\eqref{eq:resexp} repeatedly, we find that
      \begin{equation}\label{eq:pertDr}
        \braket{G_1G_2^t}=\braket{R_1R_2^t}-\braket{R_1DR_1R_2^t}-\braket{R_1R_2^t DR_2^t}+\braket{R_1DR_1DR_1R_2^t}+\dots
      \end{equation}
      where we used the short-hand notation \(G_i=G(z_i)\), \(R_i=R(z_i)\). First we prove that we can stop the expansion~\eqref{eq:pertDr} after \(q=10\epsilon^{-1}\) resolvent identities for \(G_1\) and \(G_2\), keeping the last resolvent as \(G_i\), at the price of negligible error smaller than \(N^{-5}\) with very high probability. We now bound a representative term, with \(l\) factors \(R_1\) and \(m\) factors \(R_2\) and \(l+m=q\), to explain how the expansion in~\eqref{eq:pertDr} is truncated. By a Schwarz inequality we have
      \begin{equation}\label{eq:simpschwarz}
        \begin{split}
          &|\braket{R_1DR_1\dots R_1DG_1R_2^t DR_2^t\dots R_2^t DG_2^t}| \\
          &\quad\le \braket{R_1DR_1\dots R_1G_1G_1^*R_1^*\dots R_1^*DR_1^*}^{1/2} \braket{R_2DR_2\dots R_2G_2G_2^*R_2^*\dots R_2^*DR_2^*}^{1/2} \\
          &\quad\le \frac{1}{\eta_1^2\eta_2^2} \braket{\Im R_1DR_1\dots R_1D\Im R_1DR_1^*\dots R_1^*D}^{1/2} \braket{\Im R_2DR_2\dots R_2D\Im R_2DR_2^*\dots R_2^*D}^{1/2}.
        \end{split}
      \end{equation}
      Here we estimated \(\lVert G_i G_i^*\rVert \le \eta_i^{-2}\) and used the Ward identity \(R_i R_i^*=\eta_i^{-1}\Im R_i\). Note that to go from the first to the second line we used that \(\braket{R_2^t DR_2^t\dots (R_2^t)^*D(R_2^t)^*}=\braket{R_2DR_2\dots R_2^*DR_2^*}\) by cyclicity of the trace. In the following, to bound the terms in the rhs.\ of~\eqref{eq:simpschwarz} we write \(2\ii \Im R_i=R_i-R_i^*\), since we do not need to exploit any additional gain from \(\Im R_i\). For simplicity we assume that \(l\) is even and denote \(n=4(l-1)p\), then, denoting by \(\mathbf{E}_D\) the expectation with respect to the diagonal randomness \(D\) and using that the \(p\)-th moment of the entries of \(D\) are bounded  by \(C_p N^{-p/2}\) by~\eqref{eq:momentass}, we estimate
      \begin{equation}\label{eq:pcneed}
        \begin{split}
          &\E_D|\braket{ R_1DR_1\dots R_1D R_1DR_1^*\dots R_1^*D}|^{2p} \\
          &\qquad=\mathbf{E}_D\frac{1}{N^{2p}}\sum_{i_1,\dots, i_n\in [N]} D_{i_1i_1} \dots D_{i_n i_n} (R_1)_{ab}(R_1)_{cd}\dots (R_1)_{vw} \\
          &\qquad\le \frac{1}{N^{2p}} \frac{1}{N^{2(l-1)p}}N^{2(l-1)p} \left(\frac{1}{N\eta_1}\right)^{p(l-2)}=\frac{1}{N^{2p}}\left(\frac{1}{N\eta_1}\right)^{p(l-2)},
        \end{split}
      \end{equation}
      where the indices \(a,b,c,d,\dots,v,w\) are from the set \(\{i_1, i_2,\dots, i_n\}\), their precise allocation is irrelevant. In particular, in the second line of~\eqref{eq:pcneed} we neglected the fact the the summations involve \(2p\) traces and wrote directly the summation from \(i_1\) to \(i_n\), since the cyclic structure does not play any role in the estimate. The factors \(N^{-2(l-1)p}\) and \(N^{2(l-1)p}\) in~\eqref{eq:pcneed} come from the pairings in the expectation \(\mathbf{E}_D\) and from the number of effective summations, respectively. More precisely, in~\eqref{eq:pcneed} we used that the main contribution comes from the case when a \(D_{i_ji_j}\) is paired with some other \(D_{i_li_l}\), since in this way only half of the summations collapse. For higher moments the effective summation contains even less indices. The factor \((N\eta_1)^{-p(l-2)}\) comes from the fact that after the pairings at least \(p(l-2)\) factors \(R_1\) are off-diagonal gaining a factor \((N\eta_*)^{-1/2}\) for each one of them by the local law \(|(R_1)_{ab}|\prec (N\eta_1)^{-1/2}\). Combining~\eqref{eq:simpschwarz} and~\eqref{eq:pcneed} for any \(p\in\mathbf{N}\), we conclude that
      \[
      |\braket{R_1DR_1\dots R_1DG_1R_2^t DR_2^t\dots R_2^t DG_2^t}|\prec \frac{1}{N\eta_1^2\eta_2^2}\frac{1}{(N\eta_*)^{(q-4)/2}}\le N^{-5},
      \]
      since \(\eta_1,\eta_2\ge N^{-1+\epsilon}\) and \(q=10\epsilon^{-1}\). This implies that the expansion in~\eqref{eq:pertDr} can be stopped after a finite number of terms.
      
      
      In order to conclude the proof we have to estimate the fully expanded terms in~\eqref{eq:pertDr}, i.e.\ the terms which involve only \(R_1\), \(R_2\) and no \(G_1\), \(G_2\) appear. We now give two different bounds for the terms in~\eqref{eq:pertDr}: the first one in~\eqref{eq:firstgb} is better in the regime \(|z_1+z_2|\ge \eta^*\), the second one in~\eqref{eq:secondgb} is better in the opposite regime \(|z_1+z_2|< \eta^*\). To bound these terms we use the following strategy: Step (i) for each factor \(R_1R_2^t\) first perform a resolvent identity:
      \begin{equation}\label{eq:resr}
        R_1R_2^t=R(z_1)R(z_2)^t=-R(z_1)R(-z_2)=-\frac{R(z_1)-R(-z_2)}{z_1+z_2};
      \end{equation}
      Step (ii) compute high moments with respect to \(\mathbf{E}_D\). In particular, in each term in the expansion~\eqref{eq:pertDr} we can perform one resolvent identity if in the trace only one \(R_1\) or one \(R_2\) appears, and two resolvent identities otherwise. To make this argument clearer we bound explicitly a representative term, all the other terms are bounded exactly in the same way and so omitted. Using~\eqref{eq:resr}, we start with the bound
      \begin{equation}\label{eq:specterm}
        \begin{split}
          \mathbf{E}_D|\braket{R_1DR_1DR_1R_2^t}|^{2p}\lesssim \frac{1}{|z_1+z_2|^{2p}} \mathbf{E}_D\big[|&\braket{R(z_1)DR(z_1)DR(z_1)}|^{2p} \\
          &\qquad+|\braket{R(z_1)DR(z_1)DR(-z_2)}|^{2p}\big].
        \end{split}
      \end{equation}
      Note that transposes disappeared in the rhs.\ of~\eqref{eq:specterm}, and in the following we do not make distinction between \(R_1\) and \(R_2\), or their adjoints, every term is bounded in terms of \(\eta_*=\eta_1\wedge\eta_2\). The bound of the two terms in the rhs.\ of~\eqref{eq:specterm} is analogous and so we only consider the second one
      \begin{equation}\label{eq:bast}
        \begin{split}
          &\mathbf{E}_D|\braket{R(z_1)DR(z_1)DR(-z_2)}|^{2p} \\
          &\quad=\frac{1}{N^{2p}}\mathbf{E}_D\sum_{i_1, \dots, i_{4p}\in [N]} D_{i_1i_1}\dots D_{i_{4p}i_{4p}} R_{ab} (RR)_{cd}\dots, R_{tu} (RR)_{vw} \\
          &\quad\lesssim \frac{1}{N^{2p}}\frac{1}{N^{2p}} \sum_{a_1,\dots, a_{8p}\in \{i_1, \dots , i_{2p}\}\subset [N]^{2p}} R_{a_1a_2} (RR)_{a_3a_4}\dots R_{a_{2p-3}a_{2p-2}} (RR)_{a_{2p-1}a_{2p}} \\
          &\quad \prec \frac{1}{(N\eta_*)^{2p}},
        \end{split}
      \end{equation}
      where the indices \(a,b,c,d,\dots,t,u,v,w\) are from the set \(\{i_1, i_2,\dots, i_{4p}\}\), their precise allocation is irrelevant, and we used the notation \(R\in \{R(z_1),R(-z_2)\}\). Additionally, in~\eqref{eq:bast} we used that, as in~\eqref{eq:pcneed}, the leading order contribution to \(\mathbf{E}_D\) comes from the pairings and so that from the second to the third line of~\eqref{eq:bast} \(2p\) indices collapse. To go from the third to the fourth line we used that there are exactly \(2p\) factors \((RR)_{ab}\), and that \(|(RR)_{ab}|\prec \eta_*^{-1}\) and \(|R_{ab}|\prec 1\). Finally, combining~\eqref{eq:specterm} and~\eqref{eq:bast}, we conclude that
      \begin{equation}\label{eq:firstgb}
        |\braket{R_1DR_1DR_1R_2^t}|\prec \frac{1}{N\eta_*|z_1+z_2|}.
      \end{equation}
      Using that an analogous bound holds for all the other terms in the expansion~\eqref{eq:pertDr}, this concludes the proof of~\eqref{eq:remdiag}.
      
      We now prove another bound for the terms~\eqref{eq:pertDr} which improves~\eqref{eq:firstgb} in the regime \(|z_1+z_2|<\eta^*\). Similarly to~\eqref{eq:specterm}, also in this case we bound a representative term of the expansion~\eqref{eq:pertDr}:
      \begin{equation}
        \label{eq:secondgb}
        \begin{split}
          &\mathbf{E}_D|\braket{R_1DR_1DR_1R_2^t}|^{2p} \\
          &=\frac{1}{N^{2p}}\mathbf{E}_D\sum_{i_1, \dots, i_{4p}\in [N]} D_{i_1i_1}\dots D_{i_{4p}i_{4p}} (R_1)_{ab} (R_1R_2^t R_1)_{cd}\dots, (R_1)_{tu} (R_1R_2^t R_1)_{vw} \\
          &\lesssim \frac{1}{N^{4p}}\sum_{\substack{a_1,\dots, a_{8p}\in \{i_1, \dots , i_{2p}\} \\ \{i_1, \dots , i_{2p}\}\subset [N]^{2p}}} (R_1)_{a_1a_2} (R_1R_2^t R_1)_{a_3a_4}\dots (R_1)_{a_{2p-3}a_{2p-2}} (R_1R_2^t R_1)_{a_{2p-1}a_{2p}} \\
          &\prec \frac{1}{(N\eta_1\eta_2)^p},
        \end{split}
      \end{equation}
      where the indices \(a,b,c,d,\dots,t,u,v,w\) are as in~\eqref{eq:bast}, their precise allocation is irrelevant. Additionally, as in~\eqref{eq:bast} we used that the leading order contribution to \(\mathbf{E}_D\) comes from the pairings so that from the second to the third line of~\eqref{eq:secondgb} \(2p\) indices collapse. To go from the third to the fourth line we used that there are exactly \(2p\) factors \((R_1R_2^t R_1)_{ab}\), and that \(|(R_1R_2^t R_1)_{ab}|\prec (\eta_1\eta_2)^{-1}\) and \(|(R_1)_{ab}|\prec 1\).
    \end{proof}
    
    \begin{proof}[Proof of Lemma~\ref{lem:impbneed}]
      Since both the l.h.s.\ and r.h.s.\ of~\eqref{eq:perch} are symmetric in \(\eta_1\) and \(\eta_2\), without loss of generality we assume that \(\eta_*=\eta_1\). Recalling that \(L=N(\rho\eta)_*\), with \((\rho\eta)_*:=(\rho_1\eta_1)\wedge (\rho_2\eta_2)\), it readily follows that
      \[
      \frac{1}{N\eta_1^2|z_1-z_2|}\lesssim \frac{\rho_1}{L\eta_1\eta_2},
      \]
      where we used that \(|z_1-z_2|\gtrsim \eta_1+\eta_2\). Hence if \(\rho_*=\rho_1\) there is nothing else to prove. In the remainder of the proof we assume that \(\rho_*=\rho_2\). Using the definition of \(L\), we readily see that~\eqref{eq:perch} is equivalent to
      \begin{equation}\label{eq:neednow2}
        (\eta\rho)_*\eta_2\lesssim \rho_2|z_1-z_2|\eta_1.
      \end{equation}
      We now distinguish two cases: (i) \(\rho_2\le |z_1-z_2|^{1/2}\), (ii) \(\rho_2>|z_1-z_2|^{1/2}\). If \(\rho_2\le |z_1-z_2|^{1/2}\) then we have that
      \begin{equation}\label{eq:case1}
        \rho_1=\rho_2+\landauO{|z_1-z_2|^{1/2}}=\landauO{|z_1-z_2|^{1/2}},
      \end{equation}
      by the \(1/2\)-H\"older continuity of the density. Using~\eqref{eq:case1}, we readily conclude that
      \[
      (\eta\rho)_*\eta_2\lesssim |z_1-z_2|^{1/2} \eta_*\eta_2\le \rho_2 |z_1-z_2|\eta_1,
      \]
      by \(\eta_*=\eta_1\) and \(\eta_2\lesssim \rho_2\sqrt{\eta_2}\lesssim \rho_2 |z_1-z_2|^{1/2}\), using \(\rho_2\gtrsim \sqrt{\eta_2}\) and \(\eta_2\le |z_1-z_2|\). This proves~\eqref{eq:neednow2} in this case.

      In the opposite case \(\rho_2>|z_1-z_2|^{1/2}\), using again the H\"older continuity of the density, we have that
      \[
      \rho_1=\rho_2+\landauO{|z_1-z_2|^{1/2}}\lesssim \rho_2.
      \]
      Hence we conclude that
      \[
      (\eta\rho)_*\eta_2\lesssim \eta_*\rho_2\eta_2\le \rho_2|z_1-z_2|\eta_1,
      \]
      where we used \(\eta_*=\eta_1\) and \(\eta_2\le \eta^*\lesssim |z_1-z_2|\), proving~\eqref{eq:neednow2} in this case as well. 
    \end{proof}

  \section{Proof of remaining estimates for Proposition~\ref{pro:3g}}
  \label{addproofs34}
  
    \begin{proof}[Proof of the local laws and bound with transposes in Proposition~\ref{pro:3g}]
    
    Using the bound for \(\braket{G_1^t G_2A}\) in~\eqref{eq:avll2G4} as an input, and the bound
    \[
    \braket{\underline{WG_2G_2G_1^t}A}=\landauOstd*{\frac{\sqrt{\rho_1\rho_2}}{\sqrt{NK\eta_*\eta_1}\eta_2}}
    \]
    from~\eqref{eq:bettunder2}, we prove the bound for \(\braket{G_1^t G_2G_2A}\) in~\eqref{eq:avll3G3} analogously to the proof of the bound for \(\braket{G_1G_2G_2A}\) above.
     
    Using the local laws for \(\braket{G_1G_2^t}\), \(\braket{G_1AG_2^t A}\) in~\eqref{eq:avll2G2} and~\eqref{eq:avll2G22}, respectively, and the bound
    \[
    |\braket{\underline{WG_1G_2^t AG_1A'}}|=\landauOstd*{ \frac{\sqrt{\rho_1\rho_2}}{\sqrt{K\eta_1\eta_2}}}
    \]
    from~\eqref{eq:bettunder3} as an input, the proof of the local law for \(\braket{G_2^t AG_1A'G_1}\) in~\eqref{eq:avll3G21} follows exactly in the same way as the proof of the local law for \(\braket{G_2AG_1A'G_1}\) above. 
    
    We are now only left with the proof of the local law for \(\braket{G_1^t G_2G_2}\) in~\eqref{eq:avll3G2}. The proof of the local law
    \begin{equation}\label{eq:neednow}
      \braket{G_1^t G_2G_2}=\frac{m_1m_2'}{(1-\sigma m_1m_2)^2}+\landauOprec*{\frac{1}{N\eta_2^2\eta_1}},
    \end{equation}
    for \(\sigma=\pm 1\) requires exactly the same changes to the proof of the local law for \(\braket{G_1G_2G_2}\) that the proof of \(\braket{G_1^t G_2}\) required to the proof \(\braket{G_1G_2}\). For this reason we omit the details of this proof.
    
    In the remainder we will prove that
    \begin{equation}\label{eq:lastremcas}
      \braket{G_1^t G_2G_2}=\frac{m_1m_2'}{(1-\sigma m_1m_2)^2}+\landauOprec*{\frac{\rho_1}{\sqrt{L}\eta_1\eta_2}}.
    \end{equation}
    The proof of~\eqref{eq:lastremcas} will be divided into three cases: (i) \(\sigma=1\), (ii) \(\sigma=-1\), (iii) \(|\sigma|<1\). The case \(\sigma=1\) is trivial since \(G_1^t=G_1\), hence~\eqref{eq:lastremcas} follows by the local law for \(\braket{G_1G_2G_2}\) in~\eqref{eq:avll3G1}. When \(\sigma=-1\) we can write \(W=D+\ii O\), with \(D\) a diagonal matrix and \(O\) being an skew-symmetric matrix. We now consider the regimes (i) when either \(\Im z_1\Im z_2>0\) or \(\Im z_1\Im z_2<0\) and \(|z_1+z_2|\ge \eta^*\), (ii) \(\Im z_1\Im z_2<0\) and \(|z_1+z_2|<\eta^*\). For the case (ii) the local law in~\eqref{eq:neednow}, together with~\eqref{eq:remdiag3g} the bound~\eqref{eq:impneeddr}, gives~\eqref{eq:lastremcas}. In the regime (i) we consider \(R(z_i):=(\ii O-z_i)^{-1}\), then using Lemma~\ref{lem:remtra} we have
    \begin{equation}\label{eq:1}
      \begin{split}
        \braket{G(z_1)^t G(z_2)G(z_2)}&=\braket{R(z_1)^t R(z_2)R(z_2)}+\landauOprec*{\frac{1}{N\eta_*\eta_2|z_1+z_2|}} \\
        &=-\braket{R(-z_1)R(z_2)R(z_2)}+\landauOprec*{\frac{1}{N\eta_*\eta_2|z_1+z_2|}},
      \end{split}
    \end{equation}
    where we used that \(R(z_1)^t=-R(-z_1)\). The proof of the local law
    \begin{equation}\label{eq:2}
      \begin{split}
        -\braket{R(-z_1)R(z_2)R(z_2)}&=-\frac{m(-z_1)m(z_2)'}{(1+m(-z_1)m(z_2))^2}+\landauOprec*{\frac{1}{N\eta_*\eta_2|z_1+z_2|}} \\
        &=\frac{m(z_1)m(z_2)'}{(1-m(z_1)m(z_2))^2}+\landauOprec*{\frac{1}{N\eta_*\eta_2|z_1+z_2|}},
      \end{split}
    \end{equation}
    where we used \(m(-z_1)=-m(z_1)\), follows exactly as the proof of the local law for \(\braket{G_1G_2G_2}\) in~\eqref{eq:resid3g}--\eqref{eq:impneeddr} and so omitted. Combining~\eqref{eq:1} and~\eqref{eq:2} we get
    \[
    \braket{G_1^t G_2G_2}=\frac{m(z_1)m(z_2)'}{(1-m(z_1)m(z_2))^2}+\landauOprec*{\frac{1}{N\eta_*\eta_2|z_1+z_2|}},
    \]
    which, together with the bound
    \[
    \frac{1}{N\eta_*\eta_2|z_1+z_2|}\lesssim \frac{\rho_*}{L\eta_1\eta_2}
    \]
    from Lemma~\ref{lem:impbneed}, concludes the proof of the local law for \(\braket{G_1^t G_2G_2}\) in~\eqref{eq:avll3G2} for \(\sigma=-1\).
    
    In order to conclude the proof of Proposition~\ref{pro:3g} we are now left with the proof of the local law for \(\braket{G_1^t G_2G_2}\) in~\eqref{eq:avll3G2} when \(|\sigma|<1\). In this last remaining case, similarly to~\eqref{eq:Gimpboh}, we use the equation for \(G_1^t G_2G_2\) and conclude that
    \begin{equation}\label{eq:transpeq}
      \begin{split}
        \left[1+\landauOprec*{\frac{1}{N\eta_*}}\right](1-\sigma m_1m_2)\braket{G_1^t G_2G_2}&=m_1\braket{G_2G_2}+\sigma m_1\braket{G_1^t G_2}\braket{G_2G_2} \\
        &\quad -m_1\braket{\underline{WG_1^t G_2G_2}}+\landauOprec*{\frac{\sqrt{\rho_1\rho_2}}{\sqrt{L\eta_1\eta_2}}}.
      \end{split}
    \end{equation}
    The bound of the error term in~\eqref{eq:transpeq} follows by a Schwarz inequality and~\cite[Lemma~\ref{eth-degree two lemma}]{2012.13215}, similarly to~\eqref{eq:preb} above.  Then, using the local laws for \(\braket{G_2G_2}\), \(\braket{G_1^t G_2}\) in~\eqref{eq:avll2G1} and~\eqref{eq:avll2G2}, respectively, we conclude
    \[
    \begin{split}
      \left[1+\landauOprec*{\frac{1}{N\eta_*}}\right](1-\sigma m_1m_2)\braket{G_1^t G_2G_2}&=\frac{m_1m_2'}{(1-\sigma m_1m_2)^2}+\landauOprec*{\frac{\sqrt{\rho_1\rho_2}}{\sqrt{L\eta_1\eta_2}}+\frac{\rho^*}{\sqrt{L}\eta_2}} \\
      &=\frac{m_1m_2'}{(1-\sigma m_1m_2)^2}+\landauOprec*{\frac{\rho_1}{\sqrt{L}\eta_1\eta_2}}.
    \end{split}
    \]
    This concludes the proof of the local law for \(\braket{G_1^t G_2G_2}\) for the last remaining case \(|\sigma|<1\).
  \end{proof}

    \section{Proof of the refined bounds on renormalised alternating chains: Theorem~\ref{general chain G underline theorem} and Lemma~\ref{lem:impvarbclt}}\label{aux underline bounds} 
    \subsection{Proof of Theorem~\ref{general chain G underline theorem}}\label{sec gen underline}   
    In~\cite{2012.13215} we proved a general high probability bound~\cite[Theorem~\ref{eth-chain G underline theorem}]{2012.13215} on renormalized alternating chains of resolvents and deterministic matrices. The proof of~\cite[Theorem~\ref{eth-chain G underline theorem}]{2012.13215} proceeded by a delicate power counting of high moments \(\E\braket{\un{WG_1B_1\cdots G_lB_l}}^{2p}\) after iterated cumulant expansions. For this appendix we assume that the reader is familiar with the proof of~\cite[Theorem~\ref{eth-chain G underline theorem}]{2012.13215} and here we only explain the two minor modifications that are necessary for the current setup.
    The proof of Theorem~\ref{general chain G underline theorem} follows directly from the proof of~\cite[Theorem~\ref{eth-chain G underline theorem}]{2012.13215} by replacing each \(G_k\) in the formulas~\cite[Eq.~\eqref{eth-eq:under B}--\eqref{eth-eq large eta bounds}]{2012.13215} by \(G_{k,1}\cdots G_{k,n_k}\). 
    
    \begin{enumerate}[label=(mod-\arabic*)]
      \item\label{mod1}\cite[Lemma~\ref{eth-degree two lemma}]{2012.13215} remains valid if on the lhs.\ of~\cite[Eq.~\eqref{eth-Improved 2 reduced edge est}--\eqref{eth-Improved 2 reduced edge est iso}]{2012.13215} some \(G_k\)'s are replaced by \(G_{k,1}\cdots G_{k,n_k}\) and the rhs.\ are multiplied by factors of \(\min_i \eta_{k,i}/\prod_i \eta_{k,i}\) for each replacement.
      \item\label{mod2}\cite[Lemmata~\ref{eth-lemma Itwo}--\ref{eth-lemma Ithree}]{2012.13215} hold true under the setting of Theorem~\ref{general chain G underline theorem} upon multiplying the rhs.\ of~\cite[Eq.~\eqref{eth-Itwo eq def}--\eqref{eth-deg 3 lemma iso case}]{2012.13215} by 
      \begin{equation}\label{gen bound factor}
        \prod_{k\in [l]} \Bigl(\frac{\min_i \eta_{k,i}}{\prod_i \eta_{k,i}}\Bigr)^{2p}.
      \end{equation} 
    \end{enumerate} 
    From~\ref{mod2} the proof of Theorem~\ref{general chain G underline theorem} can be concluded exactly as the one of~\cite[Theorem~\ref{eth-chain G underline theorem}]{2012.13215}.
    
    Regarding~\ref{mod1} recall that in the proof of~\cite[Lemma~\ref{eth-degree two lemma}]{2012.13215} we used three types of estimates on resolvents; (i) the naive isotropic bound \(\abs{G_{\vx\vy}}\prec 1\), (ii) the Ward improvement \((GG^\ast)_{\vx\vx}\prec \rho^\ast/\eta_\ast\) and (iii) the norm bound \(\norm{G}\le 1/\eta_\ast\). The norm bound is obviously compatible with the replacement by sub-multiplicativity of the norm. For the naive isotropic bound we obtain 
    \begin{equation}\label{G1..k new naive}
      \abs{(G_1\cdots G_n)_{\vx\vy}}\prec \frac{\min_{i}\eta_i}{\prod_{i}\eta_i}.
    \end{equation}
    by spectral decomposition. Finally, for the Ward improvement we have 
    \begin{equation*}
      [(G_1\cdots G_n)(G_1\cdots G_n)^\ast]_{\vx\vx} =\frac{(G_1\cdots G_{n-1}\Im G_n G_{n-1}^\ast \cdots G_1^\ast )}{\eta_n} \prec \frac{\rho^\ast \eta_\ast}{\prod_i \eta_i^2} = \frac{\rho^\ast}{\eta_\ast} \Bigl(\frac{\eta_\ast}{\prod_i\eta_i}\Bigr)^2,
    \end{equation*}
    hence also the Ward improvement is compatible with the replacement. 
    
    Regarding~\ref{mod2}, note that before the cumulant expansions there are \(2p\) copies of each of \(G_{k,1}\cdots G_{k,n_k}\) for \(k=1,\ldots,l\). However, along the cumulant expansions the initial products of \(G\)'s may be broken up into two or more shorter chains upon differentiation, similarly as a single \(G_{\vx\vy}\) becomes \((G \Delta^{ab} G)_{\vx\vy}  = G_{\vx a} G_{b\vy}\). The key additional observation is that the bound~\eqref{G1..k new naive} is compatible with the cumulant expansion in the sense that the bound cannot increase upon differentiation. For example, ``differentiating'' the bound 
    \[ \abs{(G_1 G_2 G_3)_{\vx\vy}} \prec \frac{\min_{i\in\set{1,2,3}}\eta_i}{\prod_{i\in\set{1,2,3}}\eta_i}\]
    we get 
    \[
    \begin{split}
    \abs{(G_1 G_2 \Delta^{ab}G_2 G_3)_{\vx\vy}} =\abs{(G_1 G_2)_{\vx a}} \abs{(G_2 G_3)_{b\vy}} &\prec  \frac{\min_{i\in\set{1,2}}\eta_i}{\prod_{i\in\set{1,2}}\eta_i}\frac{\min_{i\in\set{2,3}}\eta_i}{\prod_{i\in\set{2,3}}\eta_i} \\
    &\le\frac{\min_{i\in\set{1,2,3}}\eta_i}{\prod_{i\in\set{1,2,3}}\eta_i}.
    \end{split}
    \]
    Thus the additional factor~\eqref{gen bound factor} is an upper bound on the product of additional factors obtained from each application of~\cite[Lemma~\ref{eth-degree two lemma}]{2012.13215} in the proofs of~\cite[Lemmata~\ref{eth-lemma Itwo}--\ref{eth-lemma Ithree}]{2012.13215}.
    
    \subsection{Proof of Lemma~\ref{lem:impvarbclt}}\label{sec:impgag}
    We start with the following bound on \(\braket{{\bm x}, \underline{GAWG}{\bm y}}\):
    \begin{equation}\label{eq:mainisogon}
      \sqrt{\mathbf{E}|\braket{{\bm x}, \underline{GAWG}{\bm y}}|^2}\prec\frac{\rho}{\sqrt{N}\eta}, \qquad \eta\lesssim 1.
    \end{equation}
    Note that by the bound by~\eqref{eq:isll2G} it follows that $\sqrt{\mathbf{E}|\braket{{\bm x}, \underline{GAWG}{\bm y}}|^2}\prec \rho^{1/2}\eta^{-1/2}$
    In particular, we need an additional \(\sqrt{\rho/(N\eta)}\) gain compared to this bound. To do so we will use~\eqref{eq:isll2G} as an input to prove a better bound~\eqref{eq:mainisogon} on \(\sqrt{\mathbf{E}|\braket{{\bm x}, \underline{GAWG}{\bm y}}|^2}\).
    
    
    \begin{proof}[Proof of~\eqref{eq:mainisogon}]
      From now on, without loss of generality, we assume that  \(A=A^*\). To make the presentation cleaner we assume that \(w_2=1+\sigma\), the general case \(w_2\ne 1+\sigma\) is completely analogous and so omitted, since the difference compared to the case presented here is only in the case \(b=a\) in the summations that is of lower order.
      
      Using cumulant expansion, and the notation \({\bm \alpha}:=(\alpha_1,\dots,\alpha_k)\) we start computing
      \begin{equation}\label{eq:firstcum}
        \begin{split}
          &\mathbf{E}\braket{{\bm x}, \underline{GAWG}{\bm y}}\braket{ {\bm y},\underline{G^*WAG^*} {\bm x}}\\
          &=\frac{1}{N}\mathbf{E}\sum_{ab}\braket{{\bm x}, GA\Delta^{ab}G{\bm y}}\left(\braket{ {\bm y}, G^*\Delta^{ba} AG^*{\bm x}}-\braket{{\bm y}, \underline{G^*\Delta^{ba}G^*WAG^*} {\bm x}}\right) +\dots\\
          &\quad+\frac{\sigma}{N}\mathbf{E}\sum_{ab}\braket{{\bm x}, GA\Delta^{ab}G{\bm y}}\left(\braket{ {\bm y}, G^*\Delta^{ab} AG^* {\bm x}}-\braket{{\bm y}, \underline{G^*\Delta^{ab}G^*WAG^*} {\bm x}}\right) +\dots\\
          &\quad+\sum_{k\ge 2}\sum_{ab}\sum_{{\bm \alpha}\in \{ab,ba\}^k}\frac{\kappa(ab,{\bm \alpha})}{k!}\mathbf{E}\partial_{\bm \alpha}\left[\braket{{\bm x}, GA\Delta^{ab}G{\bm y}}\braket{{\bm y}, \underline{G^*WAG^*} {\bm x}}\right].
        \end{split}
      \end{equation}
      The dots in the second and third line of~\eqref{eq:firstcum} denote the fact that we considered only the case when the derivative hits the first \(G^*\) in \(\braket{ {\bm y},\underline{G^*WAG^*} {\bm x}}\). The case when the derivative hits the second \(G^*\) is completely analogous and so omitted. We will often use this notation in the remainder of the proof to denote that we consider only some representative terms and all the others are estimated analogously.
      
      In order to perform cumulant expansion in the remaining underlined term in the last line of~\eqref{eq:firstcum} we consider the case when either one of the \({\bm \alpha}\)-derivative hits the \(W\) or all the derivatives hit a \(G^*\). Without loss of generality we denote this derivative by \(\partial_{\alpha_1}\) and we use the notation \({\bm \alpha}_2:=(\alpha_2,\dots,\alpha_k)\). Then, performing cumulant expansion in the last line of~\eqref{eq:firstcum}, we get
      \begin{equation}\label{eq:seconcum}
        \begin{split}
          &\mathbf{E}\braket{{\bm x}, \underline{GAWG}{\bm y}}\braket{ {\bm y},\underline{G^*WAG^*} {\bm x}}\\
          &=\frac{1}{N}\E(GG^*)_{{\bm y}{\bm y}}(GA^2G^*)_{{\bm x}{\bm x}}+\frac{\sigma}{N}\E(G^t A^t G^*)_{{\bm y}{\bm x}}(G^t AG^*)_{{\bm y}{\bm x}} \\
          &+\frac{1}{N^2}\E\sum_{abcd}\left(\braket{{\bm x}, G\Delta^{dc}GA\Delta^{ab}G{\bm y}}+\braket{{\bm x}, GA\Delta^{ab}G\Delta^{dc}G{\bm y}}\right)\braket{{\bm y}, G^*\Delta^{ba}G^*\Delta^{cd}AG^* {\bm x}}+\dots \\
          &+\frac{\sigma}{N^2}\E\sum_{abcd}\left(\braket{{\bm x}, G\Delta^{cd}GA\Delta^{ab}G{\bm y}}+\braket{{\bm x}, GA\Delta^{ab}G\Delta^{cd}G{\bm y}}\right)\braket{{\bm y}, G^*\Delta^{ba}G^*\Delta^{cd}AG^* {\bm x}}+\dots \\
          &+\frac{\sigma^2}{N^2}\E\sum_{abcd}\left(\braket{{\bm x}, G\Delta^{cd}GA\Delta^{ab}G{\bm y}}+\braket{{\bm x}, GA\Delta^{ab}G\Delta^{cd}G{\bm y}}\right)\braket{{\bm y}, G^*\Delta^{ab}G^*\Delta^{cd}AG^* {\bm x}}+\dots \\
          &-\frac{1}{N}\sum_{l\ge 2}\sum_{abcd}\sum_{{\bm \beta}\in \{cd,dc\}^l}\frac{\kappa(cd,{\bm \beta})}{l!} \partial_{\bm\beta}\E \left[\braket{{\bm x}, GA\Delta^{ab}G{\bm y}}\braket{{\bm y}, G^*\Delta^{ba}G^*\Delta^{cd}AG^* {\bm x}}\right]+\dots \\
          &+\sum_{k\ge 2}\sum_{ab}\sum_{\substack{{\bm \alpha}_2\in \{ab,ba\}^{k-1}, \\ \alpha_1\in \{ab,ba\}}}\frac{\kappa(ab,\alpha_1,{\bm \alpha}_2)}{k!}\mathbf{E}\partial_{\bm \alpha}\left[\braket{{\bm x}, GA\Delta^{ab}G{\bm y}}\braket{ {\bm y}, G^*\Delta^{\alpha_1}AG^*{\bm x}}\right]+\dots \\
          &+\sum_{k,l\ge 2}\sum_{abcd}\sum_{\substack{(\alpha_1,{\bm \alpha}_2)\in \{ab,ba\}^k, \\ {\bm \beta}\in \{cd,dc\}^l}}\frac{\kappa(ab,\alpha_1,{\bm \alpha}_2)\kappa(cd,{\bm \beta})}{k!l!} \\
          &\qquad\qquad\qquad\quad\times \mathbf{E}\partial_{{\bm \alpha}_2}\partial_{\bm \beta}\left[\braket{{\bm x}, GA\Delta^{ab}G{\bm y}}\braket{{\bm y}, G^*\Delta^{\alpha_1}G^*\Delta^{cd}AG^* {\bm x}}\right]+\cdots
        \end{split}
      \end{equation}
      Here we omitted the third or higher order terms coming from the cumulant expansions in the third and fourth line of~\eqref{eq:firstcum}, since they are estimated exactly as the ones in the sixth and seventh line of~\eqref{eq:seconcum}.
      
      
      We now first estimate the last three lines of~\eqref{eq:seconcum}. We rewrite the derivative in the third to last line together with the \(a\)-summation (which can be performed as a matrix product) as 
      \[ \partial_{\bm\beta}\Bigl[ (GAG^\ast)_{\vx c} G_{b\vy} G^\ast_{\vy b} (AG^\ast)_{d\vx} \Bigr]\]
      and notice that irrespective of \(\bm\beta\) and the allocation of derivatives we have one factor of \(GAG^\ast\), one factor of \(G\) evaluated in one of \(b\vy,bc,bd\), one factor of \(G^\ast\) evaluated in one of \(\vy b,cb,db\) and one factor of \(AG^\ast\) or \(G^\ast\) evaluated in one of \(d\vx,c\vx\). Thus, in any case three Ward estimates can be performed (two in the \(b\)-index and one in the \(c\)- or \(d\)-index) and together with the estimate \(\abs{(GAG^\ast)_{\vx\vy}}\prec\sqrt{\rho/\eta}\) and the scaling \(\abs{\kappa(cd,\bm\beta)}\lesssim N^{-(l+1)/2}\) we obtain a bound of at most
      \[
      N^{3-(l+3)/2} \sqrt{\frac{\rho}{\eta}}\Bigl(\frac{\rho}{N\eta}\Bigr)^{3/2} = \frac{\rho^2}{N^{l/2}\eta^2}\le \frac{\rho^2}{N\eta^2}.
      \]
      For the second to last and last line of~\eqref{eq:seconcum} the argument is considerably simpler since a simple counting of available Ward estimates suffices. For the \(k\)-th term of the penultimate line of~\eqref{eq:seconcum} the naive size is given by \(N^{2-(k+1)/2}\) (i.e.\ using \(\abs{G_{xy}}\prec 1\) for all factors), while for each derivative allocation at least four resolvent entries can be estimated via the Ward identity, hence gaining a factor of \(\rho^2/(N\eta)^2\), yielding \(\rho^2N^{-(k+1)/2}\eta^{-2}\). For the \(k\)-th term of the last line of~\eqref{eq:seconcum} the naive size is given by \(N^{3-(k+l)/2}\), while for each derivative allocation at least five resolvent entries can be estimated via the Ward identity, hence gaining a factor of \(\rho^{5/2}/(N\eta)^{5/2}\), yielding \(N^{-(k+l-1)/2}\rho^{5/2} \eta^{-5/2}\). Thus, we finally conclude that the last three lines of~\eqref{eq:seconcum} are bounded by \(\rho^2 N^{-1}\eta^{-2}\).
      
      In order to conclude the proof of~\eqref{eq:mainisogon} we are left only with the second order terms, i.e.\ we are left with
      \begin{equation}\label{eq:thirdcum}
        \begin{split}
          &\mathbf{E}\braket{{\bm x}, \underline{GAWG}{\bm y}}\braket{ {\bm y},\underline{G^*WAG^*} {\bm x}}\\
          &=\frac{\E(GG^*)_{{\bm y}{\bm y}}(GA^2G^*)_{{\bm x}{\bm x}}}{N}+\frac{\sigma}{N}\E(G^t A^t G^*)_{{\bm y}{\bm x}}(G^t AG^*)_{{\bm y}{\bm x}}+\frac{\mathbf{E}(GAG^*G)_{{\bm x}{\bm y}}(G^*GAG^*)_{{\bm y}{\bm x}}}{N^2} \\
          &\quad+\frac{1}{N}\E(G^*G)_{{\bm y}{\bm y}}\braket{GAG^*}(GAG^*)_{{\bm x}{\bm x}}+\frac{\sigma}{N^2}\E(G(G^*)^t A^t G^t AG^*)_{{\bm x}{\bm x}}(G^*G)_{{\bm y}{\bm y}} \\
          &\quad+\frac{\sigma^2}{N}\E\braket{GG^*}(G^*A^t G^t)_{{\bm y}{\bm x}}(G^t AG^*)_{{\bm y}{\bm x}} +\frac{\sigma}{N^2}\E(GAG^*G^t(G^*)^t)_{{\bm x}{\bm y}}(G^t AG^*)_{{\bm y}{\bm x}} \\
          &\quad+\frac{\sigma^2}{N^2}\E(G^*A^t G^t AG^*)_{{\bm y}{\bm x}}(G^t G^*G^t)_{{\bm y}{\bm x}}+\landauO*{\frac{\rho^2N^\epsilon}{N\eta^2}}.
        \end{split}
      \end{equation}
      
      for any $\epsilon>0$. In order to conclude~\eqref{eq:mainisogon} we need to show that all the terms in~\eqref{eq:thirdcum} are bounded by \(\rho^2/(N\eta^2)\). For this purpose we need an additional bound which improves the bound in~\cite[Lemma~\ref{eth-degree two lemma}]{2012.13215}. We present this bound in the following lemma, which will be proven at the end of this section.
      
      \begin{lemma}\label{lem:addbb}
        We have the bound
        \begin{equation}\label{eq:adbb}
          |(G^*A^t G^t AG^*)_{{\bm y}{\bm x}}|\prec \sqrt{\frac{N\rho}{\eta}}.
        \end{equation}
      \end{lemma}
      
      
      Note that by a Schwarz inequality (after the first \(G\)) and the bound in~\cite[Lemma~\ref{eth-degree two lemma}]{2012.13215} with \(\Lambda_+\lesssim 1\), we conclude that
      \begin{equation}\label{eq:bbis1}
        |(G(G^*)^t A^t G^t AG^*)_{{\bm x}{\bm x}}|\prec \frac{N\rho}{\eta}.
      \end{equation}
      We also have that
      \begin{equation}\label{eq:bbis2}
        |(GAG^*G^t(G^*)^t)_{{\bm x}{\bm y}}|\le \frac{1}{\eta} (GA\Im G AG^*)_{{\bm x}{\bm x}}^{1/2} (G^t\Im G^t(G^*)^t)_{{\bm y}{\bm y}}^{1/2}\prec \frac{\rho\sqrt{N}}{\eta^2},
      \end{equation}
      that \((GCG^*)\le \norm{C}\Im G/\eta\), for any matrix \(C=C^*\), and that \(\norm{\Im G}\le 1/\eta\). Inserting the bounds~\eqref{eq:adbb},~\eqref{eq:bbis1}, and~\eqref{eq:bbis2} into~\eqref{eq:thirdcum} we conclude that
      \begin{equation}\label{eq:rem}
        \mathbf{E}\braket{{\bm x}, \underline{GAWG}{\bm y}}\braket{ {\bm y},\underline{G^*WAG^*} {\bm x}}=\landauO*{\frac{N^\epsilon\rho^2}{N\eta^2}},
      \end{equation}
      for any $\epsilon>0$. This concludes the proof of~\eqref{eq:mainisogon}.
      
    \end{proof}
    
    The proof of the bounds in~\eqref{eq:bettunder1}--\eqref{eq:bettunder3} is very similar to the one presented above, hence we will present the estimate only of some representative terms. In particular, the higher order cumulants can be estimated as in the bound on the last three lines of~\eqref{eq:seconcum} above and we therefore here only consider the more critical second order cumulant terms. Also note that the bounds~\eqref{eq:bettunder1}--\eqref{eq:bettunder3} follow directly from~\cite[Theorem~\ref{eth-chain G underline theorem}]{2012.13215} in case when \(\rho_1,\rho_2\) and \(\eta_1,\eta_2\) are comparable, hence the purpose of the variance calculation is just to obtain the individual dependences stated in~\eqref{eq:bettunder1}--\eqref{eq:bettunder3}. We now present the proof of the underlined terms in~\eqref{eq:bettunder1}--\eqref{eq:bettunder3} containing transposes, the proof for the underlined terms without transposes is completely analogous and so omitted.
    \begin{proof}[Proof of~\eqref{eq:bettunder1}]
      We compute the second moment performing cumulant expansion exactly as in~\eqref{eq:firstcum}. Here we write only some representative second order terms:
      \begin{equation}\label{eq:expfirst}
        \begin{split}
          \mathbf{E}|\braket{\underline{WG_1G_2^t}A}|^2&=\mathbf{E}\frac{1}{N^2}\braket{G_1G_2^t A^2(G_2^t)^*G_1^*}+\mathbf{E}\frac{1}{N^2}\braket{G_1^*G_1^t}\braket{G_1G_2^t A(G_1^*)^t G_2^*A^t} \\
          &\quad +\mathbf{E}\frac{1}{N^2}\braket{G_1G_2^t AG_1^*}\braket{G_1A(G_2^t)^*G_1^*}+\mathbf{E}\frac{1}{N^3}\braket{G_1A(G_2^t)^*G_1^*A^t(G_2^*)^t G_1^t G_1} \\
          &\quad+\dots\dots
        \end{split}
      \end{equation}
      Using that \(\lVert A^2\rVert\lesssim 1\) and Ward identity \(\eta_i G_i G_i^*=\Im G_i\), we conclude that the first term in~\eqref{eq:expfirst} is bounded by
      \[
      \frac{1}{N^2}\braket{G_1G_2^t A^2(G_2^t)^*G_1^*}\le \frac{1}{N^2\eta_1\eta_2}\braket{\Im G_1\Im G_2^t}\prec \frac{\rho_*}{N^2\eta_1\eta_2\eta_*},
      \]
      where in the last inequality we used \(|\braket{\Im G_i}|\prec \rho_i\). For the second term in~\eqref{eq:expfirst} we use Schwarz inequality to separate resolvents with and without transpose, then we use Ward identity, to obtain
      \[
      \frac{1}{N^2}\braket{G_1^*G_1^t}\braket{G_1G_2^t A(G_1^*)^t G_2^*A^t}\prec \frac{\rho_1}{N^2\eta_1}\braket{G_1A^t G_2G_2^*A^t G_1^*}\prec\frac{\rho_1^2\rho_2}{N^2\eta_1^2\eta_2},
      \]
      where we used that by~\cite[Lemma~\ref{eth-degree two lemma}]{2012.13215} it holds \(\braket{\Im G_1 A^t\Im G_2 A^t}\prec \rho_1\rho_2\). For the third term in~\eqref{eq:expfirst} we bound
      \[
      \frac{1}{N^2}|\braket{G_1G_2^t AG_1^*}|^2\le \frac{1}{N^2}\braket{G_1G_1^*}\braket{G_1A(G_2^t)^*G_2^t AG_1^*}\prec \frac{\rho_1^2\rho_2}{N^2\eta_1^2\eta_2},
      \]
      again by~\cite[Lemma~\ref{eth-degree two lemma}]{2012.13215}. We now bound the fourth term in~\eqref{eq:expfirst}:
      \begin{equation}\label{eq:termwithn3}
        \begin{split}
          &\frac{1}{N^3}\braket{G_1A(G_2^t)^*G_1^*A^t(G_2^*)^t G_1^t G_1^*} \\
          &\qquad\le \frac{1}{N^3\eta_1}\braket{G_1^t\Im G_1A(G_2^t)^*G_2^t A\Im G_1(G_1^t)^*}^{1/2}\braket{G_2^t A^t G_1G_1^*A^t(G_2^*)^t}^{1/2}\prec\frac{\rho_1\rho_2}{N^3\eta_1^3\eta_2}.
        \end{split}
      \end{equation}
      This concludes the bound for the second order terms in the cumulant expansion.
      
    \end{proof}
    
    \begin{proof}[Proof of~\eqref{eq:bettunder2}]
      
      The proof of~\eqref{eq:bettunder2} is analogous to the proof of~\eqref{eq:bettunder1}, we just pick up an additional \(1/\eta_1\) factor due to the one more \(G_1\) term. We only show the bound for a few second order terms:
      \begin{equation}\label{eq:expsecond}
        \begin{split}
          \mathbf{E}|\braket{\underline{WG_1G_1G_2^t}A}|^2&=\mathbf{E}\frac{1}{N^2}\braket{G_1G_1G_2^t A^2(G_2^t)^*G_1^*G_1^*} \\
          &\quad+\mathbf{E}\frac{1}{N^2}\braket{G_1^*G_1^t}\braket{G_1G_1G_2^t A(G_1^*)^t(G_1^*)^t G_2^*A^t} \\
          &\quad +\mathbf{E}\frac{1}{N^3}\braket{G_1A(G_2^t)^*G_1^*G_1^*A^t(G_2^*)^t G_1^t G_1^t G_1}+\dots
        \end{split}
      \end{equation}
      For the first term in~\eqref{eq:expsecond} we have that
      \[
      \frac{1}{N^2}\braket{G_1G_1G_2^t A^2(G_2^t)^*G_1^*G_1^*}\le \frac{1}{N^2\eta_1\eta_2}\braket{\Im G_1G_1\Im G_2^t G_1^*}\prec \frac{\rho_*}{N^2\eta_1^3\eta_2\eta_*}.
      \]
      For the second term in~\eqref{eq:expsecond} we bound
      \[
      \begin{split}
        &\frac{1}{N^2}\braket{G_1^*G_1^t}\braket{G_1G_1G_2^t A(G_1^*)^t(G_1^*)^t G_2^*A^t} \\
        &\qquad\quad \prec \frac{\rho_1}{N^2\eta_1^2\eta_2}\braket{G_1^*A\Im G_2 A G_1\Im G_1}^{1/2}\braket{G_1^*A^t\Im G_2 A^t G_1\Im G_1}^{1/2}\prec \frac{\rho_1^2\rho_2}{N^2\eta_1^4\eta_2}.
      \end{split}
      \]
      Similarly to~\eqref{eq:termwithn3}, for the third term in~\eqref{eq:expsecond} we get
      \[
      \frac{1}{N^3}\braket{G_1A(G_2^t)^*G_1^*G_1^*A^t(G_2^*)^t G_1^t G_1^t G_1}\prec\frac{\rho_1\rho_2}{N^3\eta_1^5\eta_2}.
      \]
      This concludes the bound for the second order terms in the cumulant expansion.
      
    \end{proof}
    
    \begin{proof}[Proof of~\eqref{eq:bettunder3}]
      
      We consider a few representative second order terms from the cumulant expansion
      \begin{equation}\label{eq:expthird}
        \begin{split}
          \mathbf{E}|\braket{\underline{WG_1G_2^t AG_1A}}|^2&=\mathbf{E}\frac{1}{N^2}\braket{G_1G_2^t AG_1A^2G_1^*A(G_2^t)^*G_1^*}+\mathbf{E}\frac{1}{N^2}|\braket{G_1G_2^t AG_1AG_1^*}|^2 \\
          &\quad+\mathbf{E}\frac{1}{N^2}\braket{G_1G_1^*}\braket{G_1G_2^t AG_1A(G_1^*)^t G_2^*A^t(G_1^*)^t A^t} \\
          &\quad+\mathbf{E}\frac{1}{N^3}\braket{G_1A(G_1^t)^*A(G_2^t)^*G_1^*A^t G_1^t A^t G_2G_1^t G_1^*}+\dots
        \end{split}
      \end{equation}
      The first term in~\eqref{eq:expthird} is bounded by
      \[
      \begin{split}
        \frac{1}{N^2}\braket{G_1G_2^t AG_1A^2G_1^*A(G_2^t)^*G_1^*}&\le \frac{1}{N^2\eta_1^2}\braket{\Im G_1G_2^t A\Im G_1 A (G_2^t)^*} \\
        &\le \frac{1}{N^2\eta_1^2}\sqrt{\braket{\Im G_1\Im G_1}\braket{G_2^t A\Im G_1 A (G_2^t)^*G_2^t A\Im G_1A(G_2^t)^*}}\\
        &\prec\frac{\rho_1^{3/2}\rho_2}{N^{3/2}\eta_1^{5/2}\eta_2}.
      \end{split}
      \]
      By a Schwarz inequality and~\cite[Lemma~\ref{eth-degree two lemma}]{2012.13215} we readily conclude that the second term in~\eqref{eq:expthird} is bounded by
      \[
      \frac{1}{N^2}|\braket{G_1G_2^t AG_1AG_1^*}|^2\prec \frac{\rho_1^2\rho_2}{N\eta_1^2\eta_2}.
      \]
      We now bound the third term in~\eqref{eq:expthird}:
      \[
      \begin{split}
        \frac{1}{N^2}\braket{G_1G_1^*}\braket{G_1G_2^t AG_1A(G_1^*)^t G_2^*A^t(G_1^*)^t A^t} &\prec\frac{\rho_1}{N^2\eta_1^2\eta_2}\braket{\Im G_1^t BG_1^*A\Im G_2^t AG_1A} \\
        &\prec \frac{\rho_1^2\rho_2}{N\eta_1^2\eta_2}.
      \end{split}
      \]
      where \(B=A\) or \(B=A^t\). In order to conclude the bound of the representative second order terms we estimate
      \[
      \begin{split}
        &\frac{1}{N^3}\braket{G_1A(G_1^t)^*A(G_2^t)^*G_1^*A^t G_1^t A^t G_2G_1^t G_1^*} \\
        &\le\frac{1}{N^3}\braket{G_1^*G_1^t G_1A(G_1^t)^*A(G_2^t)^*G_2^t AG_1^t AG_1^*(G_1^t)^*G_1}^{1/2}\braket{G_1^*A^t G_1^t A^t G_2G_2^*A^t(G_1^t)^*A^t G_1}^{1/2} \\
        &\prec \frac{\rho_1\rho_2}{N^2\eta_1^3\eta_2}.
      \end{split}
      \]
      where we used a Schwarz inequality and~\cite[Lemma~\ref{eth-degree two lemma}]{2012.13215}.
    \end{proof}
    
    We conclude this section with the proof of Lemma~\ref{lem:addbb}.
    
    \begin{proof}[Proof of Lemma~\ref{lem:addbb}]
      Similarly to~\eqref{eq:3gs}, writing the equation for \(G^*A^t G^t AG^*\) we conclude that
      \[
      \begin{split}
        &\left[1+\landauOprec*{\frac{1}{N\eta}}\right](G^*A^t G^t AG^*)_{{\bm x}{\bm y}} \\
        &=\overline{m}(A^t G^t AG^*)_{{\bm x}{\bm y}}-\overline{m}(\underline{WG^*A^t G^t AG^*})_{{\bm x}{\bm y}}+\sigma\overline{m}\braket{G^*A^t G^t}(G^t AG^*)_{{\bm x}{\bm y}} \\
        &+\overline{m}(G^*)_{{\bm x}{\bm y}}\braket{G^*A^t G^t AG^*}+ \frac{\overline{m}\sigma}{N}\big((G^*)^t G^*A^t G^t AG^*\big)_{{\bm x}{\bm y}}+\frac{\overline{m}\widetilde{w_2}}{N}\big(\mathrm{diag}(G^*)G^*A^t G^t AG^*\big)_{{\bm x}{\bm y}} \\
        &+ \frac{\overline{m}}{N}\big((G^*A^t G^t)^t G^t AG^*\big)_{{\bm x}{\bm y}}+\frac{\overline{m}\widetilde{w_2}}{N}\big(\mathrm{diag}(G^*A^t G^t)G^t AG^*\big)_{{\bm x}{\bm y}} + \frac{\overline{m}\sigma}{N}\big((G^*A^t G^t AG^*)^t G^*\big)_{{\bm x}{\bm y}} \\
        &+\frac{\overline{m}\widetilde{w_2}}{N}\big(\mathrm{diag}(G^*A^t G^t AG^*)G^*\big)_{{\bm x}{\bm y}} =\landauOprec*{\sqrt{\frac{N\rho}{\eta}}},
      \end{split}
      \]
      where we used a Schwarz inequality and~\cite[Lemma~\ref{eth-degree two lemma}]{2012.13215} (similarly to~\eqref{eq:preb}) to bound all the terms with a pre-factor \(N^{-1}\), and that
      \[
      |(\underline{WG^*A^t G^t AG^*})_{{\bm x}{\bm y}}|\prec \sqrt{\frac{N\rho}{\eta}}, \quad  |\braket{G^t AG^*}|\prec \frac{1}{\sqrt{\eta}}, \quad |\braket{G^*A^t G^t AG^*}|\prec \frac{\rho}{\eta},
      \]
      with the first bound by~\eqref{eq:adbb}, and the last two bounds from~\cite[Lemma~\ref{eth-degree two lemma}]{2012.13215} with \(\Lambda_+\lesssim 1\), and that
      \[
      |(GAG^*)_{{\bm x}{\bm y}}|\prec \sqrt{\frac{\rho}{\eta}}, \quad
      \]
      by~\eqref{eq:isll2G}.
      
    \end{proof}
    
    \section{Calculations for the functional CLT for smooth test functions}\label{sec:proofclt}
    Starting from~\eqref{eq:linstatsmrem} and using the explicit formulas in Theorem~\ref{CLT theorem}
    in this section we complete the proof of Theorem~\ref{theo:CLT} by computing the expectations
    and variances of the limiting Gaussian processes explicitly. In the following we will often use that by Stokes Theorem it follows
    \begin{equation}\label{eq:sto}
      \int_\mathbf{R}\int_{\eta_1}^{\eta_a} \partial_{\overline{z}} \psi(x+\ii \eta) h(x+\ii \eta)\dif x \dif \eta=\frac{1}{2\ii}\int_\mathbf{R}\psi(x+\ii \eta_1) h(x+\ii \eta_1)\dif x
    \end{equation}
    for any \(\eta_1\in [0,\eta_a]\), and for any \(\psi,h\in H^1\) such that \(\partial_{\overline{z}} h=0\) on the domain of integration and for \(\psi\) vanishing at the left, right and top boundary of the domain of integration.
    
    We first compute the order \(N^{-1}\) correction to the expectation in Section~\ref{sec:compexp}, then in Section~\ref{sec:compvar} we prove an approximate Wick Theorem for \(L_N(f,I)\), \(L_N(f,\mathring{A}_\mathrm{d})\), \(L_N(f,A_{\mathrm{od}})\), and, finally, we explicitly compute their variance.
    
    \subsection{Computation of the expectation}\label{sec:compexp}
    
    By~\eqref{claim exp}, using the notation \(m=m_{sc}\), for any \(z=x+\ii\eta\in\mathbf{C}\setminus\mathbf{R}\) with \(\eta\gg N^{-1}\), and for any $\epsilon>0$, it follows that
    \begin{equation}\label{eq:resexpcomp}
      \begin{split}
        \braket{\E G(z)}&=m+\frac{\sigma m^2 m'}{Nm(1-\sigma m^2)}+\frac{\kappa_4}{N}m'm^3+\frac{\widetilde{w_2} m'm}{N}+\landauO*{ \frac{N^\epsilon}{(N\eta)^{3/2}}}, \\
        \braket{\E G(z)\mathring{A}}&=\landauO*{ \frac{N^\epsilon\rho_\mathrm{sc}(z)^{1/2}}{N^{3/2}\eta}}.
      \end{split}
    \end{equation}
    In the following computations we will often omit the $N^\epsilon$ factor in the error terms. By the second line of~\eqref{eq:resexpcomp}, together with~\eqref{eq:deranest}, we readily conclude that
    \begin{equation}\label{eq:expzt}
      N^{1+a/2}\Re\int_\mathbf{R}\int_{\eta_0}^{\eta_a} \partial_{\overline{z}} f_\mathbf{C}(z) \E \braket{G(x+\ii \eta)\mathring{A}} \, \dif \eta\dif x\lesssim  N^{-(1-a)/2}\log N.
    \end{equation}
    This concludes the bound for the expectation of \(L_N(f,\mathring{A})\).
    
    
    
    Next we proceed with the computation of \(\E L_N(f,I)\). Using the first equality in~\eqref{eq:resexpcomp}, and the bound~\eqref{eq:deranest} to estimate the error term, we get that
    \begin{equation}\label{eq:compexp}
      \begin{split}
        \E \sum_{i=1}^N f(\lambda_i)&=\frac{2N}{\pi}\Re\int_\mathbf{R}\int_{\eta_0}^{\eta_a} \partial_{\overline{z}} f_\mathbf{C}(z) \left[m+\frac{\sigma m m'}{N(1-\sigma m^2)}+\frac{\kappa_4}{4N} \partial_z(m^4)+\frac{\widetilde{w_2} m'm}{N}\right] \, \dif \eta\dif x \\
        &\quad+\landauO*{N^{-(1-a)/2}}.
      \end{split}
    \end{equation}
    Then, by~\eqref{eq:sto}, it is easy to see that
    \[
    \frac{2N}{\pi}\Re\int_\mathbf{R}\int_{\eta_0}^{\eta_a} \partial_{\overline{z}} f_\mathbf{C}(z) m \, \dif y\dif x= N\int_{-2}^2 f(x)\rho_{sc}(x)\, \dif x+\landauO{N^{1+a}\eta_0^2}.
    \]
    
    
    In the following computations we will often use that the regime \(\eta\in [\eta_r,\eta_0]\), with \(\eta_r:=N^{-100}\), is added back to the integration in~\eqref{eq:compexp} at the price of a negligible error much smaller than \(N^{-(1-a)/2}\).
    
    In the following by \(\log z\) we denote the complex logarithm on \(\mathbf{C}\setminus\mathbf{R}_-\). For the second term in~\eqref{eq:compexp}, by~\eqref{eq:sto}, using that for analytic functions \(h\) it holds \(\partial_z h=-\ii\partial_\eta h\), and that
    \[
    \partial_z\log(1-\sigma m^2)=-\frac{2\sigma mm'}{1-\sigma m^2},
    \]
    we compute
    \[
    \begin{split}
      &\frac{1}{\pi}\Re\int_\mathbf{R}\int_{\eta\ge \eta_r}\partial_{\overline{z}} f_\mathbf{C}(z) \ii\partial_\eta\log(1-\sigma m^2) \, \dif x\dif \eta \\
      &\qquad=\frac{1}{\pi}\Im\int_\mathbf{R}\int_{\eta\ge \eta_r}\partial_{\overline{z}} \partial_\eta f_\mathbf{C}(z) \log(1-\sigma m^2) \, \dif x\dif \eta +\landauO{\eta_r\log \eta_r} \\
      &\qquad=\frac{1}{2\pi}\Im \lim_{\epsilon\to 0^+}\left(\int_{-2+\epsilon}^{-\epsilon}+\int_\epsilon^{2-\epsilon}\right) f'(x) \log\big[1-\sigma m^2(x+\ii\eta_r)\big]\, \dif x+\landauO{\sqrt{\eta_r}}\\
      &\qquad=\frac{1}{2\pi}\lim_{\epsilon\to 0^+}\Im \Big[f(2-\epsilon)\log\big(4-\sigma(2-\epsilon)^2+\ii \sigma(2-\epsilon)\sqrt{4-(2-\epsilon)^2}\big) \\
      &\qquad\quad\qquad\qquad\qquad\qquad-f(-2+\epsilon)\log\big(4-\sigma(2-\epsilon)^2-\ii\sigma (2-\epsilon)\sqrt{4-(2-\epsilon)^2}\big)\Big] \\
      &\qquad\quad -\frac{1}{2\pi}\lim_{\epsilon\to 0^+} \Im\Big[f(\epsilon)\log(2(1+\sigma)-\sigma\epsilon^2+\ii\sigma\epsilon\sqrt{4-\epsilon^2}) \\
      &\qquad\quad\qquad\qquad\qquad\qquad-f(-\epsilon)\log(2(1+\sigma)-\sigma\epsilon^2-\ii\sigma\epsilon\sqrt{4-\epsilon^2})\Big] \\
      &\qquad\quad+\frac{1}{\pi}\Im \int_\mathbf{R}f(x) \frac{\sigma m(x+\ii \eta_r)m'(x+\ii \eta_r)}{1-\sigma m(x+\ii \eta_r)^2}\, \dif x+\landauO{\sqrt{\eta_r}}\\
      &\qquad=\frac{1}{2\pi}\lim_{\epsilon\to 0^+} \Bigg[f(2-\epsilon)\arctan\left(\frac{\sigma(2-\epsilon)\sqrt{4-(2-\epsilon)^2}}{4-\sigma(2-\epsilon)^2}\right) \\
      &\qquad\quad\qquad\qquad\qquad\qquad-f(2-\epsilon)\arctan\left(-\frac{\sigma(2-\epsilon)\sqrt{4-(2-\epsilon)^2}}{4-\sigma(2-\epsilon)^2}\right)\Bigg] \\
      &\qquad\quad-\frac{1}{2\pi}\lim_{\epsilon\to 0^+} \Bigg[f(\epsilon)\arctan\left(\frac{\sigma\epsilon\sqrt{4-\epsilon^2}}{2(1+\sigma)-\sigma\epsilon^2}\right)-f(-\epsilon)\arctan\left(-\frac{\sigma\epsilon\sqrt{4-\epsilon^2}}{2(1+\sigma)-\sigma\epsilon^2}\right)\Bigg] \\
      &\qquad\quad-\frac{1}{2\pi}\int_{-2}^2 f(x) \frac{\sigma[2(1+\sigma)- x^2]}{[(1+\sigma)^2-\sigma x^2]\sqrt{4-x^2}}\, \mathrm{d}x\\
      &\qquad=\bm1(\sigma=1)\frac{f(2)+f(-2)}{4}+\bm1(\sigma=-1)\frac{f(0)}{2}-\frac{1}{2\pi}\int_{-2}^2  \frac{f(x)\sigma[2(1+\sigma)- x^2]}{[(1+\sigma)^2-\sigma x^2]\sqrt{4-x^2}}\, \mathrm{d}x.
    \end{split}
    \]
    To go from the second to the third line we used that \(\Im\log[1-\sigma m(x+\ii\eta_r)]=\landauO{\sqrt{\eta_r}}\) for \(x\in [-2,2]^c\).
    
    For the third term in~\eqref{eq:compexp}, using~\eqref{eq:sto}, we have
    \begin{equation}\label{eq:secterm}
      \frac{\kappa_4}{2\pi}\Re\int_\mathbf{R}\int_{\eta_0}^{\eta_a} \partial_{\overline{z}} f_\mathbf{C}(z) \partial_z(m^4) \, \dif \eta\dif x=\frac{\kappa_4}{2\pi}\Im\int_\mathbf{R} f(x)  \frac{x^4-4x^2+2}{\sqrt{4-x^2}}\, \dif x+\landauO[\big]{\sqrt{\eta_r}}.
    \end{equation}
    
    Finally, for the last term in~\eqref{eq:compexp}, using~\eqref{eq:sto} again, we conclude
    \[
    \frac{\widetilde{w_2}}{2\pi}\Im\int_\mathbf{R}f(x)(m^2)'\, \dif x=-\frac{\widetilde{w_2}}{2\pi}\int_{-2}^2f(x) \frac{2-x^2}{\sqrt{4-x^2}}\, \dif x.
    \]
    
    This concludes the computation of \(\mathbf{E}L_N(f,I)\), and so, together with the bound~\eqref{eq:expzt}, it concludes the computation of the expectation of \(L_N(f,A)=\braket{A}L_N(f,I)+L_N(f,\mathring{A})\).
    
    \subsection{Wick Theorem and computation of the variance}\label{sec:compvar}
    
    Now we proceed with the computation of moments of products of \(L_N(f,I)\), \(L_N(f,\mathring{A}_{\mathrm{d}})\) and \(L_N(f,A_{\mathrm{od}})\).
    
    By~\eqref{eq CLT statement}--\eqref{eq:variancesres} we have that
    \begin{equation}\label{eq:varres}
      \begin{split}
        &\E\left(\prod_{i\in [p]}\braket{G(z_i)A_{i,\mathrm{od}}} \right)\left(\prod_{i\in (p,q]}\braket{G(z_i)\mathring{A}_{i,\mathrm{d}}} \right)\left(\prod_{i\in (q,r]}\braket{[G(z_i)-\E G(z_i)]}\right) \\
        &\qquad=\frac{1}{N^r}\sum_{\substack{P\in\mathrm{Pair}([p]), \\ Q\in\mathrm{Pair}((p,q]), \\R\in \mathrm{Pair}((q,r])}} \prod_{\set{i,j}\in P} \left[\braket{A_{i,\mathrm{od}} A_{j,\mathrm{od}}}\frac{m_i^2 m_j^2}{1-m_i m_j} +\braket{A_{i,\mathrm{od}} A_{j,\mathrm{od}}^t}\frac{\sigma m_i^2 m_j^2}{1-\sigma m_i m_j}\right]\\
        &\qquad\quad\times \prod_{\set{i,j}\in Q} \braket{\mathring{A}_{i,\mathrm{d}}\mathring{A}_{j,\mathrm{d}}}\biggl( \frac{m_i^2 m_j^2}{1-m_i m_j}+\frac{\sigma m_i^2 m_j^2}{1-\sigma m_i m_j}+\widetilde{w_2}m_i^2m_j^2+\kappa_4 m_i^3 m_j^3 \biggr)\\
        &\qquad\quad\times \prod_{\set{i,j}\in R}\biggl(\frac{m_i' m_j'}{(1-m_i m_j)^2}+\frac{\sigma m_i' m_j'}{(1-\sigma m_i m_j)^2}+
        \widetilde{w_2}m_i'm_j'+\frac{\kappa_4}{2} (m_i^2)' (m_j^2)' \biggr) \\
        &\quad + \landauO*{\frac{N^\epsilon\Psi}{\sqrt{N\eta_*}}},
      \end{split}
    \end{equation}
    where \(\eta_*:=\min\{|\eta_i|:\,i\in [k]\}\), and
    \begin{equation}
      \Psi:=\prod_{i\in [q]} \frac{\rho(z_i)^{1/2}}{N\eta_i^{1/2}} \prod_{i\in (q,r]} \frac{1}{N\eta_i}.
    \end{equation}

    Here we considered the case when\(p\), \(q\) and \(r\) are even, the case when one of them is odd the leading term is zero. Using~\eqref{eq:varres}, together with \(|\partial_{\overline{z}}f_\mathbf{C}|\lesssim N^{2a}|f''|+N^a(|f|+|f'|)|\chi'|\) for \(\eta\in I_{\eta_0,\eta_a}:=[\eta_0,\eta_a]\), by~\eqref{eq:deranest}, it readily follows that
    \begin{equation}
    \label{eq:Wick}
      \begin{split}
        &\E \left(\prod_{i\in [p]} N^{a/2} L_N(f^{(i)},A_{i,\mathrm{od}})\right)\left(\prod_{i\in (p,q]} N^{a/2} L_N(f^{(i)},\mathring{A}_{i,\mathrm{d}})\right)\left( \prod_{i\in (q,r]}L_N(f^{(i)},I)\right) \\
        &\quad = \sum_{\substack{P\in\mathrm{Pair}([p])\\Q\in\mathrm{Pair}((p,q]) \\R\in\mathrm{Pair}((q,r])}} \prod_{\set{i,j}\in P}\Bigg[\frac{N^a}{\pi^2}\iint_\mathbf{R}  \iint_{I_{\eta_0,\eta_a}}  \partial_{\overline{z_i}} f_\mathbf{C}^{(i)}(z_i)\partial_{\overline{z_j}} f_\mathbf{C}^{(j)}(z_j) \\
        &\qquad\qquad\qquad\qquad\quad\times\left(\braket{A_{i,\mathrm{od}} A_{j,\mathrm{od}}}\frac{m_i^2 m_j^2}{1-m_i m_j}+\braket{A_{i,\mathrm{od}} A_{j,\mathrm{od}}^t}\frac{\sigma m_i^2 m_j^2}{1-\sigma m_i m_j}\right)\Bigg] \\
        &\qquad\times\prod_{\set{i,j}\in Q}\Bigg[\frac{N^a}{\pi^2}\braket{\mathring{A}_{i,\mathrm{d}}\mathring{A}_{j,\mathrm{d}}}\iint_\mathbf{R}  \iint_{I_{\eta_0,\eta_a}}  \partial_{\overline{z_i}} f_\mathbf{C}^{(i)}(z_i)\partial_{\overline{z_j}} f_\mathbf{C}^{(j)}(z_j) \\
        &\qquad\qquad\qquad\quad\times \biggl( \frac{m_i^2 m_j^2}{1-m_i m_j}+\frac{\sigma m_i^2 m_j^2}{1-\sigma m_i m_j} +\kappa_4 m_i^3 m_j^3+\widetilde{w_2} m_i^2m_j^2)\biggr)\Bigg]\\
        &\qquad\times  \prod_{\set{i,j}\in R}\Bigg[\frac{1}{\pi^2}\iint_\mathbf{R}  \iint_{I_{\eta_0,\eta_a}} \partial_{\overline{z_i}} f_\mathbf{C}^{(i)}(z_i)\partial_{\overline{z_j}} f_\mathbf{C}^{(j)}(z_j) \\
        &\qquad\qquad\quad\times\biggl(\frac{m_i' m_j'}{(1-m_i m_j)^2} +\frac{\sigma m_i' m_j'}{(1-\sigma m_i m_j)^2}+\frac{\kappa_4}{2} (m_i^2)' (m_j^2)'+\widetilde{w_2} m_i'm_j') \biggr)\Bigg] \\
        &\qquad+\landauO*{\frac{N^\epsilon}{\sqrt{N\eta_a}} },
      \end{split}
    \end{equation}
    
    for any $\epsilon>0$. The equality in~\eqref{eq:Wick} concludes the proof of the Wick pairing. In the following sections we explicitly compute the integrals in the rhs.\ of~\eqref{eq:Wick}. In Section~\ref{sec:I} we compute the variance of \(L_N(f,I)\), then in Section~\ref{sec:A} we compute the variance of \(L_N(f,A_{\mathrm{od}})\) and \(L_N(f,\mathring{A}_{\mathrm{d}})\).
    
    
    \subsubsection{Computation of the variance of \(L_N(f,I)\)}\label{sec:I}
    
    Adding the regime \(\eta\in [\eta_r,\eta_0]\), with \(\eta_r:=N^{-100}\), at the price of a negligible error much smaller \(N^{-(1-a)/2}\), we start with
    \begin{equation}
    \label{eq:firstexp}
      \begin{split}
        \frac{1}{\pi^2}\int_\mathbf{R}\int_\mathbf{R}  \int_{|\eta|\ge \eta_r}\int_{|\eta'|\ge \eta_r}  \partial_{\overline{z}} f_\mathbf{C}(z_1)\partial_{\overline{z}} g_\mathbf{C}(z_2) \Bigg[ \frac{m_1' m_2'}{(1-m_1 m_2)^2}&+\frac{\sigma m_1' m_2'}{(1- \sigma m_1 m_2)^2}\\
        &+\frac{\kappa_4}{2} (m_1^2)' (m_2^2)'+\widetilde{w_2}m_1'm_2' \Bigg].
      \end{split}
    \end{equation}
    
    By direct computations, using~\eqref{eq:sto}, the last term in~\eqref{eq:firstexp} is given by
    \begin{equation}\label{eq:newoff}
      \begin{split}
        &\frac{4\widetilde w_2}{\pi^2} \left(\Re \int_\mathbf{R}\int_{\eta\ge \eta_r}  \partial_{\overline{z}}f_\mathbf{C}(z) m' \, \dif \eta\dif x\right)\left(\Re \int_\mathbf{R}\int_{\eta\ge \eta_r}  \partial_{\overline{z}}g_\mathbf{C}(z) m' \, \dif \eta\dif x\right) \\
        &\qquad =\frac{\widetilde w_2}{4\pi^2}\left(\int_{-2}^2 f(x) \frac{x}{\sqrt{4-x^2}} \, \dif x\right)\left(\int_{-2}^2 g(x) \frac{x}{\sqrt{4-x^2}} \,\dif x\right).
      \end{split}
    \end{equation}
    
    Then, using that
    \[
    \begin{split}
      \int_\mathbf{R}\int_{|\eta|\ge \eta_r}\partial_{\overline{z}} f_\mathbf{C}\partial_z(m^2)\, \dif x\dif \eta &=2\Re \int_\mathbf{R}\int_{\eta\ge \eta_r}\partial_{\overline{z}} f_\mathbf{C}\partial_z(m^2)\, \dif x\dif \eta\\
      &=-\int_{-2}^2 f(x) \Im \partial_z(m^2)\, \dif x+\landauO{\sqrt{\eta_r}},
    \end{split}
    \]
    by~\eqref{eq:sto}, we conclude that the \(\kappa_4\) coefficient in~\eqref{eq:firstexp} is given by
    \begin{equation}\label{eq:var1}
      \frac{\kappa_4}{2\pi^2}\left(\int_{-2}^2 f(x) \frac{2-x^2}{\sqrt{4-x^2}}\, \dif x\right)\left(\int_{-2}^2 g(x) \frac{2-x^2}{\sqrt{4-x^2}}\, \dif x\right),
    \end{equation}
    neglecting a negligible error of order \(\sqrt{\eta_r}=N^{-50}\).
    
    Next, using that
    \begin{equation}\label{eq:logder}
      \partial_{z_1}\partial_{z_2}\log(1-\sigma m_1m_2)=-\frac{\sigma m_1'm_2'}{(1-\sigma m_1m_2)^2},
    \end{equation}
    with \(z=x+\ii\eta\), \(w=y+\ii\eta'\), and \(m_1=m_{\mathrm{sc}}(z)\), \(m_2=m_{\mathrm{sc}}(w)\), for \(\sigma=1\) we compute
    \begin{equation}\label{eq:stokes}
      \begin{split}
        &\frac{1}{\pi^2}\int_\mathbf{R}\int_\mathbf{R}  \int_{|\eta|\ge \eta_r}\int_{|\eta'|\ge \eta_r}  \partial_{\overline{z}} f_\mathbf{C}(z)\partial_{\overline{w}} g_\mathbf{C}(w)  \frac{m_1' m_2'}{(1-m_1 m_2)^2}  \\
        &\qquad=\frac{1}{\pi^2}\int_\mathbf{R}\int_\mathbf{R}  \int_{|\eta|\ge \eta_r}\int_{|\eta'|\ge \eta_r}  \partial_{\overline{z}} f_\mathbf{C}(z)\partial_{\overline{w}} g_\mathbf{C}(w) \partial_\eta\partial_{\eta'}\log (1-m_1 m_2) \\
        &\qquad=\frac{1}{\pi^2}\int_\mathbf{R}\int_\mathbf{R}  \int_{|\eta|\ge \eta_r}\int_{|\eta'|\ge \eta_r}  \partial_{\overline{z}} \partial_\eta f_\mathbf{C}(z)\partial_{\overline{w}} \partial_{\eta'} g_\mathbf{C}(w) \log (1-m_1 m_2)+\landauO{\sqrt{\eta_r}} \\
        &\qquad=\frac{1}{2\pi^2}\int_\mathbf{R}\int_\mathbf{R} \dif x\dif y\,   \partial_x f(x)\partial_y g(y)\Re\Big[ \log(1-m_1 m_2) -\log(1-m_1 \overline{m_2}) \Big]+\landauO{\sqrt{\eta_r}},
      \end{split}
    \end{equation}
    where to go from the third to the fourth line we used~\eqref{eq:sto} and that \((\partial_\eta f_\mathbf{C})(x\pm\ii\eta_r)= \ii f'(x)\), since \(\eta_r\ll N^{-a}\).
    
    Note that
    \begin{equation}
    \label{eq:smb}
    \begin{split}
      \Re\Big[\log(1-m_1m_2)-\log(1-m_1\overline{m_2}) \Big]&=\landauO{\sqrt{\eta_r}}, \\
        \partial_y\Re\Big[\log(1-m_1m_2)-\log(1-m_1\overline{m_2}) \Big]&=\landauO{\sqrt{\eta_r}},
        \end{split}
    \end{equation}
    if either \(x\in [-2,2]^c\) or \(y\in [-2,2]^c\). Next, using~\eqref{eq:smb}, we compute
    \begin{equation}\label{eq:vartr1}
      \begin{split}
        &\int_\mathbf{R} \dif x \int_\mathbf{R}\dif y\, \partial_x f(x) \partial_y g(y) \Re\Big[\log(1-m_1m_2)-\log(1-m_1\overline{m_2}) \Big] \\
        &\qquad= -\int_\mathbf{R}\partial_x f\int_\mathbf{R} \dif y\, \big[g(y)-g(x)\big] \partial_y \Re\Big[\log(1-m_1m_2)-\log(1-m_1\overline{m_2}) \Big]+\landauO{\sqrt{\eta_r}}.
      \end{split}
    \end{equation}
    Then, using the same computations as in~\eqref{eq:vartr1} but integrating with respect to the \(x\)-variable, we conclude
    \begin{equation}\label{eq:vartr2}
      \begin{split}
        &\int_\mathbf{R} \dif x \int_\mathbf{R}\dif y\, \partial_x f(x) \partial_y g(y) \Re\Big[\log(1-m_1m_2)-\log(1-m_1\overline{m_2}) \Big]\\
        &\qquad=-\frac{1}{2}\int_\mathbf{R}\dif x\, \big(f(y)-f(x)\big)\big(g(y)-g(x)\big) \partial_x\partial_y \Re\Big[\log(1-m_1m_2)-\log(1-m_1\overline{m_2}) \Big].
      \end{split}
    \end{equation}
    Combining~\eqref{eq:stokes}, and~\eqref{eq:vartr1}--\eqref{eq:vartr2}, we finally conclude that the first term the last line of~\eqref{eq:Wick} is given by
    \begin{equation}\label{eq:almth}
      \frac{1}{4\pi^2}\int\int_{\mathbf{R}^2}\big(f(y)-f(x)\big)\big(g(y)-g(x)\big)\Re\left[ \frac{m_1'm_2'}{(1-m_1m_2)^2}-\frac{m_1'\overline{m_2}'}{(1-m_1\overline{m_2})^2}\right] \dif x\dif y+\landauO{\sqrt{\eta_r}}.
    \end{equation}
    Next we write
    \[
    \begin{split}
      \frac{m_1'm_2'}{(1-m_1m_2)^2}&=\frac{m_1^2m_2^2}{(1-m_1m_2)^2(1-m_1^2)(1-m_2^2)} \\
      &=-\frac{1}{\sqrt{(4-x^2)(4-y^2)}}\frac{m_1m_2}{(1-m_1m_2)^2}\Big[1+\landauO{\sqrt{\eta_r}}\Big] \\
      &=-\frac{2}{\sqrt{(4-x^2)(4-y^2)}}\frac{\Re[m_1m_2]-1}{|x+m_1+m_2|^4}\Big[1+\landauO{\sqrt{\eta_r}}\Big] \\
      &=\frac{2}{\sqrt{(4-x^2)(4-y^2)}}\frac{1}{|x+m_1+m_2|^2}\Big[1+\landauO{\sqrt{\eta_r}}\Big],
    \end{split}
    \]
    and so, using similar computations for the second term in~\eqref{eq:almth}, we get
    \begin{equation}\label{eq:realalmf}
      \begin{split}
        \Re\left[ \frac{m_1'm_2'}{(1-m_1m_2)^2}-\frac{m_1'\overline{m_2}'}{(1-m_1\overline{m_2})^2}\right]&=\frac{2}{\sqrt{(4-x^2)(4-y^2)}}\frac{|x+m_1+m_2|^2+|x+m_1+\overline{m_2}|^2}{|x+m_1+m_2|^2|x+m_1+\overline{m_2}|^2} \\
        &=\frac{4-xy}{(x-y)^2\sqrt{(4-x^2)(4-y^2)}}.
      \end{split}
    \end{equation}
    Combining~\eqref{eq:almth} with~\eqref{eq:realalmf}, we conclude the computation of the variance of \(L_N(f,I)\).  Using exactly the same computations as in the case \(\sigma=1\), for \(\sigma=-1\) we get
    \[
    \begin{split}
      &\frac{1}{\pi^2}\int_\mathbf{R}\int_\mathbf{R}  \int_{|\eta|\ge \eta_r}\int_{|\eta'|\ge \eta_r}  \partial_{\overline{z}} f_\mathbf{C}(z)\partial_{\overline{w}} g_\mathbf{C}(w)  \frac{m_1' m_2'}{(1-m_1 m_2)^2} \\
      &=\frac{1}{4\pi^2}\int_{-2}^2\int_{-2}^2\frac{(f(x)-f(y))(g(-x)-g(-y))}{(x-y)^2} \frac{4-xy}{\sqrt{(4-x^2)(4-y)^2}}\, \mathrm{d}x\mathrm{d}y.
    \end{split}
    \]
    We are now left with the case \(|\sigma|<1\). In this case the rhs.\ of~\eqref{eq:logder} is smooth since \(|1-\sigma m_1m_2|\ge 1-|\sigma|\). This implies that, unlike in~\eqref{eq:vartr1}--\eqref{eq:vartr2}, we can perform integration by parts without subtracting \(f(x), f(y)\). Using~\eqref{eq:logder} once again, for \(|\sigma|<1\) we compute
    \begin{equation}\label{eq:stokes1}
      \begin{split}
        &\frac{1}{\pi^2}\int_\mathbf{R}\int_\mathbf{R}  \int_{|\eta|\ge \eta_r}\int_{|\eta'|\ge \eta_r}  \partial_{\overline{z}} f_\mathbf{C}(z)\partial_{\overline{w}} g_\mathbf{C}(w)  \frac{m_1' m_2'}{(1-m_1 m_2)^2} \\
        &\qquad=-\frac{1}{2\pi^2}\int_\mathbf{R}\int_\mathbf{R}f(x)g(y)\Re\left[ \frac{\sigma m_1'm_2'}{(1-\sigma m_1m_2)^2}-\frac{\sigma m_1'\overline{m_2}'}{(1-\sigma m_1\overline{m_2})^2}\right]\, \dif x\dif y+\landauO{\sqrt{\eta_r}}.
      \end{split}
    \end{equation}
    Then we compute
    \[
    \frac{\sigma m_1'm_2'}{(1-\sigma m_1m_2)^2}=\frac{1}{2\sqrt{(4-x^2)(4-y^2)}} \left[\frac{1+\sigma^2}{|x+m_1+\sigma m_2|^4}-\frac{1-\sigma^2}{|x+m_1+\sigma m_2|^4}\right],
    \]
    and so using similar computations for the second term in the rhs.\ of~\eqref{eq:stokes1}, we conclude that
    \[
    \begin{split}
      &\Re\left[\frac{\sigma m_1'm_2'}{(1-\sigma m_1m_2)^2}-\frac{\sigma m_1'\overline{m_2}'}{(1-\sigma m_1\overline{m_2})^2}\right] \\
      &\qquad\qquad\qquad=\frac{(1+\sigma^2)[2(1+\sigma^2)-\sigma xy]}{2[(1-\sigma^2)^2-(1+\sigma^2)\sigma xy +\sigma^2(x^2+y^2)]\sqrt{(4-x^2)(4-y^2)}} \\
      &\qquad\qquad\qquad\quad-\frac{(1-\sigma^2)^2[2(1+\sigma^2)^2+8\sigma^2-2\sigma(1+\sigma^2)xy+\sigma^2(x^2y^2-2x^2-2y^2)]}{2[(1-\sigma^2)^2-(1+\sigma^2)\sigma xy +\sigma^2(x^2+y^2)]^2\sqrt{(4-x^2)(4-y^2)}} \\
      &\qquad\qquad\qquad=\frac{1}{2}\partial_x\partial_y\log\left[\frac{(x-\sigma y)^2+(\sqrt{4-x^2}-\sigma\sqrt{4-y^2})^2}{(x-\sigma y)^2+(\sqrt{4-x^2}+\sigma\sqrt{4-y^2})^2}\right]
    \end{split}
    \]
    on \(x,y\in [-2,2]\) and zero otherwise. This concludes the computation of the variance of \(L_N(f,I)\).

    \subsubsection{Computation of the variance of \(L_N(f,A_{\mathrm{od}})\) and \(L_N(f,\mathring{A}_{\mathrm{d}})\)}\label{sec:A}
    
    In this section we proceed with the computation of the variance of \(L_N(f,\mathring{A}_{\mathrm{d}})\). The computations for the variance of \(L_N(f,A_{\mathrm{od}})\) are completely analogous and so omitted.
    
    Using~\eqref{eq:sto}, we start with
    \begin{equation}
      \begin{split}
        &\int_\mathbf{R}\int_\mathbf{R}  \int_{|\eta|\ge \eta_r} \int_{|\eta'|\ge \eta_r} \partial_{\overline{z}} f_\mathbf{C}(z)\partial_{\overline{w}} g_\mathbf{C}(w) \frac{\sigma m_1^2 m_2^2}{1-\sigma m_1 m_2} \\
        &\quad= -\frac{1}{2}\int_\mathbf{R}\int_\mathbf{R}  \int_{|\eta|\ge \eta_r}  \int_{|\eta'|\ge \eta_r} f_\mathbf{C}(x+\ii\eta) g_\mathbf{C}(y+\ii\eta') \Re\left[\frac{\sigma m_1^2 m_2^2}{1-\sigma m_1 m_2}-\frac{\sigma m_1^2 \overline{m_2}^2}{1-\sigma m_1 \overline{m_2}}\right].
      \end{split}
    \end{equation}
    Then we compute 
    \begin{equation}
      \begin{split}
        \frac{\sigma m_1^2 m_2^2}{1-\sigma m_1 m_2}-\frac{\sigma m_1^2 \overline{m_2}^2}{1-\sigma m_1 \overline{m_2}}&=-2\ii m_1\Im m_2+\frac{m_1 m_2}{1-\sigma m_1 m_2}-\frac{m_1 \overline{m_2}}{1-\sigma m_1 \overline{m_2}} \\
        &=-2\ii m_1\Im m_2-\frac{2\ii\Im m_2}{x+\ii \eta_r+2\sigma \Re m_2+m_1(1-\sigma^2|m_2|^2)}.
      \end{split}
    \end{equation}
    Taking the real part we have that
    \begin{equation}\label{eq:expcompdel}
      \begin{split}
        &\Re\left[\frac{\sigma m_1^2 m_2^2}{1-\sigma m_1 m_2}-\frac{\sigma m_1^2 \overline{m_2}^2}{1-\sigma m_1 \overline{m_2}}\right] \\
        &=2 \Im m_1\Im m_2 \bm1(|x|,|y|\le 2) \\
        &\qquad\qquad\qquad\quad- \frac{2 \Im m_2(\eta_r+\Im m_1 (1-\sigma^2\abs{m_2}^2))}{ (x+2\sigma\Re m_2 + \Re m_1(1-\sigma^2\abs{m_2}^2))^2 + (\eta_r+\Im m_1 (1-\sigma^2\abs{m_2}^2))^2} \\
        &= 2 \Im m_1\Im m_2 \bm1(|x|,|y|\le 2)-\bm1(\sigma=\pm 1)\sqrt{4-y^2}\frac{2\eta_r}{(x-\sigma y)^2+4\eta_r^2} \\
        &\quad-\frac{(1-\sigma^2)\sqrt{(4-x^2)(4-y^2)}}{2[\sigma^2(x^2+y^2)+(1-\sigma^2)^2-xy\sigma(1+\sigma^2)]}\bm1(|x|,|y|\le 2)+\landauO{\sqrt{\eta_r}} \\
        &=\big[2 \Im m_1\Im m_2-\bm1(\sigma=\pm 1)\pi\sqrt{4-x^2}\delta_{x-\sigma y}\big]\bm1(|x|,|y|\le 2) \\
        &\quad-\frac{(1-\sigma^2)\sqrt{(4-x^2)(4-y^2)}}{2[\sigma^2(x^2+y^2)+(1-\sigma^2)^2-xy\sigma(1+\sigma^2)]}\bm1(|x|,|y|\le 2)+\landauO{\sqrt{\eta_r}},
      \end{split}
    \end{equation} 
    in distributional sense.

    Similarly, using~\eqref{eq:sto}, we readily conclude
    \begin{equation}
      \int_\mathbf{R}\int_{|\eta|\ge \eta_r}\partial_{\overline{z}}f_\mathbf{C} m^3=-\int_{-2}^2f(x) \frac{\sqrt{4-x^2}(1-x^2)}{2} \dif x.
    \end{equation}
    
    Finally, using~\eqref{eq:sto} again, we compute the integral of the last term in the fourth line of~\eqref{eq:Wick} and conclude that
    \[ 
    \frac{1}{\pi^2}\int_\mathbf{R}\int_{\eta_0\le|\eta|\le \eta_a} \partial_{\overline{z}} f(z) m^2\,\dif \eta \dif x=\int_{-2}^2 f(x) x \rho_{\mathrm{sc}}(x)\,\dif x.
    \]
    This concludes the computation of the variance of \(L_N(f,\mathring{A}_d)\) and so of Theorems~\ref{theo:CLT}.
    
    \subsection{Independence of linear statistics on different scales: Proof of Theorem~\ref{theo:inddiffeta}}
\label{sec:indlinstat}




In the computations above we performed all the analysis for a single scale $N^{-a}$. However, inspecting the proof one can notice that exactly the same computations go through for test functions $f_1$ and $f_2$ living on two different scales $N^{-a_1}$, $N^{-a_2}$, for $N$-independent $a_1,a_2\in (0,1)$ such that $a_1<a_2$. The case when $a_1=0$ is analogous and so omitted. In particular, the test functions $f_i$ are given by
\begin{equation}
  \label{eq:defmulttestfun}
  f_i(x):=g_i(N^{a_i}(x-E_i)), \qquad i=1,2.
\end{equation}
The bound of the error term in~\eqref{eq:Wick} is completely analogous and it is given by $[N(\min_i\eta_{a_i})]^{-1/2}$. The bound of the leading terms follow by estimating the deterministic terms computed in Section~\ref{sec:compvar} to prove Theorem~\ref{theo:CLT}. To make this argument clearer we give the explicit bound of the representative term:
\begin{equation}
  \label{eq:lastboh}
  \frac{1}{4\pi^2}\int_{-2}^2\int_{-2}^2\frac{(f_1(x)-f_1(y))(f_2(x)-f_2(y))}{(x-y)^2} \frac{4-xy}{\sqrt{(4-x^2)(4-y^2)}}\, \mathrm{d}x\mathrm{d}y,
\end{equation}
which corresponds to $m_1'm_2'(1-m_1m_2)^{-2}$ in~\eqref{eq:varres} after plugging it into the Helffer-Sj\"ostrand formula as in~\eqref{eq:Wick} (see~\eqref{eq:almth}--\eqref{eq:realalmf} for the explicit computations that give~\eqref{eq:lastboh}). For simplicity, we assume $E_1=E_2=E$, the general case is analogous. Using the definition of $f_i$ in~\eqref{eq:defmulttestfun}, and the change of variables $z=N^{a_1}(x-E)$, $w=N^{a_2}(y-E)$, we have that
\[
|\eqref{eq:lastboh}|\lesssim \frac{1}{N^{a_1+a_2}}\int\int_{\mathrm{supp}(g_1)\cap \mathrm{supp}(g_2)} 
\Big| \frac{(g_1(z)-g_1(w))(g_2(z)-g_2(w))}{(zN^{-a_1}-wN^{-a_2})^2}\Big| \dif w\dif z\lesssim N^{-(a_2-a_1)},
\]
where we used the smoothness of $g$, that the domain of integrations have size of order one 
as a consequence of $g_i$ being compactly supported, and that
\[
\frac{4-xy}{\sqrt{(4-x^2)(4-y^2)}}=1+\landauO{N^{-a_1}}
\]
in the relevant regime.


\section{Proof of \eqref{eq:apprnetao}}
\label{sec:proofple}

The proof of \eqref{eq:apprnetao} uses an operator  version of Pleijel's formula that cannot be directly
read off from the original paper [Add a reference], so  we present a detailed proof here for completeness.

By the spectral theorem for the Wigner matrix $W$ we have that
\begin{equation}
\label{eq:specrep}
f(W)=\int_\mathbf{R} f(\lambda) \dif \rho(\lambda),
\end{equation}
where $\dif \rho(\lambda)$ is the  (projection valued) spectral measure of $W$
and $f\in BV(\mathbf{R})$ is a function of bounded variation.  Set
\begin{equation}
\label{eq:defI}
I(z):=\frac{1}{2\pi \ii}\int_{L(z)} G(w)\, \dif w,
\end{equation}
with $z=x+\ii \eta$ and $L(z)$ being the contour in Figure~\ref{figure L contour}. Recall that $G(w)=(W-w)^{-1}$ 
is the resolvent of $W$.
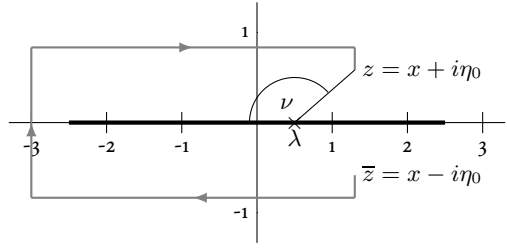
\begin{figure}[hbt] 
  \centering
  \begin{tikzpicture}
  \draw (0,-1.6) -- (0,1.6); 
  \draw (-3.3,0) -- (3.3,0);   
  \foreach \x in {-3,...,-1} {
    \draw (\x,-4pt) -- (\x,4pt) node[pos=0,below] {\x};
  }
  \foreach \x in {1,...,3} {
    \draw (\x,-4pt) -- (\x,4pt) node[pos=0,below] {\x};
  }
  \draw (-0pt,-1.2) -- (0pt,-1.2) node[pos=0,left] {-1};
  \draw (-0pt,1.2) -- (0pt,1.2) node[pos=0,left] {1};
  \node at (2.2,0.7) {$z=x+i\eta_0$}; 
  \node at (2.2,-0.7) {$\overline{z}=x-i\eta_0$}; 
  \node at (0.5,0) {$\times$}; 
  \node at (0.5,-0.2) {$\lambda$}; 
  \node at (0.4,0.25) {$\nu$}; 
  
  \draw[thick,gray,xshift=0pt]
    (1.3,-0.7) -- (1.3,-1);
  \draw[thick,gray,xshift=0pt,decoration={ markings, mark=at position 0.5 with {\arrow{latex}}}, postaction={decorate}]
    (1.3,-1) -- (-3,-1);
  \draw[thick,gray,xshift=0pt,decoration={ markings, mark=at position 0.5 with {\arrow{latex}}}, postaction={decorate}]
    (-3,-1) -- (-3,1);
  \draw[thick,gray,xshift=0pt,decoration={ markings, mark=at position 0.5 with {\arrow{latex}}}, postaction={decorate}]
    (-3,1) -- (1.3,1);
  \draw[thick,gray,xshift=0pt]
    (1.3,1) -- (1.3,0.7);
  \draw[ultra thick,black,xshift=0pt, postaction={decorate}]
    (-2.5,0) -- (2.5,0);
  \draw[xshift=0pt, postaction={decorate}]
    (0.5,0) -- (1.3,0.7);
  \draw[black] ([shift=(42:0.6cm)]0.5,0) arc (42:180:0.6cm);
  \end{tikzpicture}
  \caption{Contour $L(x)$ and angle $\nu=\nu(\lambda,z)$}\label{figure L contour}
  \end{figure} 
We assumed that $W\ge -5$ that holds with very high probability.
Then we get
\begin{equation}
\label{eq:repIrep}
I(z)=\frac{1}{2\pi \ii}\int_{L(z)} G(w)\, \dif w=\frac{1}{2\pi \ii}\int_\mathbf{R}\int_{L(z)} \frac{1}{\lambda-w}\, \dif w\dif \rho(\lambda)=\frac{1}{\pi}\int_\mathbf{R} \nu(\lambda,z)\,\dif \rho(\lambda),
\end{equation}
with $\nu(\lambda,z)\in (0, \pi)$ being the angle shown in Figure~\ref{figure L contour}. 
Since for small $\eta$ the angle $\nu(\lambda,z)$  is typically close to $\pi$  or $0$ depending on whether $\lambda<x=\Re z$
or $\lambda>x$, thus $I(z)$ well approximates the spectral projection $\rho(W\le x)$. In fact, using
\begin{equation}
\label{eq:realprel}
\eta\Re G(z)
=\int_\mathbf{R} \frac{\eta(\lambda-z)}{(\lambda-x)^2+\eta^2}\, \dif \rho(\lambda)=\int_\mathbf{R} \sin\nu(\lambda,z)\cos\nu(\lambda,z)\, \dif \rho(\lambda),
\end{equation}
and
\begin{equation}
\label{eq:imaprel}
\eta\Im G(z)
=\int_\mathbf{R} \frac{\eta^2}{(\lambda-x)^2+\eta^2}\, \dif \rho(\lambda)=\int_\mathbf{R} \sin^2\nu(\lambda,z)\, \dif \rho(\lambda),
\end{equation}
we obtain the identity
\begin{equation}
\label{eq:exprmeas}
\rho\big([-5,x)\big)=I(z)+\frac{\eta}{\pi}\Re G(z)+\int_\mathbf{R} g(\lambda,z) \,\dif \rho(\lambda), \quad z=x+\ii \eta,
\end{equation}
with
\[
g(\lambda,z):=\nu(\lambda,z)-\pi \chi_{(\pi/2,\pi]}(\nu(\lambda,z))-\sin \nu(\lambda,z)\cos\nu(\lambda,z),
\]
where we also  used that $\chi_{[-5,x)}(\lambda)=\chi_{(\pi/2,\pi]}(\nu(\lambda,z))$. A simple calculation shows that the function $g$ 
satisfies the bound 
\begin{equation}
\label{eq:boundglamz}
 |g(\lambda,z)|\lesssim \sin^2 \nu(\lambda,z).
\end{equation}


Combining \eqref{eq:specrep} and \eqref{eq:exprmeas}, and using integration by parts, we get
\begin{equation}
\label{eq:impequality}
f(W)= \int_\mathbf{R} \left[I(\lambda+\ii\eta) 
 +\frac{\eta}{\pi}\Re G(\lambda+\ii \eta)+\int g(s,\lambda+\ii\eta) \,\dif \rho(s) \right]\, \dif f(\lambda).
\end{equation}
Then choosing $\eta=\eta_0$, with $\eta_0$ given below \eqref{eq:defcontourgamma}, and using \eqref{eq:impequality} for $f(W)=P(W)$, where $P(W)$ is defined in \eqref{eq:specpro}, we conclude
\begin{equation}
\label{eq:finfineqhop}
\begin{split}
\braket{P(W)A}=&\frac{1}{2\pi \ii}\int_{\Gamma_{K,i_0}}\braket{G(w)\mathring{A}}\, \dif w \\
 &+ \sum_\pm \Big[ \frac{\eta_0}{\pi}\braket{\Re G(\gamma_{i_0}\pm \eta_{i_0}+\ii\eta_0)A}+\int_\mathbf{R} g(s,\gamma_{i_0}\pm \eta_{i_0}+\ii\eta_0) \braket{\dif \rho (s) A} \Big],
\end{split}
\end{equation}
with $\gamma_{i_0}$ defined in \eqref{eq:quantin}, $\eta_{i_0}$ in \eqref{eq:defetai0}, and $\Gamma_{K,i_0}$ in \eqref{eq:defcontourgamma}, 
and $A$ is a deterministic matrix with $\braket{A}=0$ and $\norm{A}\lesssim 1$. For the first term in the r.h.s. of \eqref{eq:impequality} we used the definition of $I(z)$ in \eqref{eq:defI}. 

Finally, using that
\[
|\braket{\Re G(x+\ii \eta_0)A}|\lesssim \frac{\sqrt{\rho(x+\ii\eta_0)}}{N\sqrt{\eta_0}}
\]
for any $|x|\le 5$ by the local law \eqref{eq:avll1G}, and that
\[
\begin{split}
\left|\int_\mathbf{R} g(s,x+\ii\eta_0) \braket{\dif \rho (s) A}\right|&=\left|\int_\mathbf{R} g(s,x+\ii\eta_0) \frac{1}{N}\sum_i \delta(s-\lambda_i) \braket{{\bm u}_i,A{\bm u}_i}\, \dif s\right| \\
&\prec \frac{1}{N^{3/2}}\sum_i \int_\mathbf{R} \sin^2 \nu(s,x+\ii\eta_0) \delta(s-\lambda_i)\, \dif s\\
&\lesssim\frac{\eta_0}{\sqrt{N}}\braket{\Im G(x+\ii\eta_0)}\prec  \frac{\eta_0\rho(x+\ii\eta_0)}{\sqrt{N}},
\end{split}
\]
for any $|x|\le 5$ by \eqref{eq:finfineqhop} we conclude the proof of \eqref{eq:apprnetao}. Here $\{\lambda_i\}_{i\in [N]}$ are the eigenvalues of $W$ and $\{{\bm u}_i\}_{i\in [N]}$ are the corresponding eigenvectors. Note that to go from the first to the second line we used that $|\braket{{\bm u}_i,A{\bm u}_i}|\prec N^{-1/2}$ by \cite[Theorem 2.2]{2012.13215} and \eqref{eq:boundglamz}, and to go from the second to the third line we used \eqref{eq:imaprel}.

\printbibliography%

\end{document}